\newif\ifArXiV

\ArXiVtrue	

\ifArXiV 
\documentclass[12pt,a4paper]{article}
\date{}
\newcommand{\myDate}{} 
\usepackage{amsthm}
\usepackage{authblk} 
\else 
\documentclass{jgaa-art}

\newcommand{\myDate}{} 
\fi 

\usepackage{graphicx}
\usepackage[english]{babel}
\usepackage[utf8x]{inputenc}
\usepackage{verbatim} 
\usepackage{epstopdf}
\usepackage{todonotes} 
\usepackage{eso-pic} 
\usepackage{verbatim} 
\usepackage[numbers]{natbib} 

\usepackage[T1]{fontenc} 
\usepackage{lmodern}
\usepackage{amsmath,amssymb}
\usepackage{wrapfig}
\usepackage{caption} 
\usepackage{enumitem} 
\usepackage{multirow}
\usepackage{tikz}
\usetikzlibrary{calc}
\usetikzlibrary{decorations.pathreplacing} 
\tikzstyle{vertex}=[circle,fill=black!25,minimum size=13pt,inner sep=0pt]
\tikzstyle{vertexSmall}=[circle,fill=black!25,minimum size=10pt,inner sep=0pt]  
\tikzstyle{vertexSmallGray}=[circle,fill=black!25,minimum size=4pt,inner sep=0pt, draw] 
\tikzstyle{vertexSmallWhite}=[circle,fill=white,minimum size=4pt,inner sep=0pt, draw]  
\tikzstyle{vertexSmallBlack}=[circle,fill=black,minimum size=4pt,inner sep=0pt, draw]
\tikzstyle{vertexBigBlack}=[circle,fill=black,minimum size=10pt,inner sep=0pt, draw]    
\tikzstyle{vertex2}=[circle,fill=black,minimum size=3pt,inner sep=0pt]
\tikzstyle{vertex3}=[fill=black,minimum size=4pt,inner sep=0pt]
\tikzstyle{inC} = [draw,line width=6pt,-,yellow]
\tikzstyle{inD} = [draw,line width=3pt,-,red!70]
\tikzstyle{bigcircle} = [draw,line width=3pt,-,green]

\usepackage{aliascnt} 

\usepackage{hyperref} 
\addto\extrasenglish{}
\addto\extrasenglish{}
\addto\extrasenglish{}

\usepackage[capitalize]{cleveref} 
\usepackage{algorithm}
\usepackage{algpseudocode}
\numberwithin{algorithm}{section}
\floatname{algorithm}{Algorithm}
\makeatletter
\newcommand{\algmargin}{\the\ALG@thistlm}
\makeatother
\newlength{\whilewidth}
\usepackage{soul,xcolor}
\setstcolor{blue}

\algdef{SE}[longWHILE]{longWhile}{EndlongWhile}[1]
{\parbox[t]{\dimexpr\linewidth-\algmargin}{%
        \hangindent\whilewidth\strut\algorithmicwhile\ #1\ \algorithmicdo\strut}}{\algorithmicend\ \algorithmicwhile}%

\algnewcommand{\longState}[1]{\State%
    \parbox[t]{\dimexpr\linewidth-\algmargin}{\hangindent\algorithmicindent\strut #1\strut}}

\algnewcommand{\Input}[1]{%
    \hspace*{\algorithmicindent} %
    \parbox[t]{\dimexpr\linewidth-\algmargin-\algorithmicindent}{%
        \hangindent\algorithmicindent\strut%
        \textbf{Input: }#1\strut}}

\algnewcommand{\Output}[1]{%
    \hspace*{\algorithmicindent} %
    \parbox[t]{\dimexpr\linewidth-\algmargin-\algorithmicindent}{%
        \hangindent\algorithmicindent\strut%
        \textbf{Output: }#1\strut}}

\newtheorem{theoremMy}{Theorem}[section]

\newaliascnt{observation}{theoremMy}
\newtheorem{observation}[observation]{Observation}
\aliascntresetthe{observation}

\newaliascnt{lemmaMy}{theoremMy}
\newtheorem{lemmaMy}[lemmaMy]{Lemma}
\aliascntresetthe{lemmaMy}

\newaliascnt{corollary}{theoremMy}
\newtheorem{corollary}[corollary]{Corollary}
\aliascntresetthe{corollary}

\ifArXiV 
\theoremstyle{definition}
\else 
\fi 
\newaliascnt{definition}{theoremMy}
\newtheorem{definition}[definition]{Definition}
\aliascntresetthe{definition}

\newcommand{\nz}{\mathbb{N}}
\newcommand{\seq}[1]{\langle #1\rangle}

\DeclareMathOperator*{\argmin}{arg\,min}
\algnewcommand{\IfThen}[2]{
  \State \algorithmicif\ #1\ \algorithmicthen\ #2}
\algnewcommand{\ForInline}[2]{
    \State \algorithmicfor\ #1\ \algorithmicdo\ #2 \algorithmicend\ \algorithmicfor}

\providecommand{\keywords}[1]{\textbf{Keywords:} #1}

\newcommand{\textAlgo}[1]{\text{\it{#1}}}

\begin{document}


\title{Monotonic Representations of Outerplanar Graphs as
Edge Intersection Graphs of Paths on a Grid}

\ifArXiV 
\author[1]{Eranda \c{C}ela}
\author[2]{Elisabeth Gaar\thanks{The second author acknowledges support by the 
Austrian Science Fund (FWF): I 3199-N31.\myDate}}
\affil[1]{TU Graz, Steyrergasse 30, Graz A-8010, Austria, 
\href{mailto:cela@math.tugraz.at}{cela@math.tugraz.at}}
\affil[2]{Institute of Production and Logistics Management, Johannes Kepler 
	University Linz, Altenberger Stra{\ss}e 69, Linz A-4040, Austria,
	\href{mailto:elisabeth.gaar@jku.at}{elisabeth.gaar@jku.at}}

\fi 

\maketitle


\begin{abstract}
 
In a representation of  a graph $G$ as an edge intersection graph of paths on a
grid (EPG) every vertex of $G$ is represented by a path on a grid and two  
paths share a grid edge iff the corresponding  vertices are
adjacent.
In a monotonic EPG
representation every path on the grid is  
ascending in both rows and columns. In a (monotonic)  $B_k$-EPG
representation every path on the grid has at most $k$ bends.
The (monotonic) bend number  $b(G)$ ($b^m(G)$) of a graph 
$G$ 
is   the smallest natural number $k$ for which there exists a (monotonic)
$B_k$-EPG representation of $G$.

In this paper we deal with  the monotonic bend number of outerplanar graphs and show
  that $b^m(G)\leqslant 2$ holds for  every outerplanar graph~$G$.
Moreover, we characterize  the maximal outerplanar graphs and the cacti
with  (monotonic)  bend number equal to $0$, $1$ and $2$
in terms of forbidden induced subgraphs.
As a byproduct we obtain low-degree polynomial time algorithms to 
 construct (monotonic) EPG
representations with the smallest possible number of bends  for
 maximal outerplanar graphs and cacti.

\end{abstract}

\addcontentsline{toc}{section}{Contents}

\section{Introduction and definitions}\label{intro:sec}

Edge intersection graphs of paths on a grid were introduced in 2009 by Golumbic, Lipshteyn and
Stern~\cite{startpaper}. A graph $G$ is called \emph{an edge intersection graph of paths on a grid} (\emph{EPG})
if there exists  a rectangular grid with horizontal and vertical grid lines and a set of paths along
the grid lines such that the vertices  of $G$ correspond to the paths and two vertices  are adjacent in $G$
 if and only if the corresponding paths share a grid edge.
 In this case we say that the paths \emph{intersect}.
 Such a representation of a graph is called an \emph{EPG representation}.

An EPG representation is called a \emph{$B_{k}$-EPG representation}, if every path has at most $k$ bends,
for some $k\in \nz$.  A path on a grid is called \emph{monotonic},
if it is ascending in both columns and rows.
An EPG representation is called \emph{monotonic} if every path is monotonic. A monotonic $B_{k}$-EPG representation is called a \emph{$B_{k}^{m}$-EPG representation}.
Furthermore a graph is called  \emph{$B_{k}$-EPG} or \emph{$B_{k}^{m}$-EPG} if there is a
$B_{k}$-EPG representation or a $B_{k}^{m}$-EPG representation, respectively.
We denote by $B_{k}$ and $B_{k}^{m}$ the class of all graphs which are $B_{k}$-EPG and $B_{k}^{m}$-EPG,
respectively.
The \emph{bend number} $b(G)$ and the \emph{monotonic bend number} $b^m(G)$ of a graph $G$ are
defined as the minimum $k \in \nz$
and $\ell \in \nz$ such that $G$ is in $B_{k}$ and $B_{\ell}^{m}$,
respectively.

EPGs  were initially  motivated by  applications in  circuit layout design and chip manufacturing.
For a detailed description of these applications and their  relationship to  
EPGs we refer to~\cite{knockknee, startpaper, AsurveyOnWiring}.
Similar concepts which have been investigated in the literature include
\emph{edge intersection graphs of paths on a tree} (EPT),
\emph{vertex intersection graphs of paths on a tree} (VPT) and \emph{vertex
   intersection graphs of paths on a grid} (VPG), see \cite{VPG,VPT1,firstEPT1,firstEPT2}.

Many  research papers have dealt with  (monotonic) EPG graphs since the introduction of the
concept, see \cite{H, B, C, E,J, startpaper, G, D, F}.
One of the  well studied questions is  the relationship between  the  classes $B_{k}$ and
$B_k^ m$, for $k \in  \nz$.
Golumbic, Lipshteyn and Stern~\cite{startpaper} showed that each graph is $B_{k}$-EPG and
$B_{\ell}^{m}$-EPG for some $k, \ell \in \nz$.
Obviously $B_{0} \subseteq B_{1} \subseteq B_{2} \subseteq \dots$ and
$B_{0}^{m} \subseteq B_{1}^{m} \subseteq B_{2}^{m} \subseteq \dots$ hold.
 Heldt, Knauer and Ueckerdt~\cite{F} have shown that the first chain of
 inclusions above is strict,  i.e.\ $B_{k} \subsetneqq B_{k+1}$ holds for every
 $k \geqslant 0$. 
Clearly $B_{k}^{m} \subseteq B_{k}$ holds for every $k$ and the relationships
$B_{0} = B_{0}^{m}$ and $B_{0} \subseteq B_{1}^{m}$ are trivial. 
 Golumbic, Lipshteyn and Stern conjectured in \cite{startpaper} that $B_{1}^{m}
 \subsetneqq B_{1}$ and this  was confirmed in~\cite{E}.

 The   recognition problem for the   classes $B_k$ and $B_k^m$,  $k \in
 \nz$,  has been investigated. 
It asks whether
an input graph $G$ is $B_k$-EPG ($B_k^m$-EPG). The recognition problem for
 $B_0$
 is solvable in linear time. This follows from the observation that 
 a graph $G$ is  $B_0$-EPG 
 if and only if  $G$ is an interval graph and from the fact that interval graphs can be
  recognized in linear time, see Booth and Lueker~\cite{BooLue76}.
Heldt, Knauer and Ueckerdt~\cite{F} have proven that the recognition problem for
$B_{1}$ is NP-complete.  Cameron, Chaplick and Ho\`{a}ng~\cite{E},
have shown that the recognition problem for the class  $B_{1}^{m}$  is
 NP-complete as well. Recently Pergel and 
 Rz{\k{a}}{\.{z}}ewski~\cite{LRecognizionB2EPGisNPComplete}
have settled the NP-completeness of the recognition problem  also for the 
classes   $B_{2}$ and  $B_{2}^{m}$.  

The computation of the (monotonic) bend number of a given graph is closely
related to the recognition problem for the classes $B_k$ ($B_k^m$), $k\in
\nz$.  The results on recognition problems mentioned
above imply  that the computation of the (monotonic) bend number $b(G)$
($b^m(G)$) of a graph $G$   is a hard problem in general.
Thus the identification of   upper bounds on $b(G)$, $b^m(G)$ 
is a reasonable task.  
 Biedl and Stern~\cite{C}  have shown that  $b(G)\le 5$ holds for planar graphs.
 This  upper bound was improved to  $4$ by
Heldt, Knauer and Ueckerdt~\cite{D}. Moreover Heldt et al.~\cite{D} have also
 shown  that $b(G)\le 2$ holds if $G$ is  outerplanar.
 This results holds also for  Halin graphs as shown by  Francis and 
 Lahiri~\cite{KHalinGraphs}.
 We strengthen the result of~\cite{D} and show that $b^m(G)\le 2$ holds if $G$ 
 is
 outerplanar. 

Another question considered in the literature is the following: given a 
certain  class of graphs and a natural number
$k$, characterize the members of the class which belong to $B_k$ ($B_k^m$). 
 Deniz, Nivelle, Ries and Schindl~\cite{SSplitGraphsB1} provided a characterization of $B_{1}$-EPG 
split graphs  for which there exists  a  $B_1$-EPG representation using only $L$-shaped paths on the  grid. 
In this paper we  provide  characterizations of  maximal outerplanar graphs
and cacti with  a given  (monotonic)  bend  number.

 The recognition problem for $B_k$ ($B_k^m$) and the characterization
 of $B_k$-EPG ($B_k^m$-EPG) 
graphs is also relevant from an algorithmic point of view.
Indeed several difficult  problems from algorithmic graph theory
turn out to be  tractable  on  $B_k$-EPG graphs for a given  $k\in \nz$, 
see~\cite{OCliqueColoringB1,NMaxIndependentSetInB1EPG,RMaxCliqueInB2,J}.
\medskip

{\bf Organization of the paper.}
In \autoref{sec:OuterplanarGraphs} we  show that the monotonic bend number of outerplanar graphs is at
most $2$. 
In \autoref{sec:nSun} we  derive the  (monotonic) bend number of  the so-called
$n$-sun graph, for
$n\in \nz$, $n\geqslant 3$. The results on $n$-sun graphs  are  relevant for
the computation of the (monotonic) bend number of outerplanar graphs. 
In \autoref{sec:outerplanar_triangulations} and \autoref{sec:Cacti} we  deal
with   maximal outerplanar graphs and  cacti, respectively,  and derive a full
characterization for their membership  in $B_{0}$, $B_{1}^{m}$, $B_{1}$ and 
$B_{2}^{m}$.
We conclude with a summary  and some open questions in
\autoref{sec:Conclusions}.
 \medskip

{\bf Terminology  and notation.} 
The crossings of two grid lines are called \emph{grid points}.
The part of a grid line between two consecutive  grid points is called a \emph{grid edge}
and the two consecutive grid points are the \emph{vertices of the grid edge}.
We distinguish between \emph{horizontal grid edges} and \emph{vertical grid edges}.
A \emph{path $P$ on a grid} is a sequence of grid points
$\seq{v_0,\ldots,v_i,\ldots,v_p}=:P$ such that any 
 two consecutive grid points
 are connected by a grid edge.  
 The later  are called \emph{edges of the path $P$}. The grid point
 $v_0$ is
called  the  \emph{start point} of $P$  and $v_p$ is called the \emph{end point} of $P$.
 A   \emph{bend point} $v_i$  of $P$, $i\not\in \{0,p\}$, 
 is a grid vertex such that one of  the   edges  $(v_{i-1},v_i)$ and
                                $(v_i,v_{i+1})$ of $P$ is   a
 horizontal grid edge and the other one is  a vertical grid edge. 
 The subpath $\seq{v_{i-1},v_i,v_{i+1}}$  is  called a \emph{bend} on  $P$.
 The part of a path $P$ between two consecutive bend points is called a \emph{segment} of $P$.
 Also the parts of   $P$  from the start point to the first bend point and from
 from the last 
bend point to the end point are called \emph{segments}.  A segment that 
consists of horizontal (vertical) edges  is called \emph{horizontal} 
(\emph{vertical}) segment.
We say that two paths on a grid \emph{intersect}, if they  have at least one common grid edge.

If a graph $G$ contains an induced subgraph
isomorphic to a graph $H$, we consider that  induced subgraph to be  a \emph{copy
    of $H$ in $G$} and say that \emph{$G$ contains a copy of $H$}.
Finally notice  that when talking about an outerplanar graph $G$,
we consider an arbitrary but fixed outerplanar embedding of $G$.





\section{Outerplanar graphs are in \texorpdfstring{$B_{2}^{m}$}{B2m}}
\label{sec:OuterplanarGraphs}


Biedl and Stern~\cite{C} proved 
 that every outerplanar graph is in $B_{3}$ and showed the existence of an
 outerplanar graph which does not  belong to $B_{1}$. Furthermore, 
 they conjectured
that every outerplanar graph is in $B_{2}$. This conjecture was confirmed by
Heldt, Knauer and Ueckerdt in~\cite{D} who proved 
that every graph with treewidth at most $2$ is in $B_{2}$ and exploited the
fact  
that  every outerplanar graph  has treewidth at most 
$2$, see~\cite{OuterplanarGraphsTreewidth2}.
%
%
%
Thus $2$  is a tight upper bound on the bend number of outerplanar graphs.
We  strengthen this result and show that all outerplanar graphs
belong even to $B_2^{m}$.

To prove this result we construct a  $B_2^m$-EPG representation of a given
 outerplanar graph $G=(V,E)$.  Without loss of generality we
assume $G$ to be connected (a $B_2^m$-EPG representation of an
arbitrary graph  can be  obtained as the union of the disjoint  $B_2^m$-EPG
representations of its connected components). 
Our construction builds upon a  so-called \emph{nice labeling} of the vertices
of $G$, which is a  particular
(not-unique) labeling  of the vertices  of an outerplanar
graph obtained as follows.
Consider an outerplanar embedding of   $G$ and  a new vertex  
$v_0\not\in V$.
Let $n=|V|$ be the order of $G$. Consider some  planar embedding of the planar graph
$G' := (V',E')$ with $V' := V \cup \{v_{0}\}$ and $E' := E \cup \{\{v_{0},i\}:
i \in V\}$,   where every edge in $E'\setminus E$ is drawn so as to start with a short
straight line segment at $v_0$. Then label the vertices of~$V$  by
$v_1,v_2,\ldots,v_n$ such that the straight line segments of the edges
$\{v_0,v_i\}$ are positioned around $v_0$ in counterclockwise order.

In the following we always consider an outerplanar graph with a 
nice
labeling and   identify  its vertices   with  the corresponding 
labels.

We  first observe that the 
nice labeling  of a connected outerplanar graph fulfills two simple but
useful  properties, the \emph{separation
    property} and the \emph{path separation property}, defined below.
  Indeed, these properties are equivalent to each other.
\begin{definition}
	Consider an outerplanar graph $G$ on $n$ vertices   labeled by $(v_{1}, 
	\dots, v_{n})$.  The labeling of the vertices is said to have the
	\emph{separation property} iff for any edge $\{v_i,v_j\}$ in $G$  with 
	$i<j$ 
	the following
	holds:  if $\{v_k,v_\ell\}$ is an edge in
	$G$  with $k,\ell\in \{1,2,\ldots,n\}$ and
	$\{k,\ell\}\cap\{i,j\}=\emptyset$, then either  $i<k<j$ and  $i<\ell<j$ 
	hold, or $k \not\in
	\{i,\ldots, j\}$ and $\ell \not\in
	\{i,\ldots, j\}$ hold.
	
	Analogously, the labeling of the vertices is said to have the
	\emph{path separation property} iff for any path $P = (v_i = v_{p_1}, 
	v_{p_2}, \dots, v_{p_r}=v_j)$ in $G$  with 
	$i<j$ the following
	holds:  if $\{v_k,v_\ell\}$ is an edge in
	$G$  with $k,\ell\in \{1,2,\ldots,n\}$ and
	$\{k,\ell\}\cap\{p_1,\dots,p_r\}=\emptyset$, then either  $i<k<j$ and  
	$i<\ell<j$ hold, or $k \not\in
	\{i,\ldots, j\}$ and $\ell \not\in
	\{i,\ldots, j\}$ hold.    
\end{definition}

\begin{observation}
Let $G=(V,E)$ be a connected outerplanar graph  with vertex set 
$V=\{1,2,\ldots,n\}$.
Any nice labeling $(v_1, v_2, \ldots, v_n)$  of the vertices has the 
separation property and the path separation property.   
\end{observation}

\begin{proof} Consider a planar embedding of  $G'$ which gives rise to  
  the nice labeling  of the vertices of $G$ as described above.
To prove  the separation property consider an  arbitrary  edge $\{v_i,v_j\}$ in
$G$ and any other edge  $\{v_k,v_\ell\}$ in
$G$  such that  $k,\ell\in \{1,2,\ldots,n\}$ and
$\{k,\ell\}\cap\{i,j\}=\emptyset$ hold. Then,  $v_0,v_i,v_j$ close a cycle $C$
of length $3$ in $G'$ and $v_k$ and $v_\ell$ have to be  either both inside or both outside of $C$.  
In the first case $i<k<j$ and  $i<\ell<j$ hold, in the second case $k  \not\in
\{i,\ldots, j\}$ and $\ell \not\in \{i,\ldots, j\}$ hold.

The path separation property can be shown by analogous arguments.
\end{proof}

Next we present \autoref{outerplanarB2m:construct} which  constructs a $B_2^m$-EPG representation of a given
connected outerplanar graph $G$ with a given  nice labeling of its
vertices. The algorithm considers the vertices of $G$
one by one in a particular  order.
The subroutine \textproc{Explore}$(v,\dots)$ constructs a  monotonic path on the
grid with at most two horizontal
segments and at most one vertical segment for every neighbor of $v$, for which
no path has been constructed yet. 
The construction of the paths is done while maintaining the following
  invariant property
   at every call of \textproc{Explore}$(v_i,\dots)$  during  the execution of the algorithm:
  the path $P_{v_i}$ on the grid corresponding to the  green vertex $v_i$ which 
  is currently being
  explored,  
  has  a \emph{free} region  $\mathcal{R}_{v_i}$ on the lower horizontal segment,
  i.e.\ a region where there is   no  intersection of $P_{v_i}$ with any of the
  paths constructed so far.
The construction of the  paths corresponding to the neighbors of $v_i$ in \textproc{Explore}$(v_i,\dots)$ is illustrated in
\autoref{fig:construction_out_in_b2m}. 
There  we highlight 
the free parts existing
at the moment when \textproc{Explore}$(v_i,\dots)$  is called  in light gray
and the free parts of the paths constructed during
\textproc{Explore}$(v_i,\dots)$ in dark gray.

The subroutine \textproc{UpdateGraph}$(v,\dots)$ deletes from $G$ the 
vertex~$v$ which
was explored  at  last together with  its incident edges.
Moreover, it deletes the edges connecting the  neighbors of $v$ among
each other,  and also the edge connecting the last neighbor of $v$ to the
green vertex with the second smallest label, if such an edge exist.
Notice that the deleted edges are precisely the edges for which
the corresponding intersections of paths on the grid have already  been constructed
during the last call of  \textproc{Explore}$(v,\dots)$. 
\begin{algorithm}[tbp]
  \footnotesize
    \caption{Construct a $B_2^m$-EPG representation of a  connected  outerplanar graph}
    \label{outerplanarB2m:construct}
    \Input{A connected  outerplanar  graph  $G = (V,E)$ with $n = |V|$ and $V = \{v_1, v_2,
        \dots, v_n\}$ ordered according to a nice
        labeling 
    }
    \Output{A $B_2^m$-EPG representation of $G$}
    \begin{algorithmic}[1]
        \Procedure{B2M\_Outerplanar}{$G$}
\State Set $\textAlgo{color}(v_i):=\textAlgo{gray}$ for all $i\in 
\{1,2,\ldots,n\}$
\State Set $\textAlgo{color}(v_1):=\textAlgo{green}$, 
$\textAlgo{ListGreen}:=\{v_1\}$\label{color}
\longState{Draw  the path
$P_{v_1}$ representing $v_1$ as a straight horizontal line on the
grid\label{draw}}
\State Set $G':= G$, $V':=V$, $E':=E$
\While{$\textAlgo{ListGreen}\neq \emptyset$}
\State Set $i:= \min \{ j\colon v_{j} \in \textAlgo{ListGreen}\}$ \label{alg: 
b2m out set i}
\State Set $i^{\ast}:= \inf \{ j\colon v_{j} \in \textAlgo{ListGreen}, j\neq 
i\}$ (may be $\infty$)
\State Set $\ell:=\deg_{G'}(v_i)$
\longState{Let $\textAlgo{ListNeighbors} 
:=(v_{i_1},v_{i_2},\ldots,v_{i_{\ell}})$ be the list 
of the $\ell$  neighbors of $v_i$ in $G'$
such
that $1\leqslant i_1<i_2<\ldots <i_\ell \leqslant n$}\label{alg: b2m out neighbors}
\State $\textAlgo{color} :=$ \Call{Explore}{$v_i$, $E'$, 
$\textAlgo{ListNeighbors}$, 
$i^\ast$, $\textAlgo{color}$}\label{algLine: b2m call explore}
\State $(V', E') :=$ \Call{UpdateGraph}{$v_i$, $V'$, $E'$, 
$\textAlgo{ListNeighbors}$, $i^\ast$} \label{alg: b2m out update graph}
\State $G' := (V',E')$
\State $\textAlgo{ListGreen} := \{v_j \in V': \textAlgo{color}(v_j) = 
\textAlgo{green}\}$ 
\label{alg: 
b2m out list green}
\EndWhile
\State \Return
\EndProcedure
\end{algorithmic}
\end{algorithm}
\begin{algorithm}[tbp]
  \footnotesize
    \begin{algorithmic}[1]
        \Procedure{Explore}{$v_i$, $E'$, $\textAlgo{ListNeighbors}$, $i^\ast$, 
        $\textAlgo{color}$}
\State Set $\textAlgo{color}(v_{i_j}):=\textAlgo{green}$ for all $j \in \{1, 2, 
\dots, 
\ell\}$ 
\label{color_in_explore}
 \For{$j = 1, \dots , \ell-1$}
\If{$j=1$ 
}  \label{startdraw_in_explore}
\longState{Construct
  the path 
  $P_{v_{i_1}}$   as
shown in \autoref{fig:construction_out_in_b2m}(a) ($P_{v_i}$ in the
picture represents the lower horizontal segment of the already constructed path 
$P_{v_i}$)}
\ElsIf{$\{v_{i_{j-1}},v_{i_j}\}\not\in E'$} 
\longState{Construct the path $P_{v_{i_j}}$  as
shown in \autoref{fig:construction_out_in_b2m}(b)}
\Else
\longState{Construct the path $P_{v_{i_j}}$ as
shown in \autoref{fig:construction_out_in_b2m}(c)}
\EndIf
\EndFor
\If{$i^\ast = \infty$  or $\{v_{i_{\ell}},v_{i^{\ast}}\}\not \in E'$  }
\longState{Construct the path 
	$P_{v_{i_\ell}}$   as
shown in \autoref{fig:construction_out_in_b2m}(d), where the dotted line is drawn iff $\{v_{i_{\ell-1}},v_{i_{\ell}}\} \in E'$ (If $i^\ast = \infty$, then $P_{v_{i^{\ast}}}$ is not there. Otherwise $P_{v_{i^{\ast}}}$ in
the picture represents
 the lower horizontal segment of the already constructed path $P_{v_{i^{\ast}}}$)}
\Else
\longState{Construct the path 
	$P_{v_{i_\ell}}$  as
shown in \autoref{fig:construction_out_in_b2m}(e), where the dotted line is 
drawn iff $\{v_{i_{\ell-1}},v_{i_{\ell}}\} \in E'$}
\EndIf\label{enddraw_in_explore}
\State \Return $\textAlgo{color}$
\EndProcedure
\end{algorithmic}
\end{algorithm}

\begin{figure}[tbp]
    \centering
    \begin{minipage}[b]{0.21\linewidth}
        \centering
        \begin{center}
\begin{tikzpicture}
  [scale=.65]

\newcommand\Vspace{0.15}

    \fill[gray!20] 
    (-0.2,4) --(-0.2,0.7)  
    to[out=-20, in=200] 
    (4.2,0.7) -- (4.2,4) 
    to[out=170, in=10] 
    (-0.2,4);
    \node[below,gray] at (3.5,0.5) {$\mathcal{R}_{v_i}$};

    \fill[fill=gray!70] (1.5,1) ellipse (0.4 and 0.6);
    \node[right,gray] at (2,1.1) {$\mathcal{R}_{v_{i_{1}}}$};

	\draw[dashed] (0,3+\Vspace) -- (1,3+\Vspace);
	\draw[dashed] (3,3+\Vspace) -- (4,3+\Vspace);

	\draw (1,3+\Vspace) --  node[above,pos=0.85]  {$P_{v_{i}}$} (3,3+\Vspace); 
	
	\draw (1,1) --  (2,1) -- node[left,pos=0.5]  {$P_{v_{i_{1}}}$} (2,3) -- (2.5,3);

\end{tikzpicture}
\end{center}
        (a)
    \end{minipage}
    \quad
    \begin{minipage}[b]{0.30\linewidth}
        \centering
        \begin{center}
\begin{tikzpicture}
  [scale=.63]

\newcommand\Vspace{0.15}

    \fill[gray!20] 
    (-0.2,4) --(-0.2,1)  
    to[out=-30, in=210] 
    (6,1) -- (6,4) 
    to[out=170, in=10] 
    (-0.2,4);
    \node[below,gray] at (5.5,0.7) {$\mathcal{R}_{v_i}$};  
    
    \fill[fill=gray!70] (3.5,2) ellipse (0.4 and 0.6);
    \node[right,gray] at (4,1.5) {$\mathcal{R}_{v_{i_j}}$};  
    
    \fill[fill=gray!70] (1.5,1) ellipse (0.4 and 0.6);
    \node[right,gray] at (2,0.5) {$\mathcal{R}_{v_{i_{j-1}}}$};

	\draw[dashed] (0,3+\Vspace) -- (1,3+\Vspace);
	\draw[dashed] (5,3+\Vspace) -- (5.8,3+\Vspace);
	\draw (1,3+\Vspace) --  node[above,pos=0.85]  {$P_{v_{i}}$} (5,3+\Vspace); 
	
	\draw (1,1) --  (2,1) -- node[left,pos=0.5]  {$P_{v_{i_{j-1}}}$} (2,3) -- (3,3); 
	
		\draw (3,2) --  (4,2) -- node[right,pos=0.5]  {$P_{v_{i_{j}}}$} (4,3) -- (4.5,3);

\end{tikzpicture}
\end{center}
        (b)
    \end{minipage}
    \quad
    \begin{minipage}[b]{0.35\linewidth}
        \centering
        \begin{center}
\begin{tikzpicture}
  [scale=.65]

\newcommand\Vspace{0.15}

    \fill[gray!20] 
    (-0.2,4) --(-0.2,1)  
    to[out=-20, in=200] 
    (7,1) -- (7,4) 
    to[out=170, in=10] 
    (-0.2,4);
    \node[below,gray] at (6.5,0.8) {$\mathcal{R}_{v_i}$};  
    
  \fill[fill=gray!70] (3.5,2) ellipse (0.4 and 0.6);
  \node[right,gray] at (4,1.5) {$\mathcal{R}_{v_{i_j}}$};  
  
  \fill[fill=gray!70] (1.5,1+\Vspace) ellipse (0.4 and 0.6);
  \node[right,gray] at (2,0.5+\Vspace) {$\mathcal{R}_{v_{i_{j-1}}}$};

	\draw[dashed] (0,3+\Vspace) -- (1,3+\Vspace);
	\draw[dashed] (6,3+\Vspace) -- (6.8,3+\Vspace);
	\draw (1,3+\Vspace) --  node[above,pos=0.85]  {$P_{v_{i}}$} (6,3+\Vspace); 
	
	\draw (1,1+\Vspace) --  (2,1+\Vspace) -- node[left,pos=0.5]  {$P_{v_{i_{j-1}}}$} (2,3) -- (4.3,3); 
	
		\draw (3,2) --  (4,2) -- node[right,pos=0.4]  {$P_{v_{i_{j}}}$} (4,3-\Vspace) -- (5.5,3-\Vspace);

\end{tikzpicture}
\end{center}
        (c)
    \end{minipage}
    \centering
    \begin{minipage}[b]{0.45\linewidth}
        \centering
        \begin{center}
\begin{tikzpicture}
  [scale=.65]

\node[draw=white,fill=white] at (7,6) {};

\newcommand\Vspace{0.15}

    \fill[gray!20] 
    (-0.2,4) --(-0.2,1)  
    to[out=-20, in=200] 
    (3.8,1) -- (3.8,4) 
    to[out=170, in=10] 
    (-0.2,4);
    \node[below,gray] at (3.3,0.8) {$\mathcal{R}_{v_i}$};  
    
    \fill[gray!20] 
    (4.6,5.8) --(4.6,4)  
    to[out=-20, in=200] 
    (9,4) -- (9,5.8) 
    to[out=170, in=10] 
    (4.6,5.8);
    \node[below,gray] at (8.4,3.8) {$\mathcal{R}_{v_{i^\ast}}$};  
    
    \fill[fill=gray!70] (5.4-\Vspace,3-\Vspace) ellipse (0.4 and 0.6);
    \node[right,gray] at (5.9-\Vspace,2.5-\Vspace) 
    {$\mathcal{R}_{v_{i_\ell}}$};  
    
    \fill[fill=gray!70] (1.5,2) ellipse (0.4 and 0.6);
    \node[left,gray] at (2.2,1.3) {$\mathcal{R}_{v_{i_{\ell-1}}}$};      
    
    \fill[fill=gray!70] (7.3,5) ellipse (0.4 and 0.6);
    \node[right,gray] at (7.6,4.5) {$\mathcal{R}_{v_{i^\ast}}$};

	\draw[dashed] (0,3+\Vspace) -- (1,3+\Vspace);
	\draw (1,3+\Vspace) --  node[above,pos=0.5]  {$P_{v_{i}}$} (4,3+\Vspace); 
    \draw[dashed] (4,3+\Vspace) -- (4,3.7+\Vspace);
	
	\draw (1.5,2) --  (2,2) -- node[right,pos=0.1]  {$P_{v_{i_{\ell-1 }}}$} (2,3) -- (3,3); 
	
	\draw[dotted] (2.5,3-\Vspace) -- (3.5,3-\Vspace);
	
	\draw (3.5,3-\Vspace) --  (6-\Vspace,3-\Vspace) -- node[right,pos=1.0]  {$P_{v_{i_{\ell}}}$} (6-\Vspace,4); 

	\draw[dashed] (4.8,5) -- (5.5,5);
    \draw[dashed] (8,5) -- (8.8,5);
    \draw (5.5,5) -- node[above,pos=0.2]  {$P_{v_{i^\ast}}$} (8,5);

\end{tikzpicture}
\end{center}
        (d)
    \end{minipage}
    \quad
    \begin{minipage}[b]{0.45\linewidth}
        \centering
        \begin{center}
\begin{tikzpicture}
  [scale=.65]

\newcommand\Vspace{0.15}

    \fill[gray!20] 
    (-0.2,4) --(-0.2,1)  
    to[out=-20, in=200] 
    (3.8,1) -- (3.8,4) 
    to[out=170, in=10] 
    (-0.2,4);
    \node[below,gray] at (3.3,0.8) {$\mathcal{R}_{v_i}$};  
    
        \fill[gray!20] 
        (4.6,5.8) --(4.6,4)  
        to[out=-20, in=200] 
        (9,4) -- (9,5.8) 
        to[out=170, in=10] 
        (4.6,5.8);
        \node[below,gray] at (8.4,3.8) {$\mathcal{R}_{v_{i^\ast}}$};  
    
    \fill[fill=gray!70] (5.4-\Vspace,3-\Vspace) ellipse (0.4 and 0.6);
    \node[right,gray] at (5.9-\Vspace,2.5-\Vspace) 
    {$\mathcal{R}_{v_{i_\ell}}$};  
    
    \fill[fill=gray!70] (1.5,2) ellipse (0.4 and 0.6);
    \node[left,gray] at (2.2,1.3) {$\mathcal{R}_{v_{i_{\ell-1}}}$};      

    \fill[fill=gray!70] (7.3,5) ellipse (0.4 and 0.6);
    \node[right,gray] at (7.6,4.5) {$\mathcal{R}_{v_{i^\ast}}$};

	\draw[dashed] (0,3+\Vspace) -- (1,3+\Vspace);
	\draw (1,3+\Vspace) --  node[above,pos=0.5]  {$P_{v_{i}}$} (4,3+\Vspace); 
    \draw[dashed] (4,3+\Vspace) -- (4,3.7+\Vspace);
	
	\draw (1,2) --  (2,2) -- node[right,pos=0.1]  {$P_{v_{i_{\ell-1}}}$} (2,3) -- (3,3); 
	
	\draw[dotted] (2.5,3-\Vspace) -- (3.5,3-\Vspace);
	\draw (3.5,3-\Vspace) --  (6-\Vspace,3-\Vspace) -- node[right,pos=0.5]  {$P_{v_{i_{\ell}}}$} (6-\Vspace,5-\Vspace) -- (6.5,5-\Vspace); 

	\draw[dashed] (4.8,5) -- (5.5,5);
	\draw[dashed] (8,5) -- (8.8,5);
	\draw (5.5,5) -- node[above,pos=0.2]  {$P_{v_{i^\ast}}$} (8,5);

\end{tikzpicture}
\end{center}
        (e)
    \end{minipage}
    \caption{The constructions of the subroutine \textproc{Explore} of \autoref{outerplanarB2m:construct}.}
    \label{fig:construction_out_in_b2m}
\end{figure}
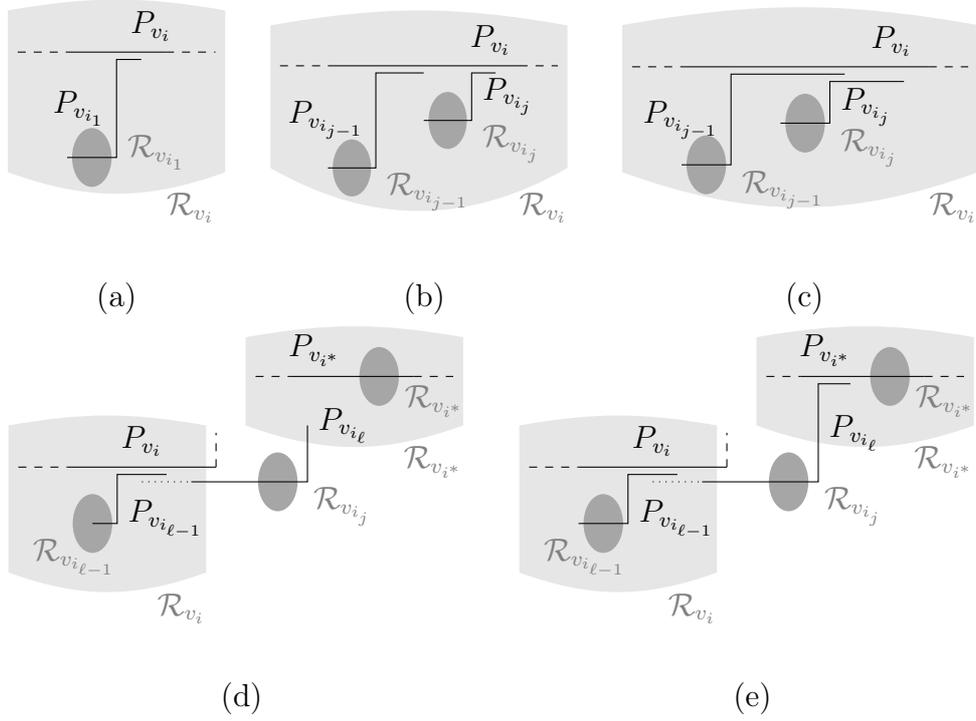

\begin{algorithm}[tbp]
  \footnotesize
    \begin{algorithmic}[1]
        \Procedure{UpdateGraph}{$v_i$, $V'$, $E'$, $\textAlgo{ListNeighbors}$, 
        $i^\ast$}
\State Set $V':=V'\setminus \{v_i\}$, $E':=E'\setminus \{\{v_i,v_k\}\colon
\{v_i,v_k\}\in E'\}$
\For{$j=1,\ldots,\ell-1$} 
\If{$\{v_{i_j},v_{i_{j+1}}\}\in E'$} 
\State $E':=E'\setminus \{\{v_{i_j},v_{i_{j+1}}\}\}$
\EndIf
\EndFor
\If{$\{v_{i_{\ell}},v_{i^{\ast}}\}\in E'$}
\State $E':=E'\setminus \{\{v_{i_{\ell}},v_{i^{\ast}}\}\}$ \label{alg:deleteEgdgeBetweenLastNeighborAndSecondSmallesGreen}
\EndIf
\State \Return $(V',E')$
\EndProcedure
\end{algorithmic}
\end{algorithm}

Next we show  that \autoref{outerplanarB2m:construct} is well-defined.
\begin{lemmaMy}\label{welldefined:lemma}
	\autoref{outerplanarB2m:construct} is well-defined, 
	i.e.\   the  paths on the grid
	can be constructed as described in the algorithm and it terminates. 
	Furthermore
	\autoref{outerplanarB2m:construct} constructs exactly one path for each 
	vertex of the graph.
\end{lemmaMy}
\begin{proof}

	Let $G = (V,E)$ be a connected outerplanar graph  with vertex set 
	$V=\{1,2,\ldots,n\}$
	and let $(v_1, v_2, \ldots, v_n)$ be a nice labeling of its vertices.

	We first prove the first statement of this lemma.
	Observe first that whenever  \autoref{outerplanarB2m:construct}  starts to 
	explore a new vertex $v_i$ in 
	line~\ref{algLine: b2m call explore}  all paths $P_v$ that correspond to 
	some 
	green vertex $v$ 
	have a \emph{free} part $\mathcal{R}_{v}$ on their lower horizontal 
	segment, 
	which is shared by no  other 
	path constructed so far. By construction these free parts are located in 
	the grid in such 
	a way that if $v_k$ and $v_j$ are two green vertices with $k<j$, then the 
	right-most point of 
	the free part of $P_{v_k}$ is located to the left and below the left-most 
	point 
	of the free 
	part of $P_{v_j}$.
	Obviously, this property  clearly holds   when $v_1$ is explored and  it is
	clearly maintained in the later constructions. 
	In \autoref{fig:construction_out_in_b2m}  the free parts of the current 
	iteration are highlighted  in light gray and the free parts determined for 
	the 
	next iteration in 
	dark gray.
	The paths $P_{v_i}$ and $P_{v_{i^*}}$ are located as  depicted in
	\autoref{fig:construction_out_in_b2m}(d) and~(e). Thus, 
	\textproc{Explore($v_i$)} and hence \autoref{outerplanarB2m:construct} can 
	really construct all the paths as
	described.
	
	Observe further that  the termination of the  algorithm is clear:  any 
	explored
	vertex is  deleted and there are only finitely many vertices in $G$.

	Next we prove the second statement of the lemma.  
	\autoref{outerplanarB2m:construct} constructs  a path 
	$P_v$ corresponding to a vertex $v$ in the same call of \textproc{Explore}  
	as 
	it  
	colors 
	$v$  green.  Due to the connectivity of $G$ each vertex is colored green 
	during 
	the algorithm, so it
	constructs at least one path for  each vertex. Thus, to prove the second
	statement of the lemma it is enough to show that whenever
	a vertex is colored green, is has not been colored green before.
	Equivalently, we  show that whenever we explore a 
	vertex $v_i$ in the 
	subroutine $\textproc{Explore}$ all its current neighbors (i.e.\ all its 
	neighbors in the 
	current~$G'$) are still colored $gray$. This is done in two steps. First 
	we  show 
	that (i) all current neighbors of $v_i$ are contained in the set $\{v_1, 
	v_2, \dots, 
	v_{i^\ast}\}$ whenever $i^\ast < \infty$. By construction $v_i$ and 
	$v_{i^\ast}$ are the only 
	green vertices in this set, so  showing that (ii) 
	$v_{i^\ast}$ is 
	not a current neighbor of $v_i$ completes the proof. 
	
	\medskip
	Proof of (i). We  show that if $i^\ast < \infty$ after executing 
	line~\ref{alg: 
		b2m 
		out neighbors} of 
	\autoref{outerplanarB2m:construct}, then $\textAlgo{ListNeighbors} 
	\subseteq \{v_1, v_2, \dots, v_{i^\ast}\}$. 
	In other words and using the notation of line~\ref{alg: b2m out neighbors} 
	we 
	prove that 
	$i_j \leqslant i^\ast$ holds for all neighbors $v_{i_j}$ of $v_i$ in~$G'$.
	Indeed,  assume by contradiction that there is a $j^\ast$ such that 
	$i_{j^\ast} > 
	i^\ast$. The 
	vertex $v_{i^\ast}$ is colored green, hence it was colored green during the 
	exploration  of
	some vertex $x \neq v_i$ which was  explored before~$v_i$.
	The vertex $x$ itself  was colored green during the exploration  of 
	some vertex
	$z$ and so on, until  $v_1$,  the very first vertex explored, is 
	reached.
	Let  $P=(v_1=v_{p_1},v_{p_2},\ldots,  v_{p_r}=v_{i^\ast})$ 
	be a path in $G$ such that $v_{p_j}$ is colored green during the 
	exploration of
	$v_{p_{j-1}}$ for every $2\leqslant j\leqslant 
	r$.                                                                 
	Obviously
	 $p_j \neq i$ and $p_j \neq i_{j^\ast}$ holds for every $1\leqslant 
	j\leqslant r$,
	because explored vertices are deleted, $v_i$ is explored in this iteration 
	and 
	and 
	$v_{i_{j^\ast}}$ is in $G'$. But then,  the path $P$ from $v_1$ 
	to 
	$v_{i^\ast}$ and the edge $\{v_i, v_{i_{j^\ast}}\}$  contradict  the 
	path separation 
	property. 
	
	\medskip
	Proof of (ii). Next we  prove  $v_{i^\ast} \not \in 
	\textAlgo{ListNeighbors}$.
	Assume for contradiction  
	that $v_{i^\ast}$ is among the current 
	neighbors of $v_i$ at some time during  the algorithm, thus  assume that
	$i_\ell = i^\ast$ holds. Consider  the value of
	$i$ at the earliest occurrence of the above phenomenon.
	Let $P$ be the path from $v_1$ to $v_{i^\ast}$ as before. 
	Analogously, let $Q=(v_1=v_{q_1},v_{q_2},\ldots, v_{q_s}=v_i)$ be  
	the path in $G$ such that $v_{q_j}$ is colored
	green during the exploration of $v_{q_{j-1}}$ for every $2\leqslant 
	j\leqslant 
	s$. 
	Both $P$ and $Q$ start in $v_1$ and end in different vertices. Moreover,    
	for 
	every
	vertex $x$ in $P$ or $Q$  there is exactly one vertex $y$ such that $x$ was 
	colored green
	during the exploration of $y$  (because $i$ was chosen minimal with the
	corresponding property).  Consequently  there is an index $j'$ such that 
	$p_{j} 
	= 
	q_{j}$ for all $j \leqslant j'$ and furthermore $p_j \neq q_k$ for all 
	$j>j'$ 
	and $k>j'$, i.e.\  the paths $P$ and $Q$ coincide for the first $j'$ 
	vertices and then use 
	different vertices.

	First we observe that $p_j, q_k \notin \{i, i+1, \dots, i^\ast\}$ for all 
	$j < 
	r$ and $k < s$, 
	because otherwise the existence of the edge $\{v_i,v_{i^\ast}\}$ and the 
	paths 
	$(v_1=v_{p_1},v_{p_2},\ldots, v_{p_{r-1}})$ and 
	$(v_1=v_{q_1},v_{q_2},\ldots, 
	v_{q_{s-1}})$ 
	would contradict the path separation property.
	Next due to the path separation property with the path~$P$ from $v_1$ to 
	$v_{i^\ast}$, the 
	fact that $Q$ and $P$ use different vertices after 
	$v_{p_{j'}}=v_{q_{j'}}$ and the fact that $q_r = i < i^\ast$ it follows 
	that 
	$q_{k} < i^\ast$.
	Therefore $q_{k} \leqslant i$ for all 
	$j' < k \leqslant s$.
	Due to the separation property these indices have to be ascending, i.e.\
	$q_{k} \leqslant 
	q_{k+1}$ for all $j' < k \leqslant s-1$. Analogously, we have $i^\ast 
	\leqslant 
	p_{j+1} 
	\leqslant p_{j}$ for all $j' < j \leqslant r-1$.
	
	Clearly, when $v_{p_{j'}}$ is explored, both $v_{q_{j'+1}}$ and 
	$v_{p_{j'+1}}$ 
	are among the 
	neighbors of $v_{p_{j'}}$, 
	because they are colored green during its exploration. Furthermore 
	$q_{j'+1} 
	\leqslant i < 
	i^\ast \leqslant p_{j'+1}$, hence $v_{q_{j'+1}}$ is explored before 
	$v_{p_{j'+1}}$. By 
	iteratively using this argument we get that all $v_{q_{k}}$,  $j'+1 
	\leqslant k
	\leqslant s$,  are 
	explored before $v_{p_{j'+1}}$. However, when exploring $v_i = v_{q_{s}}$, 
	$v_{i^\ast}$ had
	already been colored green. This is only possible if $p_{j'+1} = i^\ast$, 
	and 
	hence 
	$v_{i^\ast}$  has already been colored green during the exploration of 
	$v_{p_{j'}}$.
	
	Now let us consider the exploration of $v_{q_{s-1}}$. At this 
	point $q_{s-1}$ is the smallest green index. Furthermore the above 
	arguments 
	show that during 
	the exploration of $v_{q_{s-1}}$  the vertex $v_{i^\ast}$ has already been
	colored green.
	Hence, the second  lowest green  index is less or equal to $i^\ast$. With 
	(i) 
	this implies that the index of
	any  current neighbor of 
	$v_{q_{s-1}}$ is less than or equal to $i^\ast$. Due to the separation 
	property 
	with 
	the edge 
	$\{v_{i},v_{i^\ast}\}$, there can't be a neighbor of 
	$v_{q_{s-1}}$ between $v_{i}$ and $v_{i^\ast}$, therefore $v_{i}$ is the 
	current neighbor of 
	$v_{q_{s-1}}$ with the largest index ( the actual  $i_\ell$ in  the 
	algorithm).
	Note that
	$v_{i^\ast}$ can not   be a neighbor of $v_{q_{s-1}}$ because $i$ was 
	chosen 
	minimal. Furthermore, $i^\ast$ has to 
	be the green vertex with the second smallest index also during the 
	exploration 
	of 
	$v_{q_{s-1}}$ (the actual $i^\ast$ in the algorithm), because if there 
	would be another 
	green vertex 
	between $i$ and $i^\ast$, this vertex would remain green and during the 
	exploration of $v_i$ 
	this vertex and not $i^\ast$ would be the vertex with the second smallest 
	index. But this 
	implies that the edge between $v_{i}$ and $v_{i^\ast}$ is deleted in 
	line~\ref{alg:deleteEgdgeBetweenLastNeighborAndSecondSmallesGreen} of 
	\textproc{UpdateGraph} 
	and therefore the edge $\{v_i,v_{i^\ast}\}$ is not present in~$G'$ anymore 
	when 
	$v_i$ is 
	explored, which contradicts the definition of $v_{i^\ast}$. 
\end{proof}


Using \autoref{welldefined:lemma}
we obtain the following main result of this section.
\begin{theoremMy}
\label{out_in_b2m}
Every   outerplanar graph is in $B_{2}^{m}$.
\end{theoremMy}
\begin{proof}
	We   show  that \autoref{outerplanarB2m:construct} constructs indeed a
	$B_2^m$-EPG representation of a connected outerplanar graph $G$ with a nice
	labeling of its vertices $V=\{1,2,\ldots,n\}$.  
	By \autoref{welldefined:lemma}  \autoref{outerplanarB2m:construct} is 
	well-defined and constructs exactly one path for each 
	vertex.
	Clearly every  path $P_v$  constructed by the  algorithm
	hast  at most two bends by construction. 
	
	What is left to show is that the paths $P_u$ and $P_v$ 
	corresponding to two vertices
	$u$ and $v$  intersect
	iff $\{u,v\}$ is an edge in $G$.
	%
	%
	To prove  this, consider the exploration of $v_i$ and the   construction  
	of 
	$P_{v_{i_1}}$, \dots, $P_{v_{i_\ell}}$.  Clearly,  the 
	paths $P_{v_i}$, $P_{v_{i_1}}$, \dots, $P_{v_{i_\ell}}$ intersect iff the 
	corresponding 
	edge is in the current graph $G'$ and is deleted in the subsequent update 
	of 
	$G'$.
	Furthermore none of the paths  $P_{v_{i_1}}$, \dots, $P_{v_{i_\ell}}$
	intersects any path $P_v$ for  $v\in V\setminus\{v_{i_1}, \dots, 
	v_{i_\ell}\}$,
	except for  $v=v_i$. This is due to the fact that 
	$P_{v_{i_1}}$, \dots, $P_{v_{i_\ell}}$ are contained in the free part 
	$\mathcal{R}_{v_i}$ 
	of $P_{v_i}$.
	Hence during the exploration of $v_i$ there is a new intersection between 
	two 
	paths iff the corresponding 
	edge is in the current graph $G'$, in which case that edge  is deleted in 
	the 
	subsequent update of~$G'$.
	These arguments hold in any exploration step and this completes the proof.
\end{proof}
A straightforward analysis of the time complexity of
\autoref{outerplanarB2m:construct} reveals that a  $B_2^m$-EPG
    representation of an $n$-vertex outerplanar graph  can be 
    constructed in 
    $O(n)$ time. 
    
Finally, observe that $2$ is  a tight upper
bound on the monotonic bend number of outerplanar graphs, because $2$ is  a 
tight upper bound on
the bend number of outerplanar graphs as mentioned at the beginning of this
section.

\section{The (monotonic) bend number  of the \texorpdfstring{$n$}{n}-sun}
\label{sec:nSun}
The $n$-sun graph $S_n$ is a graph of order $2n$ defined as follows.  
\begin{definition}
    \label{def: Sn}
    Let $n\in \nz$, $n \geqslant 3$. The $n$-sun graph $S_{n} = (V,E)$ is the
    graph with vertex set
 $V=\{x_{1},\dots,x_{n},y_{1},\dots,y_{n}\}$ and edge set  $E= E_{1} \cup
 E_{2}$, where  $E_{1} = \{\{x_{i},x_{j}\} \mid 1 \leqslant i < j \leqslant
 n\}$ and 
$E_{2} = \{\{x_{i},y_{i}\}, \{x_{i+1},y_{i}\}\mid 1 \leqslant i < n\} \cup
\{\{x_{1},y_{n}\}, \{x_{n},y_{n}\}\}$.
The vertices $\{x_{1},x_{2},\ldots, x_{n}\}$  are called
\emph{central vertices} of  $S_n$ 
and the edges between them, i.e.\ the edges in $E_1$,  are called \emph{central
  edges}  of  $S_n$.
\end{definition}
  A picture of $S_{3}$ can be found in \autoref{fig:s3_m2_m3} and $S_{n}$ is depicted
  in \autoref{fig:SnInB2m}~(a).
Golumbic, Lipshteyn and  Stern~\cite{startpaper} 
 have shown that $S_{3}$ is in $B_{1}$ and that  $S_{n}$ is not in $B_{1}$ for every $n\geqslant 4$. Cameron, Chaplick and
 Ho\`{a}ng~\cite{E} have  shown that $S_{3}$ is not in $B_{0}$ and 
 not in $B_{1}^{m}$.
 %
It can be easily checked that \autoref{fig:SnInB2m}(b) depicts  a  $B_{2}^{m}$-EPG representation of $S_n$ for
   any $n\geqslant 3$. This implies the following theorem. 
\begin{theoremMy}
    \label{snInB2m}
    The $n$-sun $S_{n}$  is in $B_{2}^{m}$ for all $n \geqslant 3$.
\end{theoremMy}

\begin{figure}[tb]
    \centering
    \begin{minipage}[b]{0.41\linewidth}
        \centering
        \begin{center}
\begin{tikzpicture}
  [scale=.5]

\def \n {10}
\def \nMinus {6}
\def \nToDraw {7}
\def \radius {3cm}
\def \margin {8} 

\foreach \s in {1,...,\nMinus}
{
    
    \draw ({360/\n * (\s - 1/2)}:1.5*\radius) 
    -- ({360/\n * (\s )}:\radius);
    
    \draw ({360/\n * (\s - 1/2)}:1.5*\radius) 
    -- ({360/\n * (\s -1 )}:\radius);    
    
    \foreach \t in {\s,...,\nMinus}
    {
      \draw ({360/\n * (\s - 1)}:\radius) 
       -- ({360/\n * (\t )}:\radius);
    }

  \node[vertex] at ({360/\n * (\s - 1)}:\radius) {$x_{\s}$};
  \node[vertex] at ({360/\n * (\s - 1/2)}:1.5*\radius) {$y_{\s}$};    
}

\draw[dotted] ({360/\n * (\nToDraw - 1)}:\radius) arc ({360/\n * (\nToDraw - 1)}:{360}:\radius);

  \node[vertex] at ({360/\n * (1 - 1)}:\radius) {$\textit{ }$};
  \node[right] at ({360/\n * (1 - 1)-3}:\radius + 10) {$x_{n-2}$};  
  \node[vertex] at ({360/\n * (2 - 1)}:\radius) {$\textit{ }$};
  \node[right] at ({360/\n * (2 - 1)+6}:\radius + 6) {$x_{n-1}$};

  \node[vertex] at ({360/\n * (3 - 1)}:\radius) {$x_{n}$};
  \node[vertex] at ({360/\n * (4 - 1)}:\radius) {$x_{1}$};
  \node[vertex] at ({360/\n * (5 - 1)}:\radius) {$x_{2}$};
  \node[vertex] at ({360/\n * (6 - 1)}:\radius) {$x_{3}$};
  \node[vertex] at ({360/\n * (7 - 1)}:\radius) {$x_{4}$};

  \node[vertex] at ({360/\n * (1 - 1/2)}:1.5*\radius) {$\textit{ }$};
  \node[below] at ({360/\n * (1 - 1/2)-5}:1.6*\radius) {$y_{n-2}$};

  \node[vertex] at ({360/\n * (2 - 1/2)}:1.5*\radius) {$\textit{ }$};
    \node[above] at ({360/\n * (2 - 1/2)+2}:1.6*\radius) {$y_{n-1}$};
  
  \node[vertex] at ({360/\n * (3 - 1/2)}:1.5*\radius) {$y_{n}$};
  \node[vertex] at ({360/\n * (4 - 1/2)}:1.5*\radius) {$y_{1}$};
  \node[vertex] at ({360/\n * (5 - 1/2)}:1.5*\radius) {$y_{2}$};
  \node[vertex] at ({360/\n * (6 - 1/2)}:1.5*\radius) {$y_{3}$};

\end{tikzpicture}
\end{center}
        (a)
    \end{minipage}
    \quad
    \begin{minipage}[b]{0.52\linewidth}
        \centering
        \begin{center}
\begin{tikzpicture}
  [scale=.5]

\def \n {10}
\def \nMinus {9}
\def \nEnd {4}
\def \nEndMinus {3}
\def \nStart {8}

\def \vspace {0.1}

\def \nToDraw {7}
\def \radius {3cm}
\def \margin {8} 

\foreach \s in {1,...,\nEnd}
{
    
    \draw ({-\vspace * \s },0) -- ({-\vspace*\s},\s) -- ({\s + 1},\s)  node [below,pos=0.3]  {$P_{x_{\s}}$};
}

\foreach \s in {\nStart,...,\nMinus}
{
    
    \draw ({-\vspace * \s },0) -- ({-\vspace*\s},\s) -- ({\s + 1},\s)  node [below,pos=0.3]  {\pgfmathparse{\n - \s }%
    $P_{x_{n - \pgfmathprintnumber{\pgfmathresult}}}$};
}
    
   \draw ({-\vspace * \n },0) -- ({-\vspace*\n},\n) -- ({\n + 1},\n)  node [below,pos=0.3]  {$P_{x_{n}}$};

\foreach \s in {1,...,\nEndMinus}
{
    \draw (\s,{\s + \vspace}) -- ({\s + 0.5},{\s + \vspace}) -- ({\s + 0.5},{\s + 1 - \vspace}) --  ({\s + 1},{\s + 1 - \vspace}) node [below,pos=0.99]  {$P_{y_{\s}}$};
    
}

\foreach \s in {\nStart,...,\nMinus}
{
    \draw (\s,{\s + \vspace}) -- ({\s + 0.5},{\s + \vspace}) -- node [right,pos=0.45]  {\pgfmathparse{\n - \s }%
    $P_{y_{n - \pgfmathprintnumber{\pgfmathresult}}}$} ({\s + 0.5},{\s + 1 - \vspace}) --  ({\s + 1},{\s + 1 - \vspace});
    
}    

 \draw (0,{1 + \vspace}) -- ({0.3},{1 + \vspace}) -- node [right,pos=0.95]  {$P_{y_{n}}$} ({ 0.3},{\n - \vspace})  --  ({1},{\n - \vspace});

\draw[dotted] (0.3*\nEnd + 2*\vspace,\nEnd+ 2*\vspace) -- (0.3*\nStart - 2*\vspace,\nStart - 11*\vspace);

\draw[dotted] (\nEnd + 2*\vspace,\nEnd+ 2*\vspace) -- (\nStart - 2*\vspace,\nStart - 2*\vspace);

\draw[dotted] (- \nEnd*\vspace,0.3) -- (- \nStart*\vspace,0.3);
  
\end{tikzpicture}
\end{center}
        (b)
    \end{minipage}
    \caption{(a) The graph $S_{n}$. (b) A $B_{2}^{m}$-EPG representation of $S_{n}$.}
    \label{fig:SnInB2m}
\end{figure}
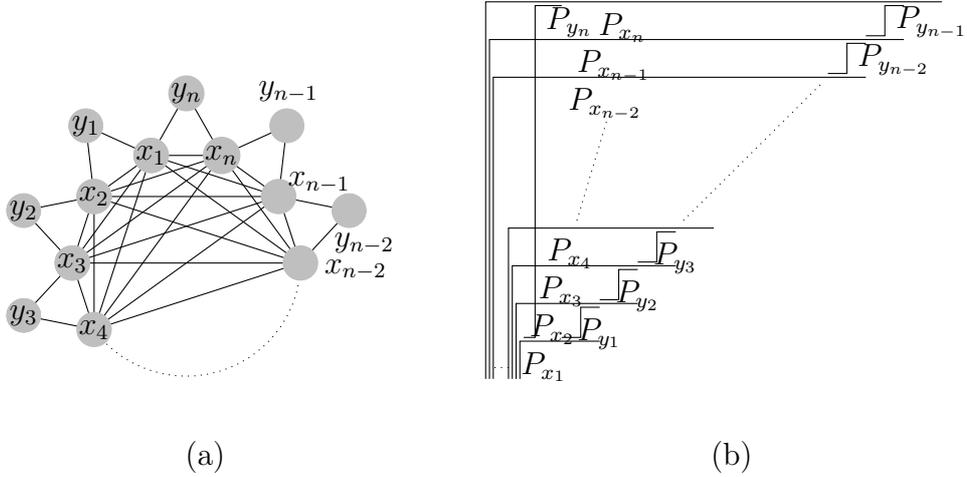
It is easily seen  that the $B_2^m$-representation  of $S_{n}$  in
\autoref{fig:SnInB2m}(b) can be constructed in $O(n)$ time.
We summarize the results for the $n$-sun as follows.
\begin{corollary}
 The (monotonic) bend number of $S_{n}$, $n\geqslant 3$, is given as follows.
$$ b(S_n)=\left \{ \begin{array}{ll}1& \mbox{for } n=3\\2 & \mbox{for 
}n\geqslant
    4\end{array} \right . \quad   
    \qquad
b^m(S_n)=2 \mbox{ for } n\geqslant 3$$
\end{corollary}

Next, we recall the following definitions 
used in the investigation of $S_n$.
%
\begin{definition}[Golumbic, Lipshteyn and Stern~\cite{startpaper}]
	Consider some graph $G$ with  a $B_{1}$-EPG representation and some grid
	edge $e$. 
	The set of all paths which  use  $e$ in the  $B_{1}$-EPG representation
	is called an \emph{edge clique}. 
	Consider a  copy of the claw graph $K_{1,3}$ in the grid, i.e.\ a grid point
	together with three (arbitrarily selected but fixed) grid edges which have 
	that
	grid point as a vertex. The set of all paths
	that use $2$ edges of this copy of the  claw $K_{1,3}$ is called a 
	\emph{claw
		clique}.
	The grid point of degree $3$ in  the claw $K_{1,3}$ is called \emph{central 
		vertex of the claw} and \emph{central grid point of the claw clique}. 
\end{definition}
Clearly,  the vertices of $G$  corresponding to the paths of an edge
clique 
form a clique. Also, the vertices of $G$ corresponding to the paths of a claw
clique form a clique. 
In fact,
a converse statement  is
also true.
\begin{lemmaMy}[Golumbic, Lipshteyn, Stern~\cite{startpaper}]
	\label{MaxCliqueIsEdgeOrClawClique}
	Let $G$ be a $B_1$-EPG graph.
	Then in every $B_{1}$-EPG representation of $G$, the paths corresponding to
	the vertices of a  maximal clique in 
	$G$ form  either  an 
	edge clique or  a claw clique. 
\end{lemmaMy}
In general we  say that   \emph{a set $X$ of vertices in  $G$
	corresponds to an edge clique (a claw clique)} in a $B_1$-EPG 
representation of
$G$ iff   the
paths corresponding to the vertices of $X$ build an edge clique (a claw
clique) in the  $B_1$-EPG representation. 

We close this section with a  simple but  useful observation on 
$S_3$, that will be used in \autoref{sec:outerplanar_triangulations}.
%
\begin{observation}[Biedl, Stern~\cite{C}]
	\label{s3_claw}
	In every $B_{1}$-EPG representation of the graph $S_{3}$ the set of  the  
	central
	vertices 
	$\{x_{1},x_{2},x_{3}\}$ corresponds to a  claw clique. The paths
	corresponding to different vertices contain different pairs of edges of the
	claw, hence the central (grid) point of the claw clique  is a bend
	point  for exactly two of  the three paths corresponding to
	$\{x_{1},x_{2},x_{3}\}$.
	In this case we say that these two \emph{paths are bent in the 
		claw}.
\end{observation}

\section{The (monotonic) bend number of maximal  outerplanar graphs}
\label{sec:outerplanar_triangulations}

In this section we  determine the bend number and the monotonic bend number 
of maximal outerplanar graphs.
\begin{definition}
A graph $G$ is called \textit{maximal outerplanar} if (a) $G$  is outerplanar
and (b)
joining any two non-adjacent vertices of $G$ by  an additional  edge  yields a graph which
is not outerplanar anymore.
\end{definition}
 Notice that maximal outerplanar graphs are closely related to
triangulations.
Indeed, it is easy to see that maximal outerplanar graphs are exactly those 
outerplanar
 graphs, 
where the boundary of the outer face is a Hamiltonian cycle (thus  containing
all vertices)   and every inner face is a triangle.
We make use of the following auxiliary graph.

\begin{definition}
    Let $G$ be a maximal outerplanar graph with an arbitrary but fixed outerplanar embedding. 
    Then the \textit{almost-dual graph
        $\widehat{G}$} of $G$  
    has a vertex for every inner face of $G$.
    Two vertices of $\widehat{G}$ are adjacent if and only if the corresponding
    inner faces of $G$ share an edge in $G$. 
  \end{definition}

Clearly, the almost-dual  $\widehat{G}$ of $G$  is an induced 
subgraph of the
planar dual of~$G$. 
\autoref{fig:const_out_tri_in_b0}(a) represents a maximal outerplanar graph
with $8$ vertices and its almost-dual graph. Observe that in this 
case the almost-dual graph is a path.
Indeed, the following observation is easy to see.
%
\begin{observation}
    \label{obs:widehatGtreeAndDegree3}
    The almost-dual     graph of any     maximal outerplanar graph 
    $G$  is a tree with maximum degree at 
    most~$3$. This 
    tree is a path iff $G$ does not contain $S_3$ as an induced
    subgraph, that is, $G$ is \emph{$S_3$-free}.
\end{observation}

Finally, we use the following notation. For any vertex $\hat{v}$ of 
$\widehat{G}$ denote by~$T_{\hat{v}}$  the 
subgraph of $G$ induced by the vertices on the border of the triangular face of
$G$ associated to $\hat{v}$  and by
$V(T_{\hat{v}})$ the corresponding set of vertices.
\subsection{Maximal Outerplanar Graphs in \texorpdfstring{$B_{0}$}{B0}}
In this section we characterize  maximal outerplanar graphs which belong 
to~$B_{0}$. 
Cameron, Chaplick, Ho\`{a}ng~\cite{E} showed that  $S_{3}$ is not in $B_{0}$. 
Thus, any maximal outerplanar graph 
which is not  $S_3$-free 
is not in $B_{0}$. We show by construction
that the non-trivial converse of this statement is also true:
any $S_3$-free  maximal outerplanar graph  is in $B_{0}$. 
We first prove that \autoref{alg: b0 con out} constructs a 
$B_0$-EPG
representation of an $S_3$-free maximal outerplanar graph.
\begin{algorithm}[tb]
	\footnotesize
	\caption{Construct a $B_0$-EPG representation of an $S_3$-free maximal 
	outerplanar
		graph $G$}
	\label{alg: b0 con out}
	\Input{An $S_3$-free  maximal outerplanar  graph  $G=(V,E)$ and 
	its almost-dual 
	path  $\widehat{G}$}
	\Output{A $B_0$-EPG representation of $G$}
	\begin{algorithmic}[1]
		\Procedure{$B_0$\_$S_3$free\_Max\_Outerplanar}{$G$, $\widehat{G}$}
		\State Let $\hat{v}_{1}, \dots, \hat{v}_{\ell}$ be the consecutive 
		vertices of the path
		$\widehat{G}$ \label{alg: b0 con out def vhat}
		\longState{Label one horizontal line of the grid with $1$ and $\ell+1$ 
			vertical lines of the grid with $1, 2, \ldots, \ell+1$}
		\longState{Let
			$(i,1)$ be the grid point where the one horizontal and
			the $i$-th vertical grid line intersect,   $1\leqslant i\leqslant 
			{\ell+1}$}
		
		\For{$v \in V$}
		\longState{
		Draw the path $P_v$ from $(a,1)$ to $(b+1,1)$, where $\hat{v}_a$ is the 
		first 
		and 
		$\hat{v}_b$ is the last vertex in the path
                  $\widehat{G}$ such that $v \in 
		V(T_{\hat{v}_a})$ and $v \in 
		V(T_{\hat{v}_b})$}
		\EndFor	
		\EndProcedure
	\end{algorithmic}
\end{algorithm}
\begin{lemmaMy} \label{lem: b0 max out algo}
	Let $G$ be an $S_3$-free maximal outerplanar graph and 
	let $\widehat{G}$ be its almost-dual graph.
	Then \autoref{alg: b0 con out} constructs a $B_0$-EPG representation of $G$.
\end{lemmaMy}
\begin{proof}
	First notice  that due to \autoref{obs:widehatGtreeAndDegree3} 
	$\widehat{G}$ is a 
	path as required by the input of \autoref{alg: b0 con out}.
	Next observe that since  $G$ is maximal outerplanar, 
	each of its
		vertices   is contained  in at least one triangle and thus in at least 
		one vertex of the path $\widehat{G}$. 
		Hence \autoref{alg: b0 con out} draws exactly one path $P_v$ for every 
		vertex $v$ of $G$.
	In order to prove the lemma we show that 
	(i) for  each vertex $w \in
	V(G)$ the set of vertices $\hat{v}$ of $\widehat{G}$, for which $w\in
	V(T_{\hat{v}})$ holds,  build a subpath of $\widehat{G}$
	and (ii)  
	the 
	paths
	on the grid $P_u$ and $P_v$ representing two vertices $u$ and $v$ of $G$ 
	intersect iff $\{u,v\}$
	is an edge in $G$.

	\medskip
	Proof of (i). 
	Assume that (i) does not hold, and consider a   $w\in
	V(G)$, for which there exist $j$ and $k
        \geqslant j + 2$  with
	$w\in V(T_{\hat{v}_j})$ and $w\in V(T_{\hat{v}_k})$, but 
	$w\not\in V(T_{\hat{v}_i})$ for any 
	$j + 1 \leqslant i \leqslant k -1$. It is easy to see that 
	for any two consecutive 
	vertices $\hat{v}_i$ and $\hat{v}_{i+1}$ on the path $\widehat{G}$, 
	the triangles $T_{\hat{v}_i}$ and 
	$T_{\hat{v}_{i+1}}$ share 	two vertices in $G$, and every other vertex of 
	$G$ is only in 
	$T_{\hat{v}_r}$ with either $r \leqslant i$ or $r \geqslant i+1$.
	Thus, because $w\in V(T_{\hat{v}_j})$ and $w\not\in V(T_{\hat{v}_{j+1}})$, 
	$w$ can only be in $T_{\hat{v}_r}$ with $r \leqslant j$.
	Furthermore, because $w\in V(T_{\hat{v}_k})$ and $w\not\in 
	V(T_{\hat{v}_{k-1}})$, $w$ can only be in $T_{\hat{v}_r}$ with $r 
	\geqslant k$, a contradiction.
	
	\medskip
	Proof of (ii). 
	Consider 
	an  edge $\{u,v\}$  of $G$. Clearly $\{u,v\}$  is part of at least one
	triangle
	$T_{\hat{v}_i}$ for some $i\in \{1,2,\ldots,\ell\}$. This implies that  
	$P_u$ 
	and $P_v$ share the grid edge connecting $(i,1)$ and $(i+1,1)$. 
	Hence $P_u$ and $P_v$ intersect.
	
	
	

	Assume now that  $P_v$ and $P_u$ intersect, so $P_u$ and $P_v$ share some  
	grid
	edge~$e$ from $(k,1)$ to $(k+1,1)$ for some $k\in \{ 1,2,\ldots, \ell\}$.
	By construction, this implies that $v$ is in $V(T_{\hat{v}_a})$ 
	for some $a \leqslant k$ and in  $V(T_{\hat{v}_b})$ for some $b \geqslant 
	k$.  Then, (i) 
        implies that $v \in V(T_{\hat{v}_k})$. 
	Analogously, $u \in V(T_{\hat{v}_k})$ holds.
	Hence
	$u$ and $v$ are adjacent.
\end{proof}

With the help of \autoref{lem: b0 max out algo} 
we show that the following theorem holds.
\begin{theoremMy}
	\label{out_tri_in_b0}
	Let $G$ be a maximal outerplanar graph. Then $G$ is in $B_{0}$ if and only 
	if $G$ is  $S_{3}$-free.
\end{theoremMy}
\begin{proof}
	Let $G$ be a maximal outerplanar graph. 
	If $G$ is not  $S_{3}$-free, then $G$ is not in $B_{0}$ because of 
	Cameron, Chaplick, Ho\`{a}ng~\cite{E}. 
	If $G$ is $S_3$-free, then 
	\autoref{alg: b0 con out} constructs a $B_0$-EPG representation as stated in
	\autoref{lem: b0 max out algo}.
	Hence, $G$ is  in $B_0$.
\end{proof}
%

\autoref{out_tri_in_b0} and 
\autoref{obs:widehatGtreeAndDegree3} imply that it can be decided in $O(n)$ time 
whether  a maximal outerplanar graph $G$ is in $B_0$. Further, it is not 
difficult to
see   that in the positive case the construction of a $B_0$-EPG representation
in \autoref{alg: b0 con out}
can be done in $O(n)$ time for a graph of order~$n$.


\begin{figure}[tb]
\centering
\begin{minipage}[b]{0.42\linewidth}
\centering
\begin{center}
\begin{tikzpicture}
  [scale=.65]
  
  \node[vertex] (1) at (1,3) {$1$};
  \node[vertex] (2) at (2,5) {$2$};
  \node[vertex] (3) at (3,1) {$3$};
  \node[vertex] (4) at (4,5) {$4$};
  \node[vertex] (5) at (6,5) {$5$};
  \node[vertex] (6) at (5,1) {$6$};
  \node[vertex] (7) at (7,1) {$7$};
  \node[vertex] (8) at (9,1) {$8$};
  \node[vertex] (9) at (8,5) {$9$};

  \node[vertex3] (d1) at (2,3) {};
  \node[vertex3] (d2) at (3,3) {};
  \node[vertex3] (d3) at (4,3) {};
  \node[vertex3] (d4) at (5,3) {};
  \node[vertex3] (d5) at (6,3) {};
  \node[vertex3] (d6) at (7,3) {};
  \node[vertex3] (d7) at (8,3) {};
  
  \node[above] (a1) at (2,3) {$\hat{v}_{1}$};
  \node[above] (a2) at (3,3) {$\hat{v}_{2}$};
  \node[above] (a3) at (4.1,3) {$\hat{v}_{3}$};
  \node[below] (a4) at (4.8,3) {$\hat{v}_{4}$};
  \node[below] (a5) at (6,3) {$\hat{v}_{5}$};
  \node[below] (a6) at (7.2,3) {$\hat{v}_{6}$};
  \node[above] (a7) at (7.9,3) {$\hat{v}_{7}$};

  \foreach \from/\to in {1/2,2/4,3/6,6/7,7/8,4/5,1/3,2/3,3/4,3/5,6/5,7/5,8/5,8/9,5/9}
  \draw (\from) -- (\to);
  
    \foreach \from/\to in {d1/d2,d2/d3,d3/d4,d4/d5,d5/d6,d6/d7}
  \draw[dashed] (\from) -- (\to);

\end{tikzpicture}
\end{center}
(a)
\end{minipage}
\quad
\begin{minipage}[b]{0.5\linewidth}
\centering
\begin{center}
\begin{tikzpicture}
  [scale=.68]

\def \gridUp {-0.35}
  
  \draw[step=1cm,color=gray!25,line width=7pt] (0.5,0.5) grid (8.5,1.5);

  \draw (1,1+\gridUp) -- (2,1+\gridUp) node[below,pos=0.5] {$P_{1}$};   
  \draw (1,1.15+\gridUp) -- (3,1.15+\gridUp) node[below,pos=0.85] {$P_{2}$};
  \draw (1,1.3+\gridUp) -- (5,1.3+\gridUp) node[above,pos=0.1] {$P_{3}$};
  \draw (2,1.45+\gridUp) -- (4,1.45+\gridUp) node[above,pos=0.2] {$P_{4}$};
  \draw (3,1.6+\gridUp) -- (8,1.6+\gridUp) node[above,pos=0.85] {$P_{5}$};
  \draw (4,1.75+\gridUp) -- (6,1.75+\gridUp) node[above,pos=0.5] {$P_{6}$};
  \draw (5,1.45+\gridUp) -- (7,1.45+\gridUp) node[below,pos=0.2] {$P_{7}$};
  \draw (6,1.3+\gridUp) -- (8,1.3+\gridUp) node[below,pos=0.2] {$P_{8}$};
  \draw (7,1.15+\gridUp) -- (8,1.15+\gridUp) node[below,pos=0.5] {$P_{9}$}; 
  
  \node[left] at (0.5,1) {$1$};
  
  \foreach \num in {1,2,3,4,5,6,7,8}
  \node[below] at (\num,0.3+\gridUp) {$\num$};

\end{tikzpicture}
\end{center}
(b)
\end{minipage}
\caption{(a) A graph $G$ and its almost-dual $\widehat{G}$ with the vertices $\hat{v}_i$ for $1\leqslant i\leqslant 7$.
(b) A $B_{0}$-EPG representation of $G$ constructed by \autoref{alg: b0 con out}.}
\label{fig:const_out_tri_in_b0}
\end{figure}
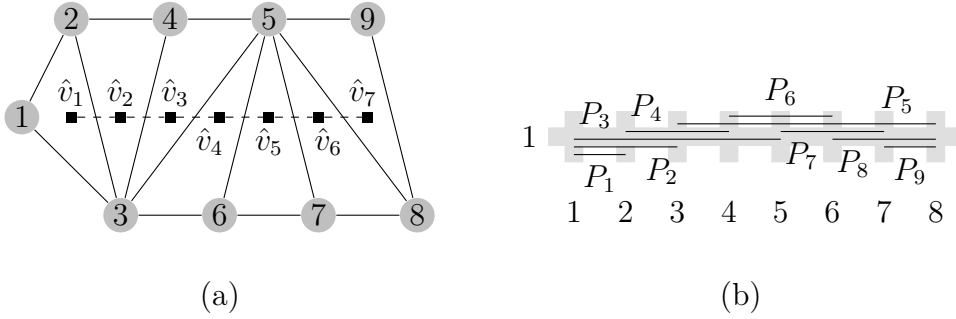

\subsection{Maximal Outerplanar Graphs in \texorpdfstring{$B_{1}$}{B1}}
\label{sec:out_tri_in_b1}
In this section we characterize the maximal outerplanar graphs which are
$B_{1}^{m}$-EPG and $B_1$-EPG, respectively.  
It turns out that for maximal outerplanar graphs $B_0$-EPG and $B_1^m$-EPG 
coincide, see \autoref{cor:OutPlanarInB1m}. Thus, surprisingly, allowing two 
more shapes of paths in the EPG representation  does not 
increase the class
of graphs which can be represented. 
\begin{corollary}\label{cor:OutPlanarInB1m}
Let $G$ be a maximal outerplanar graph. Then $G$ is in $B_{1}^{m}$ if and only
if G is $S_3$-free. 
\end{corollary}
\begin{proof}
	Since  $S_{3}$ is not in $B_{1}^{m}$ (as shown by Cameron, Chaplick, 
	Ho\`{a}ng~\cite{E}),  a graph which is not $S_{3}$-free 
	is not in $B_{1}^{m}$.
	Further, by \autoref{out_tri_in_b0} every $S_3$-free maximal outerplanar 
	graph
	is in~$B_{0}$ and hence also in 
	$B_{1}^{m}$.
\end{proof}

 Since $S_{3}$ is in $B_{1}$ (Golumbic, Lipshteyn, 
Stern~\cite{startpaper}) but  not in $B_{0}$ (Cameron, Chaplick, 
Ho\`{a}ng~\cite{E}) the class of $B_1$-EPG maximal outerplanar graphs  is
strictly larger than the class of $B_1^m$-EPG maximal outerplanar
graphs.
In the   sequel we characterize  the class of $B_1$-EPG
maximal outerplanar graphs.
To this end  we need  the concept of the  \emph{reduced graph} of a maximal 
outerplanar
graph. Recall the definition of central vertices and central edges of $S_3$ from \autoref{def: Sn}.
\begin{definition}
Let $G$ be a maximal outerplanar graph. The \textit{reduced graph
  $\widetilde{G}$} of $G$  is defined in the following way. 
For every copy of $S_{3}$ in $G$,   the central vertices and the central
edges of $S_{3}$ are colored  green. 
Then  all  non-colored vertices and  non-colored edges are removed  from $G$.
The resulting graph is the reduced graph $\widetilde{G}$.
\end{definition}
%
%
The next lemma addresses  the structural relationship of maximal outerplanar graphs and their reduced and almost-dual graphs.

\begin{lemmaMy}
    \label{lem:centralVerticesS3_trianglesGTilde_degree3verticesGhat_bijection}
    Let $G$ be a maximal outerplanar graph with  reduced graph
    $\widetilde{G}$ and  almost-dual graph $\widehat{G}$. Then
    (i)
        the   
        copies
        of~$S_{3}$ in~$G$, 
     (ii)
        the triangles in $\widetilde{G}$ and
     (iii)
        the vertices of degree 3 in $\widehat{G}$
    are in a one-to-one correspondence.
\end{lemmaMy}
\begin{proof}
	First we prove the one-to-one correspondence between (ii) the triangles in 
	$\widetilde{G}$ and (i) the  
	copies of 
	$S_3$ in $G$.
	Clearly each
	copy of $S_3$ in $G$ 
	corresponds to a triangle in $\widetilde{G}$.
	
	Now let $\{v_{1},v_{2},v_{3}\}$ be the vertices of an arbitrary triangle 
	$T$ in
	$\widetilde{G}$. 
	Notice that there are exactly two possibilities for the order in
	which edges are colored: either   all the edges of $T$ were colored at 
	once, or
	the  three edges were colored  at three different times as central edges
	belonging to three different copies of $S_3$ in $G$. 
	
	If all edges of $T$ were colored at once, then 
	$T$ corresponds to  the central vertices of a 
	copy of 
	$S_{3}$ in $G$.
	If $\{v_{1},v_{2}\}$, $\{v_{2},v_{3}\}$, $\{v_{1},v_{3}\}$ were colored at
	three  different times,   there exists vertices $v_4$, $v_5$ and $v_6$ such
	that    $\{v_{1},v_{2},v_{4}\}$, $\{v_{2},v_{3},v_{5}\}$
	and $\{v_{1},v_{3}, v_{6}\}$
	are sets of    central vertices of some  copy of $S_{3}$ in 
	$G$, respectively. 
	Since $G$  is outerplanar the vertices $v_{4}$, $v_{5}$ and $v_{6}$ are
	pairwise distinct and none of them can be adjacent to any of the
	two others. 
	Therefore, the  subgraph of $G$ induced by  the vertices $\{v_{i} | 1 
	\leqslant i
	\leqslant 6\}$ is a copy of  $S_{3}$ in $G$ and hence the vertices of $T$
	form the set of the central vertices of this copy of $S_3$.  
	So (ii) and (i)
	are in a one-to-one correspondence.
	
	To see the bijection between 
	(iii) and (i), recall that  any vertex of degree $3$ in $\widehat{G}$ 
	represents 
	a (triangular)
	inner face 
	of $G$ such that each  edge on its boundary is shared  with 
	the boundary     of another triangular inner face. Thus each   vertex of 
	degree~$3$  
	in     $\widehat{G}$ corresponds to  
	 a copy of $S_3$ 
	in $G$,
	and clearly, also vice-versa.
\end{proof}

Objects corresponding to each other in terms of
the bijections given in the proof of 
\autoref{lem:centralVerticesS3_trianglesGTilde_degree3verticesGhat_bijection} 
are referred to as  \emph{corresponding objects}.

Now we are able to prove the following structural result for $B_{1}$-EPG 
representations of maximal outerplanar graphs.
\begin{corollary}
	\label{triangle_in_reduced_graph_is_claw_clique}
	Let $G$ be a maximal outerplanar graph and $\widetilde{G}$ its reduced 
	graph. 
	In every  $B_{1}$-EPG representation of $G$ (if there is any) 
	the  vertices of any triangle in $\widetilde{G}$
	correspond to a  claw clique. 
\end{corollary}
\begin{proof}
	According to
	\autoref{lem:centralVerticesS3_trianglesGTilde_degree3verticesGhat_bijection}
	the
	vertices of every triangle in $\widetilde{G}$ are the  
	central vertices of a 
	copy
	of $S_{3}$ in $G$. \autoref{s3_claw} implies that this set of central 
	vertices  corresponds to  a claw clique.
\end{proof}

Next we consider some maximal outerplanar graphs which are not $B_1$-EPG.

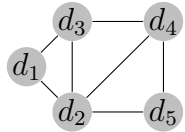
\begin{figure}[tb]
\begin{center}
\begin{tikzpicture}
  [scale=.6]
  
  \node[vertex] (d1) at (2,2) {$d_{1}$};
  \node[vertex] (d2) at (3,1) {$d_{2}$};
  \node[vertex] (d3) at (3,3) {$d_{3}$};
  \node[vertex] (d4) at (5,3) {$d_{4}$};
  \node[vertex] (d5) at (5,1) {$d_{5}$};

  \foreach \from/\to in {d1/d2,d1/d3,d2/d3,d3/d4,d4/d2,d4/d5,d5/d2}
  \draw (\from) -- (\to);
  
\end{tikzpicture}
\end{center}
\caption{The graph $M_{1}$.}
\label{fig:m1}
\end{figure}

\begin{figure}[tb]
\begin{center}
\begin{tikzpicture}
  [scale=1.15]
  
  \node[vertex] (a1) at (3,3) {$a_{1}$} ;
  \node[vertex] (a2) at (2,3) {$a_{2}$};
  \node[vertex] (a3) at (1.5,2) {$a_{3}$};
  \node[vertex] (a4) at (2.5,2) {$a_{4}$};
  
  \node[vertex] (b1) at (10,3) {$b_{1}$};
  \node[vertex] (b2) at (11,3) {$b_{2}$};
  \node[vertex] (b3) at (11.5,2) {$b_{3}$};
  \node[vertex] (b4) at (10.5,2) {$b_{4}$};

  \node[vertex] (c2) at (3.5,2) {$c_{1}$};  
  \node[vertex] (c3) at (4,3) {$c_{2}$};
  \node[vertex] (c4) at (4.5,2) {$c_{3}$};
  \node[vertex] (c5) at (5,3) {$c_{4}$};
  \node[vertex] (c6) at (5.5,2) {$c_{5}$};
  \node[vertex] (c7) at (6,3) {$c_{6}$};
    
  \node[vertex] (c8) at (7,3) {$\textit{ }$};
  \node[above] () at (7,3.1) {$c_{2\ell-6}$};
  
  \node[vertex] (c9) at (7.5,2) {$\textit{ }$};
  \node[below] () at (7.5,1.9) {$c_{2\ell-5}$};

  \node[vertex] (c10) at (8,3) {$\textit{ }$};
  \node[above] () at (8,3.1) {$c_{2\ell-4}$};
  
  \node[vertex] (c11) at (8.5,2) {$\textit{ }$};
  \node[below] () at (8.5,1.9) {$c_{2\ell-3}$};
  
  \node[vertex] (c12) at (9,3) {$\textit{ }$};
  \node[above] () at (9,3.1) {$c_{2\ell-2}$};
  
  \node[vertex] (c13) at (9.5,2) {$\textit{ }$};
  \node[below] () at (9.5,1.9) {$c_{2\ell-1}$};

  \foreach \from/\to in {a1/a2,a2/a3,a3/a4,a4/a1,b1/b2,b2/b3,b3/b4,b4/b1,a1/c2,a1/c3,c2/c3,c3/c4,c3/c5,c4/c5,c10/c11,c10/c12,c11/c12,c12/c13,c12/b1,c13/b1}
  \draw (\from) -- (\to);
  
    \foreach \from/\to in {c5/c6,c5/c7,c6/c7,c7/c8,c8/c9,c8/c10,c9/c10}
  \draw[dotted,thick] (\from) -- (\to);
  
      \foreach \from/\to in {a1/a3,a2/a4,b1/b3,b2/b4}
  \draw[dashed,thick] (\from) -- (\to);

\draw [decorate,decoration={brace,amplitude=10pt},xshift=0pt,yshift=0pt]
(10,1.5) -- (3,1.5) node [black,midway,yshift=-0.6cm] 
{$\ell$ triangles};

\end{tikzpicture}
\end{center}
\caption{The graph $M_{1}^{\ell}$ with $\ell \geqslant 0$ consecutive triangles
  between the vertices $a_{1}$ and $b_{1}$. For $\ell=0$ there are no vertices
  $c$ and the vertices $a_{1}$ and $b_{1}$ coincide. For all values of $\ell$ either  $\{a_{1},a_{3}\}$ or
  $\{a_{2},a_{4}\}$   and  either  $\{b_{1},b_{3}\}$ or  $\{b_{2},b_{4}\}$ is an edge in $M_{1}^{\ell}$.}
\label{fig:m1_n}
\end{figure}
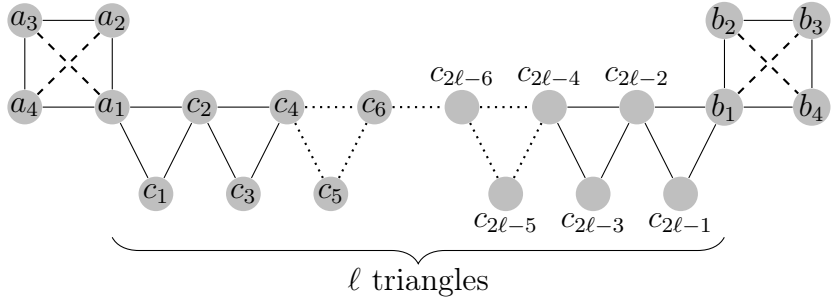

\begin{lemmaMy}
\label{m1_n_non_in_b1}
Let $G$ be a maximal outerplanar graph and $\widetilde{G}$ its reduced graph. 
If the graph $M_{1}$ depicted in \autoref{fig:m1} or the graph $M_{1}^{\ell}$
depicted in \autoref{fig:m1_n} 
is an induced subgraph of $\widetilde{G}$ for some $\ell \geqslant 0$, then $G$ is not in $B_{1}$.
\end{lemmaMy}
\begin{proof}
	Assume by contradiction that  the maximal outerplanar graph $G$ is not 
	$M$-free
	while being  $B_1$-EPG and consider a $B_1$-EPG representation of $G$.
	
	Consider  first the case where $\widetilde{G}$ contains a copy of 
	$M_{1}^{\ell}$ for
	some $\ell \geqslant 0$ (see \autoref{fig:m1_n}). 
	Since exactly one  of $\{a_{1},a_{3}\}$ and $\{a_{2},a_{4}\}$ is an  edge 
	in the
	copy of $M_{1}^{\ell}$ in $\widetilde{G}$  the vertices
	$\{a_{1},a_{2},a_{3},a_{4}\}$ form two triangles in  
	$\widetilde{G}$. 
	According  to \autoref{triangle_in_reduced_graph_is_claw_clique} the 
	vertices of any of these two triangles corresponds to  a claw clique and 
	due to 
	\autoref{s3_claw} in
	each claw clique two of the  three  paths are bent.  
	It is easily observed  that these  two claw cliques must have different 
	central 
	grid points and
	there cannot be a path
	which is bent in both claw cliques. 
	Thus  any path corresponding to some vertex in  
	$\{a_{1},a_{2},a_{3},a_{4}\}$ is
	bent in exactly one of the these two claw cliques. 
	Analogously, any   path corresponding to a vertex in  
	$\{b_{1},b_{2},b_{3},b_{4}\}$ is
	bent in exactly one of the two claw cliques corresponding  to 
	the two triangles 
	formed by $\{b_{1},b_{2},b_{3},b_{4}\}$. 
	
	If  $\ell = 0$, then $a_{1} = b_{1}$ and  
	the
	vertices of  two triangles in $\widetilde{G}$  which share only one  vertex
	correspond to claw cliques with different central grid points.   Thus,  the 
	vertex $a_{1} = b_{1}$ has to be bent in two 
	claw cliques with different central grid points, a contradiction to
	$G$ being $B_1$-EPG.
	Assume now $\ell \geqslant 1$. We set $c_0:=a_1$ and  
	$c_{2\ell}:=b_{1}$
	for notational consistency.  
	Due to  \autoref{triangle_in_reduced_graph_is_claw_clique} the vertices
	$\{c_{0},c_{1},c_{2}\}$ of the first triangle  correspond to   a claw
	clique
	$K_{1}$.  It is easily observed that   the  central grid point of $K_{1}$ 
	is 
	different from the central grid
	points of the claw cliques corresponding  to the two triangles formed by
	$\{a_1,a_2,a_3,a_4\}$. Then, 
	the path corresponding to $a_{1}=c_0$ is bent
	in one of the two claw cliques corresponding to the two triangles formed by
	$\{a_1,a_2,a_3,a_4\}$, so it cannot be bent in $K_1$.  
	Hence, the paths corresponding to $c_{1}$ and $c_{2}$ are bent
	in the claw clique $K_{1}$. Let $K_{i}$ be the claw clique corresponding to
	the vertices  $\{c_{2i-2},c_{2i-1},c_{2i}\}$ of   the $i$-th triangle
	for $1\leqslant i\leqslant \ell$. 
	By induction  we get that the paths corresponding to the vertices 
	$c_{2i-1}$ 
	and $c_{2i}$ are bent in $K_{i}$
	for $1 \leqslant i \leqslant \ell$. 
	Thus the path corresponding to  $c_{2\ell}=b_{1}$ has to be bent 
	in 
	the claw
	clique $K_{\ell}$ as well as in one of
	the claw cliques corresponding to the two triangles formed by
	$\{b_{1},b_{2},b_{3},b_{4}\}$, and this is a contradiction.

	Consider now the case where   $M_{1}$ is an induced subgraph of 
	$\widetilde{G}$
	(see \autoref{fig:m1}). 
	By analogous arguments as  for   $\{a_{1},a_{2},a_{3},a_{4}\}$ the
	path corresponding to  any vertex in
	$\{d_{1},d_{2},d_{3},d_{4}\}$ is  bent in exactly one of the  two 
	claw 
	cliques corresponding to the 
	two triangles formed by $\{d_{1},d_{2},d_{3},d_{4}\}$. But then two  of the 
	paths
	corresponding to $\{d_{2},d_{4},d_{5}\}$ have to be bent in  the  
	claw clique
	corresponding to $\{d_{2},d_{4},d_{5}\}$   while $d_2$ and $d_4$ have to be 
	bent
	also in another claw clique with a different central grid point. 
	This contradicts the definition of   a $B_1$-EPG
	representation. 
\end{proof}

\autoref{m1_n_non_in_b1} describes a class of  maximal outerplanar graphs which are not   $B_{1}$-EPG.
In fact this class contains all  maximal outerplanar graphs which are not
$B_{1}$-EPG as stated in in \autoref{out_tri_in_b1},  the main result in this
section. 
The following definition allows us to simplify notation.

\begin{definition}\label{M-free-ouerplanar}
    A maximal outerplanar graph $G$ is called \emph{$M$-free} if its reduced 
    graph
    $\widetilde{G}$ contains neither  $M_1$ nor $M_1^{\ell}$, for any $\ell \geqslant 0$, as
    an induced subgraph. 
\end{definition}

Before we can consider the construction of  a
 $B_{1}$-EPG representation of a $B_1$-EPG  maximal outerplanar graph $G$,  
we first need the following definitions.

\begin{definition} \label{def: touching neighbored}
Let $G$ be a  maximal outerplanar graph with almost-dual graph  $\widehat{G}$ and
 reduced graph $\widetilde{G}$.
  We denote by \emph{$\widehat{V}_3$} the set of vertices of
 degree $3$ in the almost-dual graph   $\widehat{G}$, so $\widehat{V}_3=\{\hat{v}\in V(\widehat{G})\colon
 \mbox{$\hat{v}$ has degree $3$ in $\widehat{G}$}\}$.

Two distinct triangles in $\widetilde{G}$ which  share an edge are called
\emph{neighbored (to each other)}. Two distinct triangles in $\widetilde{G}$ which share
a vertex, but not an edge, are called \emph{touching (each other)}. 
A \emph{sequence of touching triangles} (\emph{STT}) $T_{1}$, \dots, $T_{k}$ for some $k \in
\mathbb{N}$  is a sequence of triangles $T_{i}$ in $\widetilde{G}$ such that
the triangles $T_{i}$ and $T_{i+1}$ are touching for each $1 \leqslant i
\leqslant k - 1$. 
The \emph{surrounding of a pair of  triangles} $T$ and $T'$  \emph{neighbored
  to each other} in
$\widetilde{G}$ is the set of all triangles $T^{\ast}$ in $\widetilde{G}$, such
that $T$ or  $T'$ can be reached from $T^{\ast}$ over an STT,
  i.e.\ there is an STT $T_{1}$, \dots,
$T_{k}$ with  $T_{1}=T^{\ast}$ and $T_{k} \in \{T, T'\}$.
    
Furthermore we translate the definitions related to triangles in
$\widetilde{G}$ also to the corresponding vertices of degree $3$ in   $\widehat{G}$.  
More precisely  two vertices $\hat{v}$, $\hat{v}'$  of degree $3$ in
$\widehat{G}$ that correspond to two neighbored (touching) triangles in
$\widetilde{G}$
with respect to the bijection of \autoref{lem:centralVerticesS3_trianglesGTilde_degree3verticesGhat_bijection}  are called \emph{neighbored} (\emph{touching}). 
A vertex $\hat{v}^{\ast}$ of degree $3$ in $\widehat{G}$ is said to be
\emph{in the surrounding of two neighbored vertices} $\hat{v}$,  $\hat{v}'$ of degree $3$ in $\widehat{G}$ 
if the corresponding triangle $T_{\hat{v}^{\ast}}$ is in the surrounding of the
neighbored triangles $T_{\hat{v}}$,  $T_{\hat{v}'}$.
        
Finally a \emph{cycle of touching or neighbored triangles} (\emph{CTNT}) is defined as a
sequence of triangles $T_{1}$, \dots, $T_{k}$ in $\widetilde{G}$,  $k \in \mathbb{N}$,   such that for each $1 \leqslant i
\leqslant k-1$, $T_{i}$ and $T_{i+1}$ are either neighbored or touching
triangles,   and also $T_{1}$ and $T_{k}$ are either neighbored or touching
triangles. 
A CTNT $T_{1}$, \dots, $T_{k}$ is called
\emph{reduced} 
if for all $1 \leqslant i < j \leqslant k$ the triangles $T_{i}$ and $T_{j}$
are only neighbored or touching if either $j = i+1$, or $i=1$ and $j=k$. 
In other words a CTNT is reduced iff
triangles of that cycle which are non-consecutive are neither touching each other nor 
neighbored to each other.    
\end{definition}

Next we investigate the structure of $M$-free maximal
outerplanar graphs.
\begin{lemmaMy} \label{lem: b1 max out plan partition vertices deg 3}
Let $G$ be an $M$-free  maximal outerplanar graph 
 with  almost-dual graph $\widehat{G}$  and  reduced graph $\widetilde{G}$. 
Then the vertices in $\widehat{V}_3$ can be partitioned such that
each vertex is either $(A)$ neighbored to exactly one other vertex of degree
$3$ in $\widehat{G}$, or $(B)$ not neighbored, but in the surrounding of
exactly one pair of  neighbored vertices in $\widehat{G}$, 
or $(C)$ not in the surrounding of any pair of neighbored vertices in $\widehat{G}$.
\end{lemmaMy}
\begin{proof}
	Due to 
	\autoref{lem:centralVerticesS3_trianglesGTilde_degree3verticesGhat_bijection}
	we can consider triangles in $\widetilde{G}$ instead of considering 
	vertices of degree $3$ in $\widehat{G}$.
	
	Since  $M_{1}$ is not an induced subgraph of $\widetilde{G}$, every triangle
	in $\widetilde{G}$ is neighbored to at most one other triangle.
	Moreover if $T_1$ and $T_2$ are neighbored triangles in $\widetilde{G}$ 
	with vertex sets $V(T_1)$ and $V(T_2)$,
	respectively, then none of  the vertices in  $V(T_1)\cup V(T_2)$ can  be a
	vertex of some other pair of  neighbored  triangles,  because otherwise
	$M_{1}^{0}$  would be an     induced subgraph  of $\widetilde{G}$.
	To summarize, neighbored triangles in $\widetilde{G}$ appear in pairs of two
	and two different pairs of neighbored triangles  share neither edges nor 
	vertices.
	
	Further, it is easy to see that each triangle of $\widetilde{G}$ is in
	the surrounding of at most one pair of neighbored triangles. Indeed, if 
	a
	triangle $T$ of $\widetilde{G}$ was in the surrounding of two different
	pairs of  neighbored
	triangles $(T_{1},  T_{2})$ and   $(T_{3},T_{4})$, then 
	the graph induced by $T$, $T_{1}$, $T_{2}$, $T_{3}$, $T_{4}$ and the 
	two
	STTs from $T$ to the two pairs  $(T_1,T_2)$,  $(T_3,T_4)$  would contain
	a copy of $M_{1}^{\ell}$  for some $\ell\geqslant 0$, a contradiction.
	Hence, the surroundings of different pairs of neighbored triangles are
	disjoint and the result follows.
\end{proof}
For an $M$-free graph $G$  we denote the sets of vertices of degree $3$ in $\widehat{G}$ 
in the parts $(A)$, $(B)$ and $(C)$
 of the partition of \autoref{lem: b1 max out plan partition vertices deg 3} by $A$, $B$ and $C$, respectively.

Finally we are able to present \autoref{alg: b1 con out assignment} which 
recognizes whether an input maximal outerplanar graph $G$ is $M$-free and
computes a particular  assignment of pairs of  vertices of
$G$ to vertices in 
$\widehat{V}_3$ in the positive case. The particular assignment is
  called $\textAlgo{assigned}$.  It  maps 
$\widehat{V}_3$
to $V(G)\times V(G)$  such that both vertices of the pair
$\textAlgo{assigned}(\hat{v})$  are  central
  vertices of the copy of $S_3$  in $G$ corresponding to $\hat{v}$ in the sense
  of
  \autoref{lem:centralVerticesS3_trianglesGTilde_degree3verticesGhat_bijection}.
  Moreover, these  two vertices are  selected carefully, such that 
  no vertex of $V$ is
  assigned to more than one $\hat{v}\in \widehat{V}_3$.
This  assignment  will be used  to  
construct a  $B_{1}$-EPG
representations of  a $B_1$-EPG maximal outerplanar graph in
\autoref{alg: b1 con out}.
In particular, the path corresponding to a vertex $a \in 
	V(G)$ will bend for representing $\hat{v}$ iff $a \in V(T_{\hat{v}})$ and 
	$a$ is assigned to $\hat{v}$.

\begin{algorithm}[tbp]
	\footnotesize
	\caption{Construct an assignment $\textAlgo{assigned} \colon \widehat{V}_3 
	\to
		V(G)\times V(G)$ 
	}
	\label{alg: b1 con out assignment}
	\Input{A maximal  outerplanar  graph  $G$, its reduced graph 
	$\widetilde{G}$ and its almost-dual $\widehat{G}$}
	\Output{$\textAlgo{assigned} \colon \widehat{V}_3 \to
		V(G)\times V(G)$ or ``$G$ is not $M$-free'' 
	}
	\begin{algorithmic}[1]
		\Procedure{Assignment\_Max\_Outerplanar}{$G$, $\widehat{G}$, 
		$\widetilde{G}$}
		\State Set $\textAlgo{assigned}(\hat{v}) := \emptyset$ and 
		$\textAlgo{label}(\hat{v}) :=
		\textAlgo{unserved}$ for each $\hat{v} \in \widehat{V}_3$ \label{alg: 
		b1 con out
			assignment init label}
		\State Set $\Delta(a):=NULL$ for all $a\in V(G)$
		\State Set $\textAlgo{level}(\hat{v}) := \infty$ for each $\hat{v} \in
		\widehat{V}_3$ and $\textAlgo{level}(NULL):=\infty$
		
		\For{ $\hat{v} \in \widehat{V}_3$} \label{alg: b1 con out assignment 
		step 1 start}
		\If{$\textAlgo{label}(\hat{v}) =  \textAlgo{unserved}$ and $\exists 
		\hat{v}' \in
			\widehat{V}_3$: $\hat{v}$ is neighbored to 
			$\hat{v}'$}\label{if_unserved:1}
		\label{alg: b1 con out assignemnt def neigh tria}
		\State Set $\{a,b\} := V(T_{\hat{v}}) \cap V(T_{\hat{v}'})$ \label{alg: 
		b1 con out assignment def ab}
		\State Set $\{c\} :=   V(T_{\hat{v}})\setminus\{a,b\}$ and $\{d\}
		:=   V(T_{\hat{v}'})\setminus\{a,b\}$ \label{alg: b1 con out
			assignment def cd}
		\If{$\exists x\in  V(T_{\hat{v}})\cup V(T_{\hat{v}'})$ with 
		$\Delta(x)\neq NULL$}\label{if_prior_to_not_M-free:1}
		\State Output  ``$G$ is not $M$-free''  and STOP \label{not_M-free:1}
		\EndIf
		\State Set $\textAlgo{assigned}(\hat{v}) := \{a, c\}$ and  
		$\textAlgo{label}(\hat{v}) :=
		\textAlgo{served}$ \label{serve_neighbor_1}
		\State Set $\textAlgo{assigned}(\hat{v}') := \{b,d\}$ and 
		$\textAlgo{label}(\hat{v}') :=
		\textAlgo{served}$\label{serve_neighbor_2}
		\State Set $\textAlgo{level}(\hat{v}) := 0$ and 
		$\textAlgo{level}(\hat{v}') := 0$
		\State Set  $\Delta(a):=\hat{v}$, $\Delta(c):=\hat{v}$,
		$\Delta(b):=\hat{v}'$, $\Delta(d):=\hat{v}'$
		\label{alg: b1 con out assignment set delta}
		\EndIf
		\EndFor \label{alg: b1 con out assignment step 1 stop}
		
		\While{ $\exists \hat{v} \in \widehat{V}_3$: $\textAlgo{label}(\hat{v}) 
		=  
		\textAlgo{unserved}$ } \label{alg: b1 con out assignment step 3 start}
		\longWhile{ $\exists \hat{v} \in \widehat{V}_3$: 
		$\textAlgo{label}(\hat{v}) =
			\textAlgo{unserved}$ and \\
			$\min \{ \textAlgo{level}(\Delta(x))\colon x\in 
			V(T_{\hat{v}})\}\neq \infty 
			$}\label{alg: b1 con out assignment step 2 start}
		\State Set $a:=\argmin \{ \textAlgo{level}(\Delta(x)) \colon x\in
		V(T_{\hat{v}})\}$ and  $\hat{v}'' := \Delta(a)$\label{alg: b1 con out 
		assignment choice min level}
		\If{$\exists x\in V(T_{\hat{v}})\setminus \{a\}$ with $\Delta(x)\neq 
		NULL$}\label{precede_not_M-free:2}
		\State Output ``$G$ is not $M$-free'' and STOP  \label{not_M-free:2}
		\EndIf
		\State Set $\textAlgo{assigned}(\hat{v}) := 
		V(T_{\hat{v}})\setminus\{a\}$
		and $\textAlgo{label}(\hat{v}) := \textAlgo{served}$ 
		\label{a_not_assigned}
		\State Set $\textAlgo{level}(\hat{v}) := \textAlgo{level}(\hat{v}'') + 
		1$ \label{levels}
		\ForInline{$x\in V(T_{\hat{v}})\setminus \{a\}$}{ $\Delta(x):=\hat{v}$}
		\EndlongWhile \label{alg: b1 con out assignment step 2 stop}
		\If{$\exists \hat{v} \in \widehat{V}_3$: $\textAlgo{label}(\hat{v}) =  
		\textAlgo{unserved}$} 
		\label{alg: b1 con out assignment choice random vertex}
		\State Let $a \in V(T_{\hat{v}})$ be a random vertex of $V(T_{\hat{v}})$
		\State Set $\textAlgo{assigned}(\hat{v}) := 
		V(T_{\hat{v}})\setminus\{a\}$
		and  $\textAlgo{label}(\hat{v}) := 
		\textAlgo{served}$\label{serve_vertex_in_partition_c}
		\State Set $\textAlgo{level}(\hat{v}) := 0$
		\ForInline{$x\in V(T_{\hat{v}})\setminus\{a\}$}
		{$\Delta(x):=\hat{v}$}
		\EndIf\label{alg: b1 con out assignment choice random vertex end}
		\EndWhile \label{alg: b1 con out assignment step 3 stop}     
		
		\State \Return $\textAlgo{assigned}$
		\EndProcedure
	\end{algorithmic}
\end{algorithm}


 Next we address two  important
arrays used by the   algorithm: $\textAlgo{level}(\hat{v})$, for 
$\hat{v}\in
\widehat{V}_3$, and $\Delta(a)$, for $a\in V(G)$. 
Their meaning is as follows. 
$\Delta(a)=\hat{v}\in \widehat{V}_3$ iff $a$ belongs to  
$\textAlgo{assigned}(\hat{v})$.
The meaning of $\textAlgo{level}(\hat{v})$ is a bit
more complicated: 
in the case that the
input graph $G$ is $M$-free 
for every pair  $(\hat{v},\hat{v}')$ of neighbored vertices
in $\widehat{V}_3$ 
we have $\textAlgo{level}(\hat{v})=\textAlgo{level}(\hat{v}')=0$. 
For a vertex $\hat{v}$  in the
surrounding of some pair of neighbored vertices  $(\hat{w},\hat{w}')$ (i.e.\
$\hat{v} \in B$ according to
\autoref{lem: b1 max out plan partition vertices deg 3}) the
quantity
$\textAlgo{level}(\hat{v})=:k-1$ determines  the length
of a  STT $\hat{v}_1, \ldots, \hat{v}_k$ in $\widehat{V}_3$   starting at    $\hat{w}$ or $\hat{w}'$ and
ending at $\hat{v}=\hat{v}_k$.  Thus  in the execution of line~\ref{levels} of \autoref{alg: b1 con out assignment}  the
vertex $\hat{v}''$ coincides with $\hat{v}_{i-1}$ when $\hat{v}$ coincides with
$\hat{v}_i$ for every $i\in \{2,\ldots,k\}$.  

For each vertex $\hat{v}$ which does not belong to
the surrounding of some pair of neighbored vertices in $\widehat{V}_3$ 
(i.e.\  $\hat{v} \in C$ according to \autoref{lem: b1 max out plan partition vertices deg 3}) 
the following holds. 
There is a  STT $\hat{v}_1, \ldots, \hat{v}_\ell$ in $\widehat{V}_3$ such that 
$\textAlgo{level}(\hat{v}_1)=0$ holds, $\hat{v} =  \hat{v}_\ell$ and in the 
execution of line~\ref{levels} of the algorithm  the
vertex $\hat{v}''$ coincides with $\hat{v}_{j-1}$ when $\hat{v}$ coincides with
$\hat{v}_j$ for every $j\in \{2,\ldots,\ell\}$.  

 The following  
\autoref{lem: assignemnt well
  definied and properties} states some  properties of the assignment
constructed in \autoref{alg: b1 con out assignment}.

\begin{lemmaMy}
    \label{lem: assignemnt well definied and properties}
    Let $G$ be a maximal outerplanar graph with  almost-dual
    $\widehat{G}$ and  reduced graph $\widetilde{G}$.
    Then \autoref{alg: b1 con out assignment} is well-defined. 
    At termination 
    it either outputs ``$G$ is not $M$-free''
    or it returns an assignment that assigns two vertices $x_{\hat{v}}$, $y_{\hat{v}}$  of $G$ to every vertex $\hat{v}$  in $\widehat{V}_3$ such that
        (1)
        $x_{\hat{v}}$, $y_{\hat{v}}$ are central vertices of the copy of 
        $S_{3}$ in $G$
        corresponding to $\hat{v}$, and 
        (2)
        no vertex of $G$ is assigned to more than one vertex $\hat{v}$ of $\widehat{G}$.
\end{lemmaMy}
\begin{proof}
	We first  show that  \autoref{alg: b1 con out assignment} is
	well-defined, i.e.\  that it  can be executed as described 
	and that
	it terminates.
	Observe that by the definition of neighbored vertices $|V(T_{\hat{v}}) \cap
	V(T_{\hat{v}'})| = 2$ holds, so $a$, $b$, $c$ and $d$ can be defined as 
	described 
	in line~\ref{alg: b1 con out assignment def ab} and~\ref{alg: b1 con out
		assignment def cd}.
	Furthermore, $level(\Delta(x)) < \infty$ if and only if $\Delta(x) \neq 
	NULL$, 
	and $\Delta(x) \in \widehat{V}_3$ holds for all $x \in V(G)$.
	So $\hat{v}'' \in \widehat{V}_3$ holds in 
	line~\ref{alg: b1 con out assignment choice min level} and the sum in 
	line~\ref{levels}
	is well-defined.
	
	Next we show that the algorithm always terminates correctly. 
	This clearly holds    if \autoref{alg: b1 con out assignment} outputs ``$G$ 
	is not $M$-free''.  
	Assume now that the algorithm does not output ``$G$ is not $M$-free''.
	The algorithm starts with all  vertices of 
	$\widehat{V}_3$ being labeled \emph{unserved} and having  no assigned
	vertices  (line~\ref{alg: b1 con out assignment init label}). Then the
	algorithm  iteratively {\sl serves the vertices}  $\hat{v}$ of
	$\widehat{V}_3$,
	i.e.\ it  labels them  \emph{served} and assigns them  the vertices 
	$x_{\hat{v}}$, $y_{\hat{v}}$  from
	$G$.     Since the while loop in lines~\ref{alg: b1 con out assignment step 
	3
		start} to~\ref{alg: b1 con out assignment step 3 stop} is executed as 
	long
	as there are unserved vertices  in $\widehat{V}_3$, the algorithm 
	terminates after a finite number of
	steps. Moreover,  it has  assigned a
	pair of vertices $x_{\hat{v}}$, $y_{\hat{v}}$ from $G$ to every vertex  
	$\hat{v}$ in $\widehat{V}_3$ at termination.

	It is easy to see that the vertices $x_{\hat{v}}$, $y_{\hat{v}}$ assigned
	to $\hat{v} \in \widehat{V}_3$    fulfill property~(1) by construction.
	In order to see that also (2) is fulfilled,  observe that $\Delta(x) = 
	\hat{v}$ iff 
	$x \in \textAlgo{assigned}(\hat{v})$ for all $x$ and all $\hat{v}$. 
	Moreover, both the
	equality and the inclusion hold      immediately after the vertex $\hat{v}$ 
	is served.
	Finally, $\Delta(x) = NULL$ holds whenever a vertex $x$ of $G$ 
	is assigned to a vertex $\hat{v}$ in $\widehat{V}_3$
	in lines~\ref{serve_neighbor_1},~\ref{serve_neighbor_2}
	or~\ref{a_not_assigned}. Otherwise  the condition of the if statement in 
	lines~\ref{if_prior_to_not_M-free:1} or~\ref{precede_not_M-free:2} 
	would have been fulfilled and the algorithm would have already terminated 
	with
	``$G$ is not $M$-free''.
	Hence, a vertex $x$ of $G$ can only be assigned to 
	at most one vertex $\hat{v}$ of $\widehat{V}_3$.
\end{proof}

In order to examine \autoref{alg: b1 con out assignment} in more detail, we 
refer to the execution of lines~\ref{alg: b1 con out assignment step 1
	start} to~\ref{alg: b1 con out assignment step 1 stop} as step one, 
the first execution of lines~\ref{alg: b1 con out assignment step 2 start} to 
\ref{alg: b1 con out assignment step 2 stop} as step two and to the remaining
executions of lines~\ref{alg: b1 con out assignment step 3 start} 
to~\ref{alg: b1 con out assignment step 3 stop} as step three.
We start our investigation by considering step one in the next two lemmata.

\begin{lemmaMy}
	\label{lem: assignemnt one_outputOrAssignment}
	Let $G$ be a maximal outerplanar graph with  almost-dual
	$\widehat{G}$ and  reduced graph $\widetilde{G}$.
	Then in step one \autoref{alg: b1 con out assignment} either 
	outputs ``$G$ is not $M$-free''
	or it serves exactly the vertices in $\widehat{V}_3$ which are neighbored.
\end{lemmaMy}
\begin{proof}
	If \autoref{alg: b1 con out assignment}  
	does not output ``$G$ is not $M$-free'' in step one, 
	then it loops through all vertices of
	$\widehat{V}_3$ and serves all vertices which are unserved and
	neighbored to another vertex,
	so clearly   
	exactly neighbored vertices are served in step one. Moreover, if the 
	algorithm
	outputs  ``$G$ is not $M$-free'' in step one, it does so in
	line~\ref{not_M-free:1} and at this point there are still unserved vertices
	which belong to a pair of neighbored vertices (see
	line~\ref{if_unserved:1}). Thus in this case the algorithm does not serve 
	all
	neighbored vertices in~$\widehat{V}_3$. 
\end{proof}

\begin{lemmaMy}
	\label{lem: assignemnt one_outputCorrect}
	Let $G$ be a maximal outerplanar graph with  almost-dual
	$\widehat{G}$ and  reduced graph $\widetilde{G}$.
	Then \autoref{alg: b1 con out assignment} 
	outputs ``$G$ is not $M$-free''
	in line~\ref{not_M-free:1} iff $\widetilde{G}$ contains a copy of $M_1$ or 
	$M_1^0$.
\end{lemmaMy}
\begin{proof}
	Consider first  the case that   $\widetilde{G}$ contains a copy of $M_1$ 
	(see \autoref{fig:m1}).   Let $\hat{v}_1$, $\hat{v}_2$ and $\hat{v}_3$ be 
	the vertices in 
	$\widehat{V}_3$ that induce an $M_1$ in $\widetilde{G}$. We  show that  
	\autoref{alg: b1 con out assignment}     outputs ``$G$ is not $M$-free''
	in line~\ref{not_M-free:1} in this case.
	Assume by contradiction that this is not the case. Thus,    the algorithm 
	completes  the for loop which
	starts at line~\ref{alg: b1 con out assignment step 1 start}. 
	Assume that 
	$\hat{v}_1$ is served by the algorithm before $\hat{v}_2$ and  $\hat{v}_3$.
	The other cases ($\hat{v}_2$ or $\hat{v}_3$ are served first) 
	can be handled analogously. 
	Consider the
	execution of line~\ref{if_unserved:1} with $\hat{v}=\hat{v}_1$ or 
	$\hat{v}'=\hat{v}_1$ prior
	to serving $\hat{v}_1$. There
	are two cases: either $\{\hat{v}, \hat{v}'\} \neq \{\hat{v}_1,\hat{v}_2\}$ 
	or $\{\hat{v}, \hat{v}'\} = \{\hat{v}_1,\hat{v}_2\}$. 
	In the first case the condition in line~\ref{if_prior_to_not_M-free:1}
	is fulfilled at the latest by  the execution of this line with
	$\hat{v}=\hat{v}_2$ or $\hat{v}'=\hat{v}_2$ and the algorithm would stop 
	with   ``$G$ is not $M$-free'' in line~\ref{not_M-free:1}, a contradiction. 
	In the latter case the condition in line~\ref{if_prior_to_not_M-free:1}
	is fulfilled at the latest by  the execution of this line with
	$\hat{v}=\hat{v}_3$  or $\hat{v}'=\hat{v}_3$  and the algorithm would stop 
	with   ``$G$ is not $M$-free'' in line~\ref{not_M-free:1}, again a 
	contradiction.

	Consider now the case   that  $\widetilde{G}$ contains a copy of $M_1^0$ 
	but no
	copy of $M_1$.  Let $\hat{v}_1$, $\hat{v}_2$, $\hat{v}_3$ and $\hat{v}_4$ 
	be vertices of $\widehat{V}_3$ that induce an $M_1^0$ 
	such that $\hat{v}_1$ and $\hat{v}_2$ are neighbored to each other 
	and  $\hat{v}_3$ and $\hat{v}_4$ are neighbored to each other.
	We  show by contradiction  that 
	\autoref{alg: b1 con out assignment} 
	outputs ``$G$ is not $M$-free'' in line~\ref{not_M-free:1}. 
	
	Notice that since $\widetilde{G}$ does not contain a copy of $M_1$, 
	each neighbored triangle is neighbored to exactly one other triangle.
	Assume  w.l.o.g.\  that during the algorithm
	$\hat{v}_1$ and $\hat{v}_2$ coincide with $\hat{v}$ and 
	$\hat{v}'$ 
	in line~\ref{alg: b1 con out assignemnt def neigh tria} and then, later,  
	$\hat{v}_3$ and $\hat{v}_4$ 
	coincide with $\hat{v}$ and $\hat{v}'$ 
	at that line. 
	At that point one of the vertices of $G$ contained 
	in the triangles corresponding to $\hat{v}_1$ and $\hat{v}_2$ 
	is among $x \in V(T_{\hat{v}}) \cup V(T_{\hat{v}'})$. 
	So with the same arguments as above \autoref{alg: b1 con out assignment}  
	outputs 
	``$G$ is not $M$-free'' in line~\ref{not_M-free:1}.
	
	We complete the proof by showing  that $\widetilde{G}$ contains a copy of
	$M_1$ or $M_1^0$ if \autoref{alg: b1 con out assignment} 
	outputs ``$G$ is not $M$-free''     in line~\ref{not_M-free:1}.
	
	Let $\hat{v}_1$ and $\hat{v}_2$ be
	the vertices $\hat{v}$ and $\hat{v}'$ 
	in line~\ref{alg: b1 con out assignemnt def neigh tria} at the 
	last execution of this line before the algorithm 
	stops. 
	Clearly $\hat{v}_1$ and $\hat{v}_2$ are neighbored.
	Since \autoref{alg: b1 con out assignment} stops,  there has to be a vertex 
	$x \in V(T_{\hat{v}}) 
	\cup V(T_{\hat{v}'})$ 
	such that~$\Delta(x)$ was set to either 
	$\hat{v}$ or $\hat{v}'$ in line~\ref{alg: b1 con out assignment set delta} 
	in a previous iteration. 
	Let $v_{3}$ and $v_{4}$ be the vertices 
	$\hat{v}$ or $\hat{v}'$  at that previous iteration.
	Clearly, $x$ is contained in $T_{\hat{v}_1}$ or $T_{\hat{v_2}}$ 
	and also contained in $T_{\hat{v_3}}$ and~$T_{\hat{v_4}}$, 
	so at least one of $v_{1}$ and $v_{2}$ 
	is neighbored to or touching at least one of~$v_{3}$ and~$v_{4}$.
	If the later vertices  are neighbored, then $\widetilde{G}$ contains a copy 
	of $M_1$.
	Otherwise $\widetilde{G}$ contains a copy of~$M_1^0$.
\end{proof}

Notice that \autoref{lem: assignemnt one_outputOrAssignment}
and \autoref{lem: assignemnt one_outputCorrect} 
imply
that if $G$ is $M$-free, then in step one 
\autoref{alg: b1 con out assignment} serves 
exactly the vertices in $A$ (see \autoref{lem: b1 max out plan partition 
vertices deg 3}).
In the next two lemmata we consider step two 
of 
the algorithm 
in more detail. 
\begin{lemmaMy}\label{servedstep2_implies_surrounding_notneighbored}
	Let $G$ be a maximal outerplanar graph with  almost-dual
	$\widehat{G}$ and  reduced graph $\widetilde{G}$. Assume that  
	\autoref{alg: 
		b1 con
		out assignment} serves a vertex $\hat{v}$ in $\widehat{V}_3$ in step 
	two.
	Then $\hat{v}$  is not neighbored, but it is in the surrounding of   
	neighbored vertices. 
\end{lemmaMy}
\begin{proof}
	Clearly, $\hat{v}$ is not neighbored, otherwise it would have been already
	served at step one (see \autoref{lem: assignemnt one_outputOrAssignment}). 
	We show that $\hat{v}$ is in the surrounding of neighbored     vertices.
	Consider the vertex  $\hat{v}''$ in line~\ref{alg: b1 con out assignment 
		choice
		min level}  when $\hat{v}$ is served in  step two. 
	Then $\hat{v}$ and  $\hat{v}''$ are touching. If $\hat{v}''$ is a
	neighbored vertex, then $\hat{v}$ is in the surrounding of $\hat{v}''$ and
	we are done. 
	Otherwise, 
	$\hat{v}''$ has been served in step two and   
	the above argument can be  inductively repeated  with $\hat{v}''$ in the 
	role of $\hat{v}$.
	Since   $level(\hat{v}'') <      level(\hat{v})$, a neighbored vertex,
	i.e.\ a vertex with level equal to $0$, is reached after a finite number of 
	repetitions.
	Therefore~$\hat{v}$ is in the
	surrounding of neighbored vertices, but not neighbored.
\end{proof}

\begin{lemmaMy}
	\label{lem: assignemnt two_outputOrAssignment}
	Let $G$ be a maximal outerplanar graph with  almost-dual
	$\widehat{G}$ and  reduced graph $\widetilde{G}$. Assume that 
	\autoref{alg: b1 con out assignment} has not terminated with ``$G$ is not 
	$M$-free'' in 
	step
	one. 
	Then in step two it  either 
	outputs ``$G$ is not $M$-free''  
	or it serves exactly the vertices in $\widehat{V}_3$ which are not 
	neighbored, but 
	in the surrounding 
	of neighbored vertices.
\end{lemmaMy}
\begin{proof}
	Assume first that \autoref{alg: b1 con
		out assignment} does not output  ``$G$ is not $M$-free'' in step two. 
	Thus  the algorithm has completed step two, i.e.\ it has completed 
	the
	first     run of the while loop in lines~\ref{alg: b1 con out assignment 
		step 2 start} to~\ref{alg: b1 con out assignment step 2 stop}. 
	\autoref{servedstep2_implies_surrounding_notneighbored} implies that each
	vertex served in step two is not neighbored, but in the surrounding of 
	neighbored vertices. 
	%
	Now we consider a vertex~$\hat{v}$, that 
	is not neighbored, but in the 
	surrounding of 
	neighbored vertices. Assume by contradiction   that it  is not  served in 
	step two.
	Clearly, $\hat{v}$ has not been
	served in step one,  due to     \autoref{lem: assignemnt 
	one_outputOrAssignment}.
	Thus $\hat{v}$ is not served yet at the end of step two.
	Consider  an STT $T_{\hat{v}_{1}}$, $\dots$,
	$T_{\hat{v}_{k}}$ such that
	$\hat{v} = \hat{v}_{k}$ and $\hat{v}_{1}$ belongs to  a neighbored pair of
	vertices. Obviously,     $\hat{v}_{1}$ was already served at the beginning 
	of step
	two. 
	Observe that  $k \geqslant i
	\geqslant 2$ holds  for  the smallest  index   $i\in \{1,2,\ldots,k\}$ such 
	that  the vertex
	$\hat{v}_{i}$ is not served at the end of step two. 
	Thus,  at the end of step two $\hat{v}_{i-1}$ is a served 
	vertex, $\hat{v}_{i-1}$
	and $\hat{v}_{i}$ are touching and $\hat{v}_{i}$ is not served, implying
	that the while-condition in line~\ref{alg: b1 con out assignment step 2
		start} is fulfilled, a contradiction to the completion of  step two.
	
	On the other hand,   if \autoref{alg: b1 con
		out assignment}  outputs  ``$G$ is not $M$-free'' in step two, this is
	because the if condition in line~\ref{precede_not_M-free:2} is fulfilled, 
	meaning that there is
	an unserved vertex $\hat{v}$ touching a served vertex $\hat{v}''$ (see
	line~\ref{alg: b1 con out assignment choice min level}).
	According to  \autoref{lem: assignemnt one_outputOrAssignment} and 
	\autoref{servedstep2_implies_surrounding_notneighbored}
	$\hat{v}''$ is neighbored 
	or lies in the surrounding of a neighbored vertex. 
	Consequently, vertex $\hat{v}$ 
	also lies in the surrounding of a neighbored vertex, while being unserved.
\end{proof}


Under the premise that \autoref{alg: b1 con out assignment} correctly 
recognizes whether a graph is $M$-free, 
\autoref{lem: assignemnt two_outputOrAssignment}
implies
that if $G$ is $M$-free, then in step two 
\autoref{alg: b1 con out assignment} serves 
exactly the vertices in $B$ (see \autoref{lem: b1 max out plan partition 
vertices deg 3}).
Furthermore, under this premise, it follows from 
\autoref{lem: assignemnt well definied and properties},  
\autoref{lem: assignemnt one_outputOrAssignment} 
and  \autoref{lem: assignemnt two_outputOrAssignment}
that 
if~$G$ is $M$-free, then in step three 
\autoref{alg: b1 con out assignment} serves 
exactly the vertices in~$C$ (see \autoref{lem: b1 max out plan partition 
vertices deg 3}).
The next two lemmata are used to establish this premise.

\begin{lemmaMy} \label{lem: b1 max out plan no large reduced cycle}
	Let $G$ be an $M$-free  maximal outerplanar graph,
	with  almost-dual graph $\widehat{G}$  and   reduced graph  $\widetilde{G}$.
	Then $\widetilde{G}$ does not contain a reduced CTNT $T_{1}$, \dots, 
	$T_{k}$, for any  $k \geqslant 4$. Moreover,
	if a  reduced CTNT of length~$3$ is contained in $\widehat{G}$,   then
	all the triangles in that CTNT share a common vertex. 
\end{lemmaMy}
\begin{proof}
	We  prove \autoref{lem: b1 max out plan no large reduced cycle} by
	contradiction. Assume that $C = (T_{1}, \dots, T_{k})$ is a 
	reduced CTNT
	in $\widetilde{G}$ for some $k\geqslant 4$.
	Due to \autoref{lem: b1 max out plan partition vertices deg 3}
	touching triangles can be in the surrounding of at most one pair of 
	neighbored triangles. 
	Therefore at most two consecutive triangles can be neighbored in a CTNT,
	all other consecutive triangles  are touching.
	For each $1 \leqslant i \leqslant k-1$ let $v_{i}$ be a vertex shared
	by $T_{i}$ and $T_{i+1}$ and  let 	$v_{k}$
	be a vertex shared by $T_{k}$ and $T_{1}$. 
	Clearly, $v_{1}$, \dots, $v_{k}$ are pairwise distinct (otherwise 
	$C$  would not be reduced) and  form a cycle $C'$ of length
	$k\geqslant 4$ in $\widetilde{G}$ and hence also in $G$. 
	Since $G$ is maximal outerplanar, all  the vertices of the triangles
	$T_{1}$, \dots, $T_{k}$ which are not in $C'$ lie on the outside of
	$C'$ and all faces of $G$  within $C'$  are triangles. 
	Now we distinguish two cases: (i)  there are two consecutive neighbored 
	triangles in $C$ and (ii)  there are no two consecutive 
	neighbored  triangles in~$C$.
	
	In   Case (i)  let $i^{\ast}$ such that $v_{i^{\ast}}$ is  the vertex of 
	$C'$ which is shared by 
	the 
	neighbored triangles and  let $v_{i^{\ast}}'$ be the other vertex shared by
	the neighbored triangles. Notice that   $v_{i^{\ast}}'$ is not in $C'$. 
	It is easy to see that if we replace $v_{i^{\ast}}$ by $v_{i^{\ast}}'$ 
	in~$C'$ we get another cycle $C''$ in $G$ such that $v_{i^{\ast}}$ is in 
	the 
	interior of the
	cycle~$C''$, which is a contradiction to the outerplanarity of $G$. 
	
	In  Case (ii)  all faces of $G$  within the cycle $C'$
	are triangles and all of them are center triangles of some copy of $S_3$ in 
	$G$,
	thus all of them are present also in $\widetilde{G}$. But this implies that 
	a
	copy  of $M_{1}$ is contained in $\widetilde{G}$, a contradiction.
	
	Thus 
	there are no	reduced CTNT of length $k$ with $k\geqslant 4$ in 
	$\widetilde{G}$.
	\medskip
	
	Now consider a reduced CTNT $C=(T_1, T_2, T_3)$ of length $3$ in 
	$\widetilde{G}$.
	We distinguish again the  cases (i) and (ii).
	
	In Case (i) let w.l.o.g.\ $T_1$ and $T_2$ 
	be two  neighbored triangles in $C$. 
	Then~$T_3$ must share a vertex $a$ with  $T_1$ and a vertex
	$b$ with  $T_2$. 
	If $a\neq b$ and  $|\{a,b\} \cap V(T_1)\cap
	V(T_2)| \in \{0,2\}$, then  $G$ would not be outerplanar, a contradiction.
	If $a\neq b$ and  $|\{a,b\} \cap V(T_1)\cap
	V(T_2)| = 1$, then  $\widetilde{G}$ would contain a copy of~$M_1$, which 
	contradicts $G$ being $M$-free. 
	So $a = b$ is shared by all triangles of $C$.
	
	In Case (ii), if there are no neighbored triangles in $C$, then $T_1$, 
	$T_2$ and~$T_3$ are
	pairwise touching. Again, if they would not all three touch at a common
	vertex, then $\widetilde{G}$ would contain a copy of
	$M_1$, contradicting $G$ being  $M$-free. 
\end{proof}

\begin{lemmaMy}
	\label{lem: assignemnt twothree_outputCorrect}
	Let $G$ be a maximal outerplanar graph with  almost-dual
	$\widehat{G}$ and  reduced graph $\widetilde{G}$.
	Then \autoref{alg: b1 con out assignment} 
	outputs ``$G$ is not $M$-free''
	in line~\ref{not_M-free:2} iff $\widetilde{G}$ contains a copy of 
	$M_1^\ell$ for some $\ell \geqslant 1$ and no copy of $M_1$ and $M_1^0$.
\end{lemmaMy}
\begin{proof}
	First, let $G$ be such that $\widetilde{G}$ contains a copy of $M_1^\ell$
	for some $\ell \geqslant 1$ and no copy of $M_1$ and $M_1^0$  (see
	\autoref{fig:m1} and  \autoref{fig:m1_n}).  
	We show that \autoref{alg: b1 con out assignment} 
	outputs ``$G$ is not $M$-free''
	in line~\ref{not_M-free:2}.
	Let us denote by  $\hat{a}_1$ and $\hat{a}_2$ ($\hat{b}_1$ and
	$\hat{b}_2$) the vertices in $\widehat{V}_3$
	corresponding to the  triangles with vertices $\{a_1,a_2,a_3,a_4\}$
	($\{b_1,b_2,b_3,b_4\}$), and by  $\hat{v}_i$ the vertex in $\widehat{V}_3$
	corresponding to the triangle with vertices
	$\{c_{2i-2},c_{2i-1},c_{2i}\}$,    $1\leqslant i \leqslant\ell$,  
	in  an arbitrary but fixed 
	copy of 
	$M_1^\ell$ in $\widetilde{G}$.
	Due to \autoref{lem: assignemnt one_outputOrAssignment} and \autoref{lem: 
		assignemnt one_outputCorrect} 
	the pairs of neighbored vertices $\hat{a}_1$, $\hat{a}_2$ and $\hat{b}_1$,
	$\hat{b}_2$   have been served in step
	one and    the
	algorithm has not terminated  with ``$G$ is not $M$-free'' in step one.
	According to 
	\autoref{lem: assignemnt one_outputOrAssignment} and 
	\autoref{lem: assignemnt two_outputOrAssignment}  \autoref{alg: b1 con out 
	assignment}  either 
	outputs ``$G$ is not $M$-free''  in step two
	or  all vertices~$\hat{v}_i$ are served at the end of step two. 
	
	%
	
	There is noting to show    if the algorithm  stops with ``$G$ is not 
	$M$-free''in step two.
	So assume w.l.o.g. that all vertices $\hat{v}_i$, $1\leqslant i
	\leqslant\ell$, are served at the end of
	step 2 and consider the moment when  line~\ref{a_not_assigned} is 
	executed for
	the last such  vertex $\hat{v}_i$ with 
	$\hat{v}=\hat{v}_i$.  It can be shown inductively 
	that the  vertices  $c_{2i-2}$  and  $c_{2i}$ of
	$V(T_{\hat{v}_i})$ have already been assigned to other vertices. %
	Hence $\Delta(c_{2i-2})\neq NULL$ and $\Delta(c_{2i})\neq     NULL$,
	implying  that the if 
	condition in line~\ref{precede_not_M-free:2}
	is fulfilled. Thus the algorithm stops with  ``$G$ is not $M$-free'' in
	line~\ref{not_M-free:2}.

	Next we assume   that 
	\autoref{alg: b1 con out assignment} 
	outputs ``$G$ is not $M$-free''
	in line~\ref{not_M-free:2}.
	If $\widetilde{G}$ would contain a copy of $M_1$ or $M_1^0$, 
	then
	\autoref{alg: b1 con out assignment} 
	would output ``$G$ is not $M$-free''
	already in line~\ref{not_M-free:1} (according to   \autoref{lem: assignemnt
		one_outputCorrect}).
	We complete the proof by showing that 
	$\widetilde{G}$ contains a copy of $M_1^\ell$ for some $\ell \geqslant 1$.

	Let $\hat{v}$ be the vertex for which the if condition in 
	line~\ref{precede_not_M-free:2}
	is fulfilled right before termination  with  ``$G$ is not $M$-free'' 
	in
	line~\ref{not_M-free:2}. Consider  two vertices $a \neq x \in 
	V(T_{\hat{v}})$,
	with $\Delta(a)\neq NULL $ and  $\Delta(x)\neq NULL $ (see
	lines~\ref{alg: b1 con out assignment choice min level} 
	and~\ref{precede_not_M-free:2}). Thus
	$a$ and $x$ have  been assigned to some already served vertices, say
	$\hat{a}$ and $\hat{x}$ in $\widehat{V}_3$. 
	Notice that $\hat{a}\neq \hat{x}$,  otherwise
	$\hat{a}=\hat{x}$  and $\hat{v}$ would be neighbored to each
	other and  $\hat{v}$ would have already been served in step one due to
	\autoref{lem: assignemnt one_outputOrAssignment}. 
	Notice, moreover, that  $\hat{a}$ and $\hat{v}$ can not be 
	neighbored due to the outerplanarity of~$G$.
	Furthermore, 
	$\hat{a}$ and $\hat{x}$ do not touch, otherwise $\widetilde{G}$ would
	contain a copy of~$M_1$. Hence $T_{\hat{a}}$ and $T_{\hat{x}}$ do not share 
	a vertex.
	
	The vertices $\hat{a}$ and $\hat{x}$ have been served in step one or two, 
	and
	hence each of them is neighbored or in the surrounding of
	some neighbored vertices (see \autoref{lem: assignemnt
		one_outputOrAssignment} and
	\autoref{servedstep2_implies_surrounding_notneighbored}).
	Hence there exist two STT $S_1$ and $S_2$ such that 
	(i) $\hat{a}$ is the vertex preceding $\hat{v}$ in $S_1$, 
	(ii) $\hat{x}$ the vertex preceding $\hat{v}$ in $S_2$,
	and both $S_1$ and $S_2$
	(iii) end in $\hat{v}$, 
	(iv) start in a neighbored triangle and
	(v) contain as few triangles as possible.
	If  the only vertex in $S_1$ touching or being neighbored to some vertex 
	different from
	$\hat{v}$ in~$S_2$ is $\hat{v}$, then
	the union of $S_1$ and $S_2$ together with the vertices, to which the start 
	vertices of $S_1$ and $S_2$ are neighbored, 
	build an $M_1^{\ell}$ with $\ell\geqslant 1$. 
	
	Otherwise,  consider  the subsequence $S_3$  of $S_1$ which  ends at
	$\hat{v}$ and starts at the last vertex
	$\hat{y}\neq \hat{v}$  which touches or is neighbored to some vertex 
	different from   $\hat{v}$
	in $S_2$.
	Then $S_3$ and $S_2$ contain a CTNT. Since $\hat{a}$ and
	$\hat{x}$ do not touch and are not neighbored to each other there is a 
	reduced CTNT of the above CTNT containing
	$\hat{v}$, $\hat{a}$,  $\hat{x}$ and at least one other vertex. The
	existence of such a reduced CTNT of length at least $4$ implies that $G$ is
	not $M$-free (see \autoref{lem: b1 max out plan no large reduced
		cycle}). The latter statement
	together with the absence of copies  of~$M_1$ and $M_1^0$
	in $\widetilde{G}$ implies the existence of a copy of $M_1^{\ell}$, for 
	some $\ell \geqslant 1$, in~$\widetilde{G}$. 
\end{proof}


\autoref{lem: assignemnt well definied and 
properties} is one of many ingredients to prove 
the following theorem.
\begin{theoremMy}
	\label{cor:Recognition max out plan B1 quadratic}
	Let $G$ be a maximal outerplanar graph. 
	\autoref{alg: b1 con out assignment} correctly decides
	whether $G$ is $M$-free. 
\end{theoremMy}
\begin{proof}
	As an immediate consequence of \autoref{lem: assignemnt one_outputCorrect} 
	and 
	\autoref{lem: assignemnt twothree_outputCorrect} \autoref{alg: b1
		con out assignment} outputs ``$G$ is not $M$-free'' 
	if and only if $G$ is not $M$-free.
\end{proof}

It can be shown that 
\autoref{alg: b1 con out assignment} can be implemented to run in 
$O(n^2)$ time, where  $n$ is the order of the input graph $G$. 
Thus, it is possible to decide in $O(n^2)$ time whether  a maximal outerplanar
graph $G$ of order $n$ is $M$-free.  Moreover, in the positive case an  assignment with the properties stated in 
\autoref{lem: assignemnt well definied and properties} can be constructed in
$O(n^2)$ time.
\autoref{alg: b1 con out} uses such an  assignment 
to   construct a $B_1$-EPG representation of  
a maximal outerplanar $M$-free 
graph $G$. 

\begin{algorithm}[tbp]
	\footnotesize
	\caption{Construct a $B_1$-EPG representation for a maximal outerplanar 
	$M$-free
		graph $G$}     \label{alg: b1 con out}
	\Input{A maximal  outerplanar   $M$-free graph $G$,  its almost-dual 
	$\widehat{G}$, its reduced graph
		$\widetilde{G}$ and an assignment $\textAlgo{assigned} \colon 
		\widehat{V}_3 \to 
		V(G)\times
		V(G)$    as computed by   \autoref{alg: b1 con out assignment}}
	\Output{A $B_1$-EPG representation of $G$}
	\begin{algorithmic}[1]
		\Procedure{B1\_Mfree\_Max\_Outerplanar}{$G$, $\widehat{G}$, 
		$\widetilde{G}$, $\textAlgo{assigned}$}
		\State Choose an arbitrary vertex $\hat{v}$ of degree $1$ from 
		$V(\widehat{G})$
		\State Let $\hat{u}$ be the neighbor of $\hat{v}$ in $V(\widehat{G})$
		\State Let $\{a,b\} := V(T_{\hat{u}}) \cap V(T_{\hat{v}})$ and 
		$\{c\} :=   V(T_{\hat{v}})\setminus\{a,b\}$ 
		\State Draw the paths $P_a$, $P_b$ and $P_c$ and set
		$\mathcal{R}_{\hat{u}}$ 
		as depicted in \autoref{fig:const_out_tri_in_b1_4}(c) 
		
		\State Set $ToConstruct := \{(\hat{u},\hat{v}, \mathcal{R}_{\hat{u}})\}$
		\While{$ToConstruct \neq \emptyset$}
		\State Let $(\hat{u},\hat{v}, \mathcal{R}_{\hat{u}}) \in
		ToConstruct$ \label{alg: b1 con out treat one element start}
		\State Set $ToConstruct := ToConstruct\setminus 
		\{(\hat{u},\hat{v},\mathcal{R}_{\hat{u}})\}$
		\label{alg: b1 con out treat one element start1}
		\State Let $\{a,b\} := V(T_{\hat{u}}) \cap
		V(T_{\hat{v}})$\label{alg:start_processing_triple}
		\State Let $\{c\} :=   V(T_{\hat{v}})\setminus\{a,b\}$ and $\{d\} :=   
		V(T_{\hat{u}})\setminus\{a,b\}$

		\If{ $\deg(\hat{u}) = 3$}
		
\State Let $a' \in \textAlgo{assigned}(\hat{u}) \cap \{a,b\}$ and $\{b'\} = 
\{a,b\} \setminus \{a'\}$
\State Let $\hat{w} \in V(\widehat{G})$ and $e\in V(G)\setminus\{b'\}$ 
such that $\{a',d,e\} = V(T_{\hat{w}})$
\State Let $\hat{w}' \in V(\widehat{G})$ and $f \in 
V(G)\setminus\{a'\}$ such that $\{b',d,f\} = V(T_{\hat{w}'})$

\If{$\textAlgo{assigned}(\hat{u}) = \{a',b'\}$}\label{can_this_be}
\longState{Extend the paths $P_{a'}$ and $P_{b'}$, draw the path
	$P_d$ in the region $\mathcal{R}_{\hat{u}}$ and set
	$\mathcal{R}_{\hat{w}}$, $\mathcal{R}_{\hat{w}'}$ as depicted
	in  \autoref{fig:const_out_tri_in_b1_2}(b)}
\Else{ ($\textAlgo{assigned}(\hat{u}) = \{a',d\}$)}
\longState{Extend the paths $P_{a'}$ and $P_{b'}$, draw the
	path $P_d$ in the region $\mathcal{R}_{\hat{u}}$ and set
	$\mathcal{R}_{\hat{w}}$, $\mathcal{R}_{\hat{w}'}$ as depicted
	in  \autoref{fig:const_out_tri_in_b1_2}(c)}
\EndIf                
		\State $ToConstruct := ToConstruct \cup 
		\{(\hat{w},\hat{u},\mathcal{R}_{\hat{w}}), 
		(\hat{w}',\hat{u},\mathcal{R}_{\hat{w}'})\}$            
		\ElsIf{ $\deg({\hat{u}}) = 2$}           
		\State Let $\hat{w} \neq \hat{v}$ be the other neighbor of $\hat{u}$ in 
		$\widehat{G}$
		\State Let $\{a'\}:=  V(T_{\hat{w}})\cap \{a,b\}$ and $\{b'\}:=
		\{a,b\} \setminus \{a'\}$  \label{a'_b'_introduced}
		\State Let $e$ such that $\{a',d,e\} = V(T_{\hat{w}})$
		\longState{Extend the paths $P_{a'}$ and $P_{b'}$, draw the
			path $P_d$ in the region $\mathcal{R}_{\hat{u}}$  and set 
			$\mathcal{R}_{\hat{w}}$  as depicted in 
			\autoref{fig:const_out_tri_in_b1_3}(b)}
		\State $ToConstruct := ToConstruct \cup 
		\{(\hat{w},\hat{u},\mathcal{R}_{\hat{w}})\}$
		\Else{ $\deg({\hat{u}}) = 1$}
		\longState{Extend the paths $P_a$ and $P_b$ and draw the path $P_d$ in 
		the region
			$\mathcal{R}_{\hat{u}}$ as depicted in
			\autoref{fig:const_out_tri_in_b1_4}(b)}           
		\EndIf \label{alg: b1 con out treat one element end}
		\EndWhile
		\EndProcedure
	\end{algorithmic}
\end{algorithm}


\begin{figure}[tb]
    \centering
    \begin{minipage}[b]{0.25\linewidth}
        \centering
        \begin{center}
\begin{tikzpicture}
  [scale=.6]
  
  
  \node[vertex] (ve) at (1,1) {$e$};
  \node[vertex] (vd) at (3,1) {$d$};
  \node[vertex] (va) at (2,2.7) {$a'$};
  \node[vertex] (vb) at (4,2.7) {$b'$};
  \node[vertex] (vf) at (5,1) {$f$};
  \node[vertex] (vc) at (3,4.4) {$c$};
  
  \node[vertex3,label={[label distance=-5pt,scale=0.8]below left:{$\hat{v}$}}] 
  (vabc) at (3,3.5) {};

  \node[vertex3,label={[label distance=-5pt,scale=0.8]below left:{$\hat{w}$}}] 
  (vade) at (2,1.9) {};
  \node[vertex3,label={[label distance=-5pt,scale=0.8]above left:{$\hat{u}$}}] 
  (vabd) at (3,1.9) {};
  \node[vertex3,label={[label distance=-5pt,scale=0.8]below 
  left:{$\hat{w}'$}}] 
  (vbdf) 
  at (4,1.9) {};
  
  \node (vae) at (0.7,1.9) {};
  \node (ved) at (2,0.2) {};

  \node (vdf) at (4,0.3) {};
  \node (vbf) at (5.3,1.9) {};
  
  \node (vac) at (2,4) {};
  \node (vcb) at (4,4) {};

  \foreach \from/\to in {ve/vd,ve/va,vd/va,vd/vb,va/vb,vd/vf,vb/vf,va/vc,vb/vc}
  \draw (\from) -- (\to);
  
  \foreach \from/\to in
  {vade/vabd,vabd/vabc,vabd/vbdf}
  \draw[dashed,thick] (\from) -- (\to);
  
  \foreach \from/\to in
  {vabc/vac,vabc/vcb,vbdf/vbf,vbdf/vdf,vade/ved,vade/vae}
  \draw[dotted,thick] (\from) -- (\to);

\end{tikzpicture}
\end{center}
        (a)
    \end{minipage}
    \quad
    \begin{minipage}[b]{0.325\linewidth}
        \centering
        \begin{center}
\begin{tikzpicture}
  [scale=.7]

    
    \fill[gray!20] 
    (0.3,4) node[above,gray] {$\mathcal{R}_{\hat{u}}$} --(0.3,1)  
    to[out=-30, in=210] 
    (5.5,1) -- (5.5,4) 
    to[out=150, in=30] 
    (0.3,4); 
    
  \fill[fill=gray!70] (4.5,3) ellipse (0.8 and 0.5);
  \fill[fill=gray!70] (1.35,3) ellipse (0.8 and 0.5);
  \node[above,gray] at (4.5,3.5) {$\mathcal{R}_{\hat{w}'}$};
  \node[above,gray] at (1.35,3.5) {$\mathcal{R}_{\hat{w}}$};

  \draw (1.35,3) -- (4.5,3) node[above,pos=0.5] {$P_{d}$};  
  \draw[dotted] (0.85,3) -- (1.35,3);
  \draw[dotted] (4.5,3) -- (5,3);
         
  \draw (3,2) -- (3,2.85) -- (4.5,2.85) node[below,pos=0.25] {$P_{b'}$};
  \draw[dotted] (4.5,2.85) -- (5,2.85);
  \draw[dashed] (3,0.85) -- (3,2);

  \draw (1.35,2.85) -- (2.85,2.85) -- (2.85,2) node[left,pos=0.5] {$P_{a'}$};
  \draw[dotted] (1.35,2.85) -- (0.85,2.85);
  \draw[dashed] (2.85,0.85) -- (2.85,2);

\end{tikzpicture}
\end{center}
        (b)
    \end{minipage}
    \quad
    \begin{minipage}[b]{0.325\linewidth}
        \centering
        \begin{center}
\begin{tikzpicture}
  [scale=.7]
  
  
    \fill[gray!20] 
    (0.3,4) node[above,gray] {$\mathcal{R}_{\hat{u}}$} --(0.3,1)  
    to[out=-30, in=210] 
    (5.5,1) -- (5.5,4) 
    to[out=150, in=30] 
    (0.3,4);  
  
  \fill[fill=gray!70] (4.5,3) ellipse (0.8 and 0.5);
  
  \fill[fill=gray!70] (2.85,1.35) ellipse (0.5 and 0.8);
  \node[above,gray] at (4.5,3.5) {$\mathcal{R}_{\hat{w}'}$};
  \node[right,gray] at (3.35,1.35) {$\mathcal{R}_{\hat{w}}$};

  \draw (2,3) -- (4.5,3) node[above,pos=0.5] {$P_{b'}$};  
  \draw[dashed] (0.85,3) -- (2,3);
  \draw[dotted] (4.5,3) -- (5,3);
  
  \draw (3,1.35) -- (3,2.85) -- (4.5,2.85) node[below,pos=0.25] {$P_{d}$};
  \draw[dotted] (4.5,2.85) -- (5,2.85);
  \draw[dotted] (3,1.35) -- (3,0.85);

  \draw (2,2.85) -- (2.85,2.85) -- (2.85,1.35) node[left,pos=0.3] {$P_{a'}$};
  \draw[dashed] (0.85,2.85) -- (2,2.85);
  \draw[dotted] (2.85,1.35) -- (2.85,0.85);

\end{tikzpicture}
\end{center}
        (c)
    \end{minipage}
    \caption{(a) A part of a graph $G$ with $\widehat{G}$ and (b), (c) how 
    \autoref{alg: b1 con out} constructs its $B_{1}$-EPG
        representation. Dotted edges may or may not exist.}
    \label{fig:const_out_tri_in_b1_2}
\end{figure}

\begin{figure}[tb]
    \centering
    \begin{minipage}[b]{0.45\linewidth}
        \centering
        \begin{center}
\begin{tikzpicture}
  [scale=.6]
  
  
  \node[vertex] (ve) at (1,1) {$e$};
  \node[vertex] (vd) at (3,1) {$d$};
  \node[vertex] (va) at (2,2.7) {$a'$};
  \node[vertex] (vb) at (4,2.7) {$b'$};
  \node[vertex] (vc) at (3,4.4) {$c$};
  
  \node[vertex3,label={[label distance=-5pt,scale=0.8]below left:{$\hat{v}$}}] 
  (vabc) at (3,3.5) {};
  \node[vertex3,label={[label distance=-5pt,scale=0.8]below left:{$\hat{w}$}}] 
  (vade) at (2,1.9) {};
  \node[vertex3,label={[label distance=-5pt,scale=0.8]above left:{$\hat{u}$}}] 
  (vabd) at (3,1.9) {};
  
  \node (vae) at (0.5,1.9) {};
  \node (vbd) at (4,1.9) {};
  \node (ved) at (2,0.2) {};

  \node (veg) at (0,1) {};
  \node (vag) at (1,3.3) {};
  
  \node (vac) at (2,4) {};
  \node (vcb) at (4,4) {};

  \foreach \from/\to in {ve/vd,ve/va,vd/va,vd/vb,va/vb,va/vc,vb/vc}
  \draw (\from) -- (\to);
  
  \foreach \from/\to in
  {vade/vabd,vabd/vabc}
  \draw[dashed,thick] (\from) -- (\to);
  
  \foreach \from/\to in
  {vabc/vac,vabc/vcb, vade/vae, vade/ved}
  \draw[dotted,thick] (\from) -- (\to);
  
\end{tikzpicture}
\end{center}
        (a)
    \end{minipage}
    \quad
    \begin{minipage}[b]{0.45\linewidth}
        \centering
        \begin{center}
\begin{tikzpicture}
  [scale=.7]
  
    \fill[gray!20] 
    (-0.3,4) node[above,gray] {$\mathcal{R}_{\hat{u}}$} --(-0.3,2)  
    to[out=-30, in=210] 
    (6.3,2) -- (6.3,4) 
    to[out=150, in=30] 
    (-0.3,4);   

  \fill[fill=gray!70] (5,3) ellipse (1 and 0.5);  
  \node[above,gray] at (5,3.5) {$\mathcal{R}_{\hat{w}}$};
  
  \draw (1,2.85) -- (5,2.85) node[above,pos=0.15] {$P_{b'}$};
  \draw[dashed] (1,2.85) -- (0,2.85);
  \draw[dotted] (5,2.85) -- (5.5,2.85);

  \draw (1,2.7) -- (3.5,2.7) node [below,pos=0.5] {$P_{a'}$};
  \draw[dashed] (1,2.7) -- (0,2.7);

  \draw (2.5,3) -- (5,3) node [above,pos=0.4] {$P_{d}$};
  \draw[dotted] (5,3) -- (5.5,3);  

\end{tikzpicture}
\end{center}
        (b)
    \end{minipage}
    \caption{(a) A part of a graph $G$ with $\widehat{G}$ and (b) how 
    \autoref{alg: b1 con out} constructs its $B_{1}$-EPG
        representation. Dotted edges may or may not exist.}
    \label{fig:const_out_tri_in_b1_3}
\end{figure}

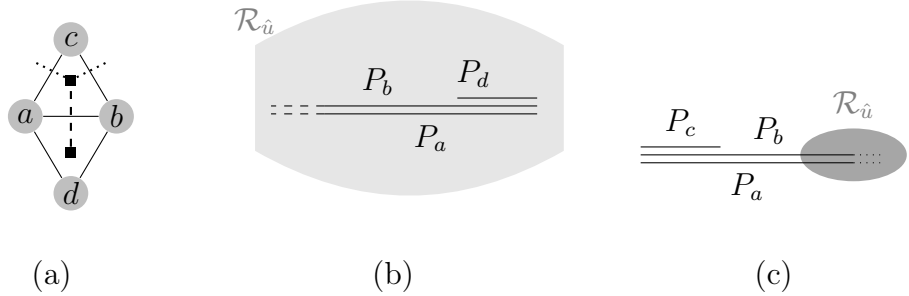
\begin{figure}[tb]
    \centering
    \begin{minipage}[b]{0.25\linewidth}
        \centering
        \begin{center}
\begin{tikzpicture}
  [scale=.6]
  

  \node[vertex] (vd) at (3,1) {$d$};
  \node[vertex] (va) at (2,2.7) {$a$};
  \node[vertex] (vb) at (4,2.7) {$b$};
  \node[vertex] (vc) at (3,4.4) {$c$};
  
  \node[vertex3,label={[label distance=-5pt,scale=0.8]below left:{$\hat{v}$}}] 
  (vabc) at (3,3.5) {};
  \node[vertex3,label={[label distance=-5pt,scale=0.8]above left:{$\hat{u}$}}] 
  (vabd) at (3,1.9) {};

  \node (vbd) at (4,1.9) {};

  \node (vag) at (1,3.3) {};
  
  \node (vac) at (2,4) {};
  \node (vcb) at (4,4) {};

  \foreach \from/\to in {vd/va,vd/vb,va/vb,va/vc,vb/vc}
  \draw (\from) -- (\to);
  
  \foreach \from/\to in
  {vabd/vabc}
  \draw[dashed,thick] (\from) -- (\to);
  
  \foreach \from/\to in
  {vabc/vac,vabc/vcb}
  \draw[dotted,thick] (\from) -- (\to);
  
\end{tikzpicture}
\end{center}
        (a)
    \end{minipage}
    \quad
    \begin{minipage}[b]{0.325\linewidth}
        \centering
        \begin{center}
\begin{tikzpicture}
  [scale=.7]
  
    \fill[gray!20] 
    (-0.3,4) node[above,gray] {$\mathcal{R}_{\hat{u}}$} --(-0.3,2)  
    to[out=-30, in=210] 
    (5.5,2) -- (5.5,4) 
    to[out=150, in=30] 
    (-0.3,4);

  \draw (1,2.85) -- (5,2.85) node[above,pos=0.25] {$P_{b}$};
  \draw[dashed] (1,2.85) -- (0,2.85);

  \draw (1,2.7) -- (5,2.7) node [below,pos=0.5] {$P_{a}$};
  \draw[dashed] (1,2.7) -- (0,2.7);

  \draw (3.5,3) -- (5,3) node [above,pos=0.2] {$P_{d}$};

\end{tikzpicture}
\end{center}
        (b)
    \end{minipage}
    \quad
    \begin{minipage}[b]{0.325\linewidth}
        \centering
        \begin{center}
\begin{tikzpicture}
  [scale=.7]

  \fill[fill=gray!70] (5,3) ellipse (1 and 0.5);  
  \node[above,gray] at (5,3.5) {$\mathcal{R}_{\hat{u}}$};
  
  \draw (1,2.85) -- (5,2.85) node[below,pos=0.5] {$P_{a}$};
  \draw[dotted] (5,2.85) -- (5.5,2.85);

  \draw (1,3.15) -- (2.5,3.15) node [above,pos=0.5] {$P_{c}$};

  \draw (1,3) -- (5,3) node [above,pos=0.6] {$P_{b}$};
  \draw[dotted] (5,3) -- (5.5,3);  

\end{tikzpicture}
\end{center}
        (c)
    \end{minipage}    
    \caption{(a) A part of a graph $G$ with $\widehat{G}$ and (b) how 
    \autoref{alg: b1 con out} constructs its $B_{1}$-EPG
        representation. Dotted edges may or may not exist. (c) The starting 
        construction of \autoref{alg: b1 con out}. }
    \label{fig:const_out_tri_in_b1_4}
\end{figure}
\smallskip

The next result establishes the correctness of \autoref{alg: b1 con 
out}.
\begin{lemmaMy} \label{lem: b1 max out algo}
	Let $G$ be a maximal outerplanar $M$-free graph with  almost-dual
	$\widehat{G}$  and   reduced graph $\widetilde{G}$.
	Then \autoref{alg: b1 con out} constructs a $B_1$-EPG representation of $G$.
\end{lemmaMy}
\begin{proof}
	We first prove  that all paths can be constructed as described 
	in 
	\autoref{alg: b1 con out}. Towards that end, note that we have the 
	following invariants 
	throughout the algorithm.
	For each $(\hat{u},\hat{v},\mathcal{R}_{\hat{u}}) \in ToConstruct$ 
	the vertices $\hat{u}$ and $\hat{v}$ are adjacent in $\widehat{G}$ and  
	their 
	corresponding triangles $T_{\hat{u}}$ and $T_{\hat{v}}$ share two vertices 
	$a$ and $b$ 
	in $G$. 
	Furthermore the paths of the three vertices of $T_{\hat{v}}$ have already 
	been 
	constructed. Moreover $\mathcal{R}_{\hat{u}}$ indicates a region of the
	$B_1$-EPG representation, 
	where the paths $P_a$ and $P_b$ intersect and where  they can be extended
	such that no other      paths are in this region so far.
	This implies that all the paths can really be constructed as  described by 
	\autoref{alg: b1 con out}.
	
	Next we prove that \autoref{alg: b1 con out} constructs exactly one  path 
	for each vertex of $G$. 
	After the execution of the operations in 
	lines~\ref{alg:start_processing_triple} 
	to~\ref{alg: b1 con out treat one element end} of the while loop for the
	triple $(\hat{u},\hat{v},\mathcal{R}_{\hat{u}})$  all paths corresponding 
	to vertices of 
	$T_{\hat{u}}$ have been constructed and all neighbors of $\hat{u}$ have 
	been added to 
	$ToConstruct$. This implies that  for each vertex of $G$ the path
	representing  it on the grid has been constructed at this time, thus the
	algorithm constructs one path per vertex.   
	Further, the algorithm constructs new  paths  
	only  for
	those vertices the paths of which have not been constructed yet.
	Thus the algorithm does not construct more than one path per vertex. 
	
	Next we  show  that two paths $P_a$ and $P_b$ intersect iff 
	the vertices~$a$ and $b$ are adjacent. 
	When constructing the paths, it is easy to see that if $P_a$ and 
	$P_b$ intersect, then 
	their corresponding vertices are adjacent.
	On the other hand, if $a$ and $b$ are adjacent vertices of $G$, then there
	is a
	vertex  $\hat{v}$ of $\widehat{G}$ such that $\{a,b\} \subseteq 
	V(T_{\hat{v}})$. 
	\autoref{alg: b1 con out} 
	constructs the 
	paths $P_a$ and $P_b$ as intersecting paths, when 
	$\hat{v}$ or one of its neighbors  containing both $a$ and $b$ is
	considered in     lines~\ref{alg: b1 con out treat one element start} 
	to~\ref{alg: b1 con out treat one element end}.   
	
	Finally we show  that every path has at most one 
	bend. By construction a path $P_a$ corresponding to a vertex $a \in 
	V(G)$ only bends 
	if $a \in V(T_{\hat{v}})$ and $a$ was assigned to $\hat{v}$ in the 
	assignment given by \autoref{alg: b1 con out assignment}. 
	Due to \autoref{lem: assignemnt well definied and properties} each vertex 
	of $G$ is 
	assigned to at most one vertex of $\widehat{G}$, 
	hence each path $P_a$ bends at most once. 
\end{proof}

\autoref{lem: b1 max out algo} implies the following 
characterization of  $B_1$-EPG 
maximal outerplanar graphs.
\begin{theoremMy}
    \label{out_tri_in_b1}
    Let $G$ be a maximal outerplanar graph. 
    Then $G$ is  $B_{1}$-EPG if and only if $G$ is $M$-free.
\end{theoremMy}
\begin{proof}
	If $G$ is not $M$-free, then 
	$G$ is not $B_{1}$-EPG by \autoref{m1_n_non_in_b1}. 
	Otherwise, \autoref{alg: b1 con out} constructs a $B_1$-EPG representation 
	according to \autoref{lem: b1 max out algo}, so $G$ is $B_1$-EPG.
\end{proof}

\autoref{out_tri_in_b1} and 
\autoref{cor:Recognition max out plan B1 quadratic} imply that it can be 
decided 
in $O(n^2)$ time
whether  a given maximal outerplanar graph $G$ of order  $n$ is $B_1$-EPG. 
Furthermore,  \autoref{alg: b1 con out} can be implemented to run  in 
 $O(n)$ time for an input graph $G$ of order $n$. Thus,  a 
$B_1$-EPG representation of a $B_1$-EPG maximal outerplanar graph $G$ of order
$n$ can be
constructed in $O(n^2)$ time.

%

Recalling  that all outerplanar graphs
are in $B_{2}^{m}$ (see  
\autoref{out_in_b2m}),
we obtain  
the following  full characterization of 
 the maximal outerplanar graphs belonging  to   $B_{0}$, $B_{1}^{m}$, $B_{1}$ 
 and
 $B_{2}^{m}$.
\begin{corollary}
 The (monotonic) bend number of a maximal outerplanar graph  $G$ is given as follows.
\begin{align*}
 b(G) &=\left \{ \begin{array}{ll}
0 & \mbox{if $G$ is  $S_3$-free}\\
1& \mbox{if $G$ is $M$-free but not  $S_3$-free}\\
2 & \mbox{otherwise}\end{array} \right .\\
b^m(G)&= \left \{ \begin{array}{ll} 0 & \mbox{if $G$ is  
$S_3$-free}\\
2  & \mbox{otherwise}\end{array} \right .
\end{align*}
\end{corollary}

\section{The  (monotonic) bend number  of cacti}
\label{sec:Cacti}

\begin{definition}
A graph is called a \textit{cactus} if it is connected and any two simple cycles
in it  have at most one vertex in common.
\end{definition}

It is easy to see that every cactus is outerplanar, so due to 
\autoref{out_in_b2m} cacti are in $B_2^m$. 
Notice that  cacti are in some sense the opposite to maximal
outerplanar graphs: in a cactus  edges which connect non-consecutive
vertices belonging to the same simple cycle are not allowed, whereas in maximal
outerplanar graphs there have to be as many edges as possible connecting such
vertices.

\subsection{Cacti in \texorpdfstring{$B_{0}$}{B0}}
\label{sec:cacti_B_0}

Consider first some   simple  cacti which are not in $B_{0}$:  cycles
$C_{r}$ for $r \geqslant 4$, 
and the graphs $M_2$ and $M_3$ depicted in  \autoref{fig:s3_m2_m3}. 
Indeed, 
as shown in~\cite{lekkeikerker1962representation} (cf.\ Introduction) 
an interval graph
  (a) does not contain an induced cycle of length more than three and (b) in  
  any
  triple of  pairwise distinct and pairwise non-adjacent vertices, there 
  exists
  at least one vertex which is  a
  neighbor to any path connecting the two other vertices.  
It can be easily seen that none of the graphs mentioned above fulfills both (a)
and (b), hence none of these graphs is  an  interval graph.
Therefore, they are not contained  in~$B_{0}$ (and equivalently in
$B_{0}^m$), because   a graph $G$ is in~$B_{0}$ (and equivalently in~$B_{0}^m$)
iff~$G$ is an interval graph,  as mentioned in 
\autoref{intro:sec}.

\begin{figure}[tbp]
  \centering
   \begin{minipage}[b]{0.30\linewidth}
        \centering
        \begin{center}
\begin{tikzpicture}
  [scale=.5]
  
  \node[vertex] (a) at (1,1) {$x_{3}$};
  \node[vertex] (b) at (3,1) {$x_{1}$};
  \node[vertex] (c) at (2,2.7) {$x_{2}$};
  \node[vertex] (ap) at (4,2.7) {$y_{1}$};
  \node[vertex] (bp) at (0,2.7) {$y_{2}$};
  \node[vertex] (cp) at (2,-0.7) {$y_{3}$};
  
  \foreach \from/\to in {a/b,a/c,a/bp,a/cp,b/c,b/ap,b/cp,c/ap,c/bp}
  \draw (\from) -- (\to);
  
\end{tikzpicture}
\end{center}
        (a)
    \end{minipage}
    \begin{minipage}[b]{0.30\linewidth}
        \centering
        \begin{center}
\begin{tikzpicture}
  [scale=.5]
  
  \node[vertex] (a) at (1,1) {$a$};
  \node[vertex] (b) at (-1,3) {$b$};
  \node[vertex] (c) at (1,3) {$c$};
  \node[vertex] (d) at (3,3) {$d$};
  \node[vertex] (e) at (-1,5) {$e$};
  \node[vertex] (f) at (1,5) {$f$};
  \node[vertex] (g) at (3,5) {$g$};

  \foreach \from/\to in {a/b,a/c,a/d,b/e,c/f,d/g}
  \draw (\from) -- (\to);
  
\end{tikzpicture}
\end{center}
        (b)
    \end{minipage}
    \quad
    \begin{minipage}[b]{0.30\linewidth}
        \centering
        \begin{center}
\begin{tikzpicture}
  [scale=.5]
  
  \node[vertex] (a) at (1,1) {$a$};
  \node[vertex] (b) at (3,1) {$b$};
  \node[vertex] (c) at (2,3) {$c$};
  \node[vertex] (d) at (-1,1) {$d$};
  \node[vertex] (e) at (5,1) {$e$};
  \node[vertex] (f) at (2,5) {$f$};

  \foreach \from/\to in {a/b,a/c,b/c,a/d,b/e,c/f}
  \draw (\from) -- (\to);
  
\end{tikzpicture}
\end{center}
        (c)
    \end{minipage}
    \caption{(a) The graph $S_{3}$. (b) The graph $M_{2}$. (c) The graph $M_{3}$.}
    \label{fig:s3_m2_m3}
\end{figure}

\begin{definition}\label{MC-free-ouerplanar}
    A cactus is called \emph{$MC$-free} if it 
    contains neither  $M_{2}$, nor $M_{3}$, nor $C_{r}$, for any $r \geqslant 4$, as
    an induced subgraph. 
\end{definition}

Since $C_r$, for any $r\geqslant 4$, $M_2$ and $M_3$ are not in $B_0$,   a 
cactus 
that is not $MC$-free 
is not
in $B_0$. The next theorem 
 shows that  the converse is also true. Its constructive proof
 is based on  \autoref{alg: b0 cactus}.
\begin{algorithm}[tbp]
	\footnotesize
	\caption{Construct a $B_0$-EPG representation of an $MC$-free cactus}
	\label{alg: b0 cactus}
	\Input{An $MC$-free  cactus   $G = (V(G),E(G))$ and
		the set   $\delta_G(v)$ of  edges incident to  $v$ in $G$,  $\forall 
		v\in V(G)$}
	\Output{A $B_0$-EPG representation of $G$}
	\begin{algorithmic}[1]
		\Procedure{B0\_MCfree\_cactus}{$G$}
		\State{Set $\bar{\Delta}:=\seq{}$ and $\Delta':=\seq{}$, where $\seq{}$ 
		denotes an empty list}
		\State{Set $\bar{G}:=G$}
		\While{$G$ contains a triangle with vertex set  
		$\{v_1,v_2,v_3\}$}\label{BeginWhile:barG}
		\State{Determine a vertex $v$ of degree two in $\{v_1$,$v_2$,$v_3\}$}
		\State{Set $V(\bar{G}):=V(\bar{G})\setminus \{v\}$ and $E(\bar{G}):=
			E(\bar{G})\setminus \{\delta_G(v)\}$}
		\State{Set $\bar{\Delta}:=\seq{\bar{\Delta},v}$}
		\EndWhile\label{EndWhile:barG}
		\For{$v\in V(\bar{G})$} \label{BeginConstr:EPGbarG}
		\State{If $v$ is a leaf in $\bar{G}$, set $\Delta':=\seq{\Delta',v}$}
		\EndFor
		\longState{Set $V(G'):=V(\bar{G})\setminus \Delta'$, 
		$E(G'):=E(\bar{G})\setminus
			(\cup_{v\in \Delta'} \delta_{\bar{G}}(v))$, where 
			$\delta_{\bar{G}}(v)$ is the set of
			edges incident with $v$ in  $\bar{G}$}
		\State{Set $G' = (V(G'),E(G'))$}\label{EndConstr:G'}
		\State{Set $k:=|V(G')|$}\label{BeginConstr:EPGG'}
		\For{$i = 1, \dots,k$}
		\longState{Construct the path $P_{w_i}$ corresponding to the $i$-vertex 
		$w_i$ in
			$G'$ as shown in \autoref{fig:path_in_b0}}
		\EndFor\label{EndConstr:EPGG'}
		\For{$v\in \Delta'$}\label{BeginExtend:EPGG'}
		\State{Let  $w\in V(\bar{G})$ such that  
			$\{(w,v)\}=\delta_{\bar{G}}(v)$} 
		\longState{Construct a path
			$P_v$  representing $v$ as a subpath of
			$P_w$  which is not contained  in  any other
			path constructed already}
		\EndFor\label{EndExtend:EPGbarG}
		\For{$v\in \bar{\Delta}$}\label{BeginExtend:EPGG}
		\longState{Construct a path $P_v$ representing $v$ as a common edge of 
		the (already constructed) paths
			$P_{v_1}$, $P_{v_2}$ for $\{(v,v_1),(v,v_2)\}=\delta_G(v)$}
		\EndFor\label{EndExtend:EPGG}
		\EndProcedure
	\end{algorithmic}
\end{algorithm}
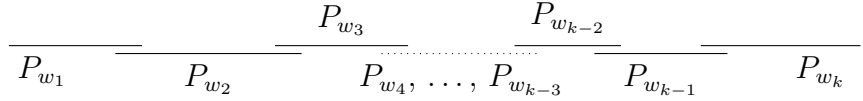
\begin{figure}[tbp]
	\begin{center}
\begin{tikzpicture}
  [scale=.7]

  \draw (1,1) -- (3.5,1) node[below,pos=0.25] {$P_{w_{1}}$};   
  \draw (3,0.85) -- (6.5,0.85) node[below,pos=0.5] {$P_{w_{2}}$};
  \draw (6,1) -- (8.5,1) node[above,pos=0.5] {$P_{w_{3}}$};
  \draw[dotted] (8,0.85) -- (11,0.85) node[below,pos=0.5] {$P_{w_{4}}$,  \dots, $P_{w_{k-3}}$};
  
  \draw (10.5,1) -- (12.5,1) node[above,pos=0.5] {$P_{w_{k-2}}$};   
  \draw (12,0.85) -- (14.5,0.85) node[below,pos=0.5] {$P_{w_{k-1}}$};
  \draw (14,1) -- (17,1) node[below,pos=0.75] {$P_{w_{k}}$};
  
\end{tikzpicture}
\end{center}
	\caption{The $B_{0}$-EPG representation of the path $G'$ of \autoref{alg: 
	b0 cactus}.}
	\label{fig:path_in_b0}
\end{figure}
\begin{theoremMy}\label{B_0_cactus:theo}
 A cactus $G$ is in $B_{0}$ if and only if $G$ 
is $MC$-free.
\end{theoremMy}
\begin{proof}
	As mentioned in \autoref{sec:cacti_B_0} a graph which is not $MC$-free is 
	not in
	$B_0$. The sufficient condition of \autoref{B_0_cactus:theo} is proven by 
	construction, more precisely,
	by showing that \autoref{alg: b0 cactus} correctly constructs a $B_0$-EPG
	representation of an $MC$-free cactus $G$.
	
	Notice that if  $G$ contains any cycles, then those are  triangles.
	Furthermore notice that if the vertex set $\{v_{1},v_{2},v_{3}\}$ induces a
	triangle in $G$, then  there is no  vertex $v_{4} \in V(G)\setminus
	\{v_1,v_2,v_3\}$ adjacent to at least two (assume w.l.o.g.\ $v_{1}$ and
	$v_{2}$) vertices of $\{v_{1},v_{2},v_{3}\}$, because otherwise  the simple 
	cycles $\{ v_{1},v_{2},v_{3}\}$
	and $\{v_{1},v_{2},v_{4}\}$ would share an edge and contradict the property 
	of $G$ being a  cactus. 
	Hence if the vertex set $\{v_{1},v_{2},v_{3}\}$ induces a triangle in $G$, 
	then
	at least one vertex in
	$\{v_{1}, v_{2}, v_{3}\}$ has degree $2$, otherwise $G$ would contain 
	$M_{3}$ as an induced subgraph. 
	Assume w.l.o.g.\ that $v_{3}$ has degree $2$.
	
	If  the graph $G\setminus v_3$ obtained  from $G$ by deleting the vertex
	$v_{3}$ and its incident  edges 
	$\{v_{1},v_{3}\}$,  $\{v_{2},v_{3}\}$ had a $B_{0}$-EPG representation, 
	the latter could be extended to a 
	$B_{0}$-EPG representation of $G$ by  simply representing $v_3$ by just one 
	grid   edge in the
	$B_0$-EPG representation of $G\setminus v_3$ shared by the paths $P_{v_1}$ 
	and $P_{v_2}$ corresponding to $v_1$ and
	$v_2$, respectively. 
	Obviously, this argument holds for any triangle in $G$.  
	Following this idea,   \autoref{alg: b0 cactus} first constructs the 
	graph   
	$\bar{G}$ obtained
	by removing from $G$  one  vertex  of 
	degree~$2$ together with its incident edges from every triangle in $G$. 
	This is 
	done in  lines~\ref{BeginWhile:barG}-\ref{EndWhile:barG}.  
	In lines~\ref{BeginConstr:EPGbarG}-\ref{EndExtend:EPGbarG} the algorithm
	constructs
	a $B_0$-EPG representation of~$\bar{G}$, which is then extended to  a 
	$B_0$-EPG
	representation of $G$ in lines~\ref{BeginExtend:EPGG}-\ref{EndExtend:EPGG} 
	as  explained above.

	Next consider the construction of a $B_0$-EPG representation of $\bar{G}$.
	Observe first that $\bar{G}$  is a cactus which  does not
	contain an induced $M_2$ or $C_{r}$, for any $ r 
	\geqslant 3$.
	Thus $\bar{G}$ is  a tree. More precisely $\bar{G}$ is a so-called
	caterpillar, i.e.\ the   graph $G'$ obtained from  $\bar{G}$ by 
	deleting all its
	leaves and the edges incident to them is a path. 
	Indeed, there is  no  vertex $v$  with degree more
	than $2$ in $G'$, because if  
	a  vertex  $v$ with    $3$  neighbors  $v_{1}$, $ v_{2}$,
	$v_{3}$ would exists  in
	$G'$, then  each of  $v_{1}$, $ v_{2}$,
	$v_{3}$ would have  degree at least $2$ in $\bar{G}$
	and  $\bar{G}$  would contain  $M_{2}$ as an induced
	subgraph.

	Let   $G'$ be the   path $(w_1,w_2,\ldots,w_k)$ for some $k \leqslant 
	|V(G)|$ and  $w_i\in V(G)$,
	$1\leqslant i\leqslant k$. It is straightforward to construct a $B_0$-EPG 
	representation of 
	$(w_1,w_2,\ldots,w_k)$ as  depicted in
	\autoref{fig:path_in_b0}. Moreover this $B_0$-EPG representation of $G'$ 
	can be
	easily extended to a $B_{0}$-EPG representation of $\bar{G}$ in the
	following way. 
	For any   leaf $v$ of $\bar{G}$ consider  its neighbor  $w_{\ell}$ in
	$\bar{G}$.  Since $w_{\ell}$  belongs to  $G'$,  there is a path
	$P_{w_{\ell}}$  representing $w_{\ell}$ in the $B_{0}$-EPG 
	representation of $G'$.
	Construct a subpath $P_{v}$ of $P_{w_{\ell}}$ to represent $v$,
	such that $P_v$   does not intersect any other path constructed so far. 
	Obviously this constriction yields a $B_{0}$-EPG representation of 
	$\bar{G}$.
	Following these ideas the algorithm first  constructs the path $G'$ in lines
	\ref{BeginConstr:EPGbarG}-\ref{EndConstr:G'}, then it constructs the 
	$B_0$-EPG representation of $G'$ in
	lines~\ref{BeginConstr:EPGG'}-\ref{EndConstr:EPGG'} and finally it extends 
	the
	latter  to a $B_0$-EPG
	representation of $\bar{G}$ in 
	lines~\ref{BeginExtend:EPGG'}-\ref{EndExtend:EPGbarG}.
\end{proof}

A standard analysis of the time complexity of \autoref{alg: b0 
cactus} reveals that
a $B_0$-EPG
representation  a cactus of order $n$ belonging to  $B_0$ can be constructed  
in $O(n\log(n))$ time.
The existence of a faster construction of a  $B_0$-EPG representation of such a
cactus remains an open problem. 
%

\subsection{Cacti in \texorpdfstring{$B_{1}$}{B1}}
In this section we show that every cactus belongs to $B_1$  by demonstrating 
that 
\autoref{alg: b1m cactus} constructs a $B_1^m$-EPG representation of an
arbitrary cactus.
\begin{algorithm}[tbp]
	\footnotesize
	\caption{Construct a $B_1^m$-EPG representation of a cactus  $G$}
	\label{alg: b1m cactus}
	\Input{A cactus $G = (V(G),E(G))$ with $n =|V(G)|$  and
		the set   $\delta(v)$ of  edges incident with  $v$ in $G$,  $\forall 
		v\in V(G)$}
	\Output{A $B_1^m$-EPG representation of $G$}
	\begin{algorithmic}[1]
		\Procedure{B1m\_cactus}{$G$}
		\State{Set $col(v):=gray$ for all $v\in V(G)$}
		\State{Select an arbitrary vertex $v_0\in V(G)$}
		\State Set $col(v_0):=green$
		\longState{Draw  the path
			$P_{v_0}$ representing $v_0$ as a straight horizontal line on the 
			grid }
		\longState{Set the free part $\mathcal{R}_{v_0}$ of $P_{v_0}$ to be the 
		whole grid}
		\While{There is a green vertex  in $V(G)$}\label{BeginWhile}
		\State{Select  a vertex $v\in V(G)$ with $col(v)=green$}\label{select}
		\longState{Determine all $k_v$ cycles $C_1(v)$, \ldots, $C_{k_v}(v)$ in
			$G$,
			which contain $v$ and only gray  vertices except for $v$ ($k_v\in 
			\nz\cup \{0\}$)}\label{begin:explore}
		\longState{Determine $C_0(v):=\{u\in V(G)\colon \{u,v\} \in \delta(v), 
		col(u)=gray\}
			\setminus
			\cup_{i=1}^{k_v}V(C_i(v))$ 
		}
		\For{$u\in C_0(v)$} 
		\longState{Construct the  path $P_u$ representing $u$ such that it lies 
		in the free
			part ${\cal R}_v$ of $P_v$ and only intersects $P_v$, but no other 
			already
			constructed path,  as shown in \autoref{fig:cacti_in_b1m}(a)}
		\State{Set $col(u):=green$}\label{color:green}
		\EndFor
		\For{$i=1,\ldots,k_v$}
		\longState{Let $\seq{v,u_1,u_2,\ldots,u_{\ell-1},v}$ be the vertices of 
		$C_i(v)$ in the order they are visited in the cycle $C_i(v)$}
		\For{$j=1,\ldots, \ell-1$}
		\longState{Construct the grid  path $P_{u_j}$ representing  $u_j$ such 
		that}
		\State {(1) $P_{u_j}$   lies on the free
			part ${\cal R}_v$ of $P_v$}
		\State{(2) $P_{u_j}$  intersects $P_{u_{j-1}}$}
		\State{(3)$P_{u_j}$  intersects $P_v$ if $j=\ell-1$}
		\State{(4) $P_{u_j}$ does not intersect   other already
			constructed paths}
		\longState{ as shown in \autoref{fig:cacti_in_b1m}(b), 
		\autoref{fig:cacti_in_b1m}(c) and
			\autoref{fig:cacti_in_b1m}(d) for $\ell=3$, $\ell=4$ and  
			$\ell\geqslant 5$, respectively}
		\State{Set $col(u_j):=green$}
		\EndFor
		\EndFor
		\State{Set $col(v):=red$}\label{end:explore}
		\EndWhile\label{EndWhile}
		\EndProcedure
	\end{algorithmic}
\end{algorithm}

\begin{figure}[tbp]
	\centering
	\begin{minipage}[b]{0.39\linewidth}
		\centering
		\begin{center}
    \begin{small}
\begin{tikzpicture}
  [scale=.58]

    \fill[gray!20] (-3.5,3.5)--(-3.5,0)  
        to[out=-30, in=210] 
        (3,0) -- (3,3.5) 
        to[out=150, in=30] 
        (-3.5,3.5); 
        \node[above,gray] at (-3.5,3.5) {$\mathcal{R}_{v}$};
        
   \draw[dashed] (-1.5,3.5)--(-1.5,0)  
   to[out=-30, in=210] 
   (1,0) -- (1,3.5) 
   to[out=150, in=30] 
   (-1.5,3.5); 
   \node[below] at (-0.25,  -0.6) {$C_{0}(v)$};
   \draw[dashed]   (-2.5,   0.85) ellipse (.5 and 0.8);
   \node[below] at (-2.7,   0   ) {$C_{0}(v)$};   
   \draw[dashed]   (2, 0.85) ellipse (.5 and 0.8);
   \node[below] at (2.2, 0   ) {$C_{0}(v)$};

  \draw (-1.5,0.85) -- (1,0.85) node[below,pos=0.7] {$P_{v}$};
  \draw[dashed] (-4,0.85) -- (-1.5,0.85);
  \draw[dashed] (1,0.85) -- (3.5,0.85);
  
  \fill[fill=gray!70] (0,2) ellipse (.8 and 0.5);
  \node[above,gray] at (1.1,2.2) {$\mathcal{R}_{u}$};   
  \draw (-1,1) -- (0,1) -- (0,3) node[left,pos=1] {$P_{u}$};

\end{tikzpicture}
\end{small}
\end{center}
		(a)
	\end{minipage}
	\quad
	\begin{minipage}[b]{0.56\linewidth}
		\centering
		\begin{center}
    \begin{small}
\begin{tikzpicture}
  [scale=.58]

    \fill[gray!20] (-0.5,4.5)--(-0.5,0)  
        to[out=-30, in=210] 
        (8,0) -- (8,4.5) 
        to[out=150, in=30] 
        (-0.5,4.5); 
        \node[above,gray] at (-0.5,4.5) {$\mathcal{R}_{v}$};

   \draw[dashed] (1.5,4.5)--(1.5,0)  
   to[out=-30, in=210] 
   (6,0) -- (6,4.5) 
   to[out=150, in=30] 
   (1.5,4.5);
   \node[below] at (3.75,  -0.8) {$C_{i}(v)$};
   \draw[dashed]   (0.5,   0.85) ellipse (.5 and 0.8);
   \node[below] at (0.3,   0   ) {$C_{i-1}(v)$};   
   \draw[dashed]   (7, 0.85) ellipse (.5 and 0.8);
   \node[below] at (7.2, 0   ) {$C_{i+1}(v)$};        
  
  \draw (1.5,0.85) -- (6,0.85) node[below,pos=0.7] {$P_{v}$}; 
  \draw[dashed] (-1,0.85) -- (1.5,0.85);
  \draw[dashed] (6,0.85) -- (8.5,0.85);
  
  \fill[fill=gray!70] (3,3) ellipse (.8 and 0.5);   
  \fill[fill=gray!70] (5,2) ellipse (.8 and 0.5);
  \node[above,gray] at (3.8,3.5) {$\mathcal{R}_{u_1}$}; 
  \node[above,gray] at (6.2,2.3) {$\mathcal{R}_{u_2}$}; 
  \draw (2,1) -- (3,1) -- (3,4) node[left,pos=0.3] {$P_{u_{1}}$};
  \draw (2.5,1.15) -- (5,1.15) -- (5,3) node[above,pos=1] {$P_{u_{2}}$};

\end{tikzpicture}
\end{small}
\end{center}
		(b)
	\end{minipage}
	\centering
	
	\begin{minipage}[b]{0.36\linewidth}
		\centering
		\begin{center}
\begin{small}
\begin{tikzpicture}
  [scale=.51]

    \fill[gray!20] (4.5,5)--(4.5,-2)  
        to[out=-30, in=210] 
        (12.5,-2) -- (12.5,5) 
        to[out=150, in=30] 
        (4.5,5); 
        \node[above,gray] at (4.5,5) {$\mathcal{R}_{v}$};
        
   \draw[dashed] (6.5,5)--(6.5,-2.1)  
   to[out=-30, in=210] 
   (10.5,-2.1) -- (10.5,5) 
   to[out=150, in=30] 
   (6.5,5); 
   \node[below] at (8.5,  -2.6) {$C_{i}(v)$};
   \draw[dashed]   (5.5,   0.85) ellipse (.5 and 0.8);
   \node[below] at (5.2,   0   ) {$C_{i-1}(v)$};   
   \draw[dashed]   (11.5, 0.85) ellipse (.5 and 0.8);
   \node[below] at (11.9, 0   ) {$C_{i+1}(v)$};

  \draw (6.5,0.85) -- (10.5,0.85) node[below,pos=0.85] {$P_{v}$};
  \draw[dashed] (4,0.85) -- (6.5,0.85);
  \draw[dashed] (10.5,0.85) -- (13,0.85);
  
  \fill[fill=gray!70] (7.85,4) ellipse (.8 and 0.5);   
  \fill[fill=gray!70] (9,2.5) ellipse (.5 and 0.8); 
  \fill[fill=gray!70] (8.15,-1) ellipse (.8 and 0.5);     
  \node[above,gray] at (8.7,4.2) {$\mathcal{R}_{u_1}$}; 
  \node[below,gray] at (9.7,1.9) {$\mathcal{R}_{u_2}$}; 
  \node[below,gray] at (8.95,-1.5) {$\mathcal{R}_{u_3}$}; 
  \draw (7,1) -- (7.85,1) -- (7.85,5) node[left,pos=0.3] {$P_{u_{1}}$};
  \draw (8,0) -- (8,2.5) -- (10,2.5) node[above,pos=0.9] {$P_{u_{2}}$};
  \draw (8.15,-2) -- (8.15,0.7) -- (9,0.7) node[below,pos=0.6] {$P_{u_{3}}$};

\end{tikzpicture}
\end{small}
\end{center}
		(c)
	\end{minipage}
	\quad
	\begin{minipage}[b]{0.60\linewidth}
		\centering
		\begin{center}
    \begin{small}
\begin{tikzpicture}
  [scale=.51]

    \fill[gray!20] (9,6)--(9,0)  
        to[out=-20, in=200] 
        (23.5,0) -- (23.5,6) 
        to[out=160, in=20] 
        (9,6); 
        \node[above,gray] at (9,6) {$\mathcal{R}_{v}$};
        
   \draw[dashed] (11,6)--(11,0.5)  
      to[out=-20, in=200] 
      (21.5,0.5) -- (21.5,6) 
      to[out=160, in=20] 
      (11,6); 
   \node[below] at (16,  -0.5) {$C_{i}(v)$};
   \draw[dashed]   (10,   0.85) ellipse (.5 and 0.8);
   \node[below] at (10,   0   ) {$C_{i-1}(v)$};   
   \draw[dashed]   (22.5, 0.85) ellipse (.5 and 0.8);
   \node[below] at (22.5, 0   ) {$C_{i+1}(v)$};

  \draw (11,0.85) -- (21.5,0.85) node[below,pos=0.5] {$P_{v}$};
  \draw[dashed] (8.5,0.85) -- (11,0.85);
  \draw[dashed] (21.5,0.85) -- (24,0.85);
      
  \fill[fill=gray!70] (12.5,4.5) ellipse (.8 and 0.5);   
  \fill[fill=gray!70] (13.5,3) ellipse (.5 and 0.8); 
  \fill[fill=gray!70] (16,3.15) ellipse (.5 and 0.8);
  \fill[fill=gray!70] (19.5,3.45) ellipse (.5 and 0.8);   
  \fill[fill=gray!70] (20.65,5) ellipse (.8 and 0.5);   
  \node[above,gray] at (13.2,4.7) {$\mathcal{R}_{u_1}$}; 
  \node[below,gray] at (13.5,2.4) {$\mathcal{R}_{u_2}$}; 
  \node[below,gray] at (16,2.55) {$\mathcal{R}_{u_3}$};   
  \node[below,gray] at (19.5,2.95) {$\mathcal{R}_{u_{\ell-2}}$};
  \node[above,gray] at (19.75,5.2) {$\mathcal{R}_{u_{\ell-1}}$};   
  \draw (11.5,1) -- (12.5,1) -- (12.5,5.5) node[left,pos=0.5] {$P_{u_{1}}$};
  \draw (12.65,2) -- (12.65,3) -- (15,3) node[below,pos=0.8] {$P_{u_{2}}$};
  \draw (14.5,3.15) -- (17.5,3.15) node[above,pos=0.18] {$P_{u_{3}}$};
  \draw[dotted] (17,3.3) -- (18.2,3.3);
  \draw (18.5,3.45)  -- node[above,pos=-0.1] {$P_{u_{\ell-2}}$}(20.5,3.45) -- (20.5,4);
  \draw (19.5,1) -- (20.65,1) -- (20.65,6)node[left,pos=0.1] {$P_{u_{\ell-1}}$};

\end{tikzpicture}
\end{small}
\end{center}
		(d)
	\end{minipage}
	\caption{The  construction of a $B_{1}^{m}$-EPG representation of a cactus 
	with free parts of the paths. These constructions are used if ${\cal R}_v$ 
	lies on a horizontal line of $P_v$. If ${\cal R}_v$ lies on a vertical line 
	of $P_v$, then the constructions have to be rotated by $90^o$ clockwise and 
	then  flipped  vertically.}
	\label{fig:cacti_in_b1m}
\end{figure}
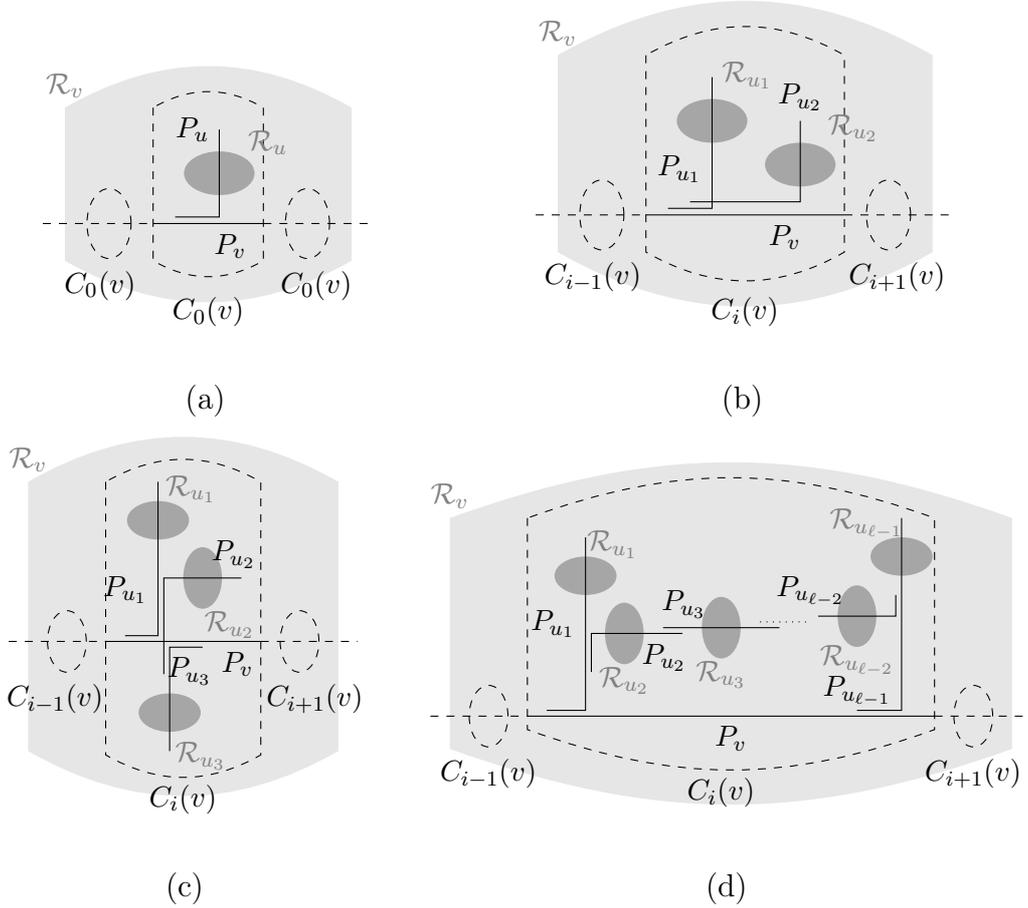
\begin{theoremMy} \label{B1m:cactus}
Every cactus is in $B_{1}^{m}$.
\end{theoremMy}
\begin{proof}
	We show that  \autoref{alg: b1m cactus}  correctly
	constructs a $B_1^m$-EPG representation for an arbitrary input cactus $G$.
	Note that 
	during the algorithm each vertex has a color $col(v)$ and that the colors
	of the vertices change during the algorithm. At the beginning all vertices
	are colored gray.  A vertex $v$ is gray 
	iff the corresponding grid path 
	$P_v$ has not been constructed yet. A vertex~$v$ turns  green as soon as 
	$P_v$  
	has been constructed. Finally,  $v$ is colored red as 
	soon as the paths       corresponding to all   neighbors of $v$ have been 
	constructed in the $B_1^m$-EPG
	representation.
	
	Essentially  the algorithm consists of a while loop which is repeated as
	long as there exist green vertices.  After selecting a green vertex $v$
	(line~\ref{select}),  the     algorithm \emph{explores} $v$ 
	in the  lines
	\ref{begin:explore}-\ref{end:explore}. Then,  after the
	exploration, $v$ turns red in line \ref{end:explore}.
	The
	algorithm terminates when all green vertices have turned red, hence after
	repeating the while loop $|V|$ times.   
	
	Notice that whenever a path $P_v$ is constructed, also a free region 
	$\mathcal{R}_v$ is indicated, in which a part of one of the two segments of 
	$P_v$ is contained and in which no other paths have been constructed yet.
	\smallskip
	
	Observe further that the following properties IA1, IA2 and IA3
	always hold when the exploration of some 
	vertex is completed, i.e.\ at line \ref{end:explore}.  We refer to these 
	three
	properties as the \emph{invariants of the algorithm}.
	\begin{itemize}
		\setlength{\itemsep}{0pt}
		\item[(IA1)] For any green or red  vertex $u$ there is an $\ell \in 
		\nz\cup\{0\}$
		and a sequence $S_u=\seq{v_0,v_1,\ldots, v_\ell}$ of red vertices
		such that $u$ has turned green during the exploration of $v_\ell$, and 
		if 
		$\ell\geqslant 1$, 
		$v_j$ has
		turned green during the exploration of $v_{j-1}$ for all $j\in \{1, 
		\ldots, 
		\ell\}$. 
		\item[(IA2)] Moreover, if $\ell\geqslant 1$, then any two vertices 
		$v_j$, $v_{j+1}$, 
		$0\leqslant
		j\leqslant \ell-1$,  in~$S_u$ are
		connected by a path of non-gray vertices all of which have been gray 
		when 
		the
		exploration of $v_{j}$ started.
		\item[(IA3)] In particular, for any green or
		red  vertex $u$ there is a path $Q_{v_0,u}$  consisting of non-gray 
		vertices 
		connecting $v_0$ and
		$u$ in $G$. 
	\end{itemize}
	\smallskip

	The proof of the theorem  is completed by proving  the next four 
	claims.  
	\begin{itemize}
		\setlength{\itemsep}{0pt}
		\item[(i)] 
		Consider a  green vertex $v$ currently selected  in
		line~\ref{select} of \autoref{alg: b1m cactus} and a cycle $C$ which 
		contains $v$. 
		Let $V(C)$ denote the set of vertices in~$C$. Then either all or none 
		of the vertices
		in $V(C)\setminus\{v\}$ are gray. 
		\item[(ii)] For every vertex $u$ of $G$ the algorithm constructs 
		exactly one path on the grid
		representing $u$. 
		\item[(iii)] All  constructed paths are monotonic and have at most one 
		bend.
		\item[(iv)] Two   paths $P_u$ and
		$P_w$ on the grid intersect iff the corresponding
		vertices $u$ and $w$ are adjacent in $G$.
	\end{itemize}
	
	\medskip
	Proof of  (i). This fact is trivially true if  $v=v_0$ (i.e.\ $v$ is the 
	first vertex
	explored at all). 
	Assume now that $v\neq v_0$ and  that there is a cycle $C$ in $G$  
	containing $v$ as well as  some
	gray vertex $u$ and some non-gray vertex $w\neq v$. Then $C$ does not 
	contain $v_0$
	(otherwise all vertices in  $C$ would have been colored green during the
	exploration of $v_0$).
	IA3  implies that there is  a path~$Q_{v_0,v}$  connecting
	$v_0$ and $v$ and consisting
	of non-gray vertices, and that  there is also a path $Q_{v_0,w}$ connecting 
	$v_0$ and $w$
	and consisting of non-gray vertices. But then $Q_{v_0,v}$, $Q_{v_0,w}$ and 
	a 
	path $Q_{v,w}$  connecting $v$ and $w$ along $C$ close a cycle $C'$. 
	Moreover  $C$ and
	$C'$  intersect along path  $Q_{v,w}$ which contains at least two vertices, 
	and
	this is a contradiction to $G$ being a cactus.

	\medskip
	Proof of (ii).  The construction of at least one path $P_u$ representing
	$u$, for every vertex $u$ of $G$, follows from the connectivity of
	$G$. Indeed consider  a path
	$Q_{v_0,u}=(v_0,v_1,\ldots, v_\ell=u)$  connecting $v_0$ and $u$ in~$G$. 
	The vertex~$v_1$ is colored green during the exploration of $v_0$. Since 
	the algorithm 
	does not terminate  before the exploration of all  green vertices, at some 
	point~$v_1$ will be explored and then $v_2$ will turn green at the latest.
	By repeating this  argument along the vertices of 
	$Q_{v_0,u}$ we conclude that 
	$u$ will turn  green at some point, meaning  that a grid path $P_u$ 
	representing
	$u$ is constructed.
	
	A path corresponding to a vertex is constructed only if a vertex turns
	green. Furthermore, each vertex turns green
	only once, because during the exploration of vertex $v$ the sets 
	$V(C_{i}(v))$,
	$0 \leqslant i \leqslant k_v$ have pairwise empty intersection since $G$ is 
	a
	cactus.
	Hence during the exploration of a vertex (only) gray vertices  turn green 
	at most
	once.
	Hence for each vertex exactly one  corresponding path is constructed.
	
	\medskip
	Proof of  (iii). This claim follows directly from the construction: in 
	\autoref{fig:cacti_in_b1m} all depicted paths are
	monotonic and have at most one bend. Moreover a $90^o$ clockwise rotation 
	of the paths depicted in 
	\autoref{fig:cacti_in_b1m} followed by a vertical flipping yields monotonic 
	paths with at most one bend.

	\medskip
	Proof of (iv). We first consider an edge $\{u,w\}$ in $G$ and show that 
	$P_u$ intersects $P_w$.
	Assume w.l.o.g.\ that $w$ was explored before $u$ by the algorithm. 
	
	If  $u$ was gray at the moment when $w$ is explored, then 
	$u$   belongs to
	some $C_i(w)$, $0\leqslant i\leqslant k_w$, and by construction  $P_u$   
	intersects
	$P_w$.
	
	Assume now that $u$ is not gray when $w$ is explored. 
	Let $v_u$ ($v_w$) be the vertex during the exploration of which $u$ ($w$)
	turned green.
	If $v_u=v_w$,  then   $u$,  $w$, $v_u$  are  contained in one of the cycles
	$C_i(v_u)$, $1 \leqslant i \leqslant k_{v_u}$ and $P_u$,  $P_v$ intersect
	by construction.

	Next we show that  $v_u\neq v_w$ cannot happen. Indeed,  if $v_u\neq
	v_w$, then consider the sequences $S_{v_u}$ and $S_{v_w}$
	described in IA1. Notice that both sequences 
	start
	with $v_0$ and denote by $x$ the last common vertex in  $S_{v_u}$ 
	and~$S_{v_w}$. Clearly $x\neq v_u$ or $x\neq v_w$ by assumption. IA2
	implies  the existence of two paths $Q_{u,x}$, $Q_{w,x}$  in  $G$ joining 
	$x$
	with $u$ and $w$, respectively, and (except for~$x$) consisting of gray 
	vertices at the moment of
	the exploration of $x$. But then $Q_{u,x}$, $Q_{w,x}$  and  $\{u,w\}$  
	would close a
	cycle with all 
	vertices  but $x$ being gray when $x$ is selected for exploration. This 
	would
	imply that both $u$ and $w$ turn green during the exploration of $x$,   
	hence
	$v_u=v_w=x$ would hold.
	
	To summarize, in all possible cases whenever $u$ and $w$ are adjacent in 
	$G$, the paths $P_u$ and $P_v$ intersect.
	\smallskip
	
	Finally, we show that two vertices $u$ and $w$ of $G$ are adjacent,
	whenever the corresponding grid  paths $P_u$, $P_w$ 
	intersect. Assume w.l.o.g.\ that 
	$P_w$ was
	constructed by the algorithm  before $P_u$. Observe that by construction
	all grid  paths  constructed after the completion of 
	the exploration of $w$  do not intersect $P_w$.  
	Thus $P_u$ has been constructed during the
	exploration of $w$. But then, by construction  $P_w$ and $P_u$
	intersect iff $w$ and $u$ are  neighbors.
\end{proof}
A standard time complexity  analysis of \autoref{alg: b1m
  cactus} reveals that  a $B_1^m$-EPG
    representation of a cactus $G$ of order $n$  can be constructed in $O(n^2)$ 
    time.
The existence of   a faster algorithm remains an open 
question. 

Finally, we summarize the results of this  section as follows.
\begin{corollary}
    The (monotonic) bend number of a cactus $G$ is given as follows.
    \begin{align*}
        b(G) = b^m(G) &=&\left \{ \begin{array}{ll}
            0 & \mbox{if $G$ contains no copy of 
                $M_{2}$, $M_{3}$, $C_{r}$, for  $r \geqslant 4$,}\\
            1 & \mbox{otherwise}\end{array} \right.
    \end{align*}
\end{corollary}


\section{Conclusions and open problems}
\label{sec:Conclusions}

In this paper we focused on EPG representations and on the
(monotonic) bend number   of outerplanar graphs. In particular, we dealt with 
two subclasses of outerplanar graphs: the maximal outerplanar graphs and the
cacti.  The main contribution of the paper is the full characterization of
graphs with (monotonic) bend number equal to $0$, $1$ or $2$ in the subclasses
mentioned above.  All presented proofs are constructive and lead to
efficient algorithms for the   construction of   the corresponding  EPG
representations. 

It was
already known from \cite{C,D} that $2$ is the best possible  upper bound on 
the bend number
of  outerplanar graphs. 
In this paper we  showed that $2$ is also an upper bound on  the monotonic bend
number of outerplanar graphs, i.e.\  $b^m(G)\leqslant 2$ holds for every outerplanar graph $G$. 
The result of~\cite{C} implies  the existence of  an
outerplanar graph which is not in $B_{1}^{m}$, so we conclude that $2$ is a
best possible 
 upper bound on the monotonic bend number of outerplanar graphs.
Thus,  the best possible  upper bounds on the monotonic
bend 
number  and on the bend number  coincide for  outerplanar graphs.

In~\cite{D} the inequality $b(G)\leqslant 2$  for the class of   graphs $G$
of  treewidth at most $2$ is proven, the outerplanar graphs being a proper 
subclass of
the later. So far  no  upper bound is known  on $b^m(G)$ for
graphs $G$  with treewidth bounded by $2$. In particular,  it is an open question
whether the coincidence mentioned above extends  to the class of graphs with treewidth
bounded by~$2$.

 Another  challenging open question concerns    the monotonic bend number of  
 planar graphs. As shown in \cite{D},  $b(G) \leqslant 4$ holds   for every
 planar graph.  Also the existence of a 
 planar graph with bend number at least $3$ is shown in \cite{D}. So the best
 possible upper bound on the bend number of planar graphs is $3$ or $4$. With
 respect to the upper bound on the  monotonic bend number of planar graphs we
 only now that it is at least $3$. 
\smallskip

Our  precise results on the (monotonic) bend number of maximal outerplanar
graphs
can be summarized as follows. 
 We distinguished two types  of maximal outerplanar graphs:
(I) maximal outerplanar graphs which do not contain $S_3$ as
an induced subgraph, and (II) maximal outerplanar graphs which do  contain $S_3$ as
an induced subgraph. 
We showed that $b(G)=b^m(G)=0$ holds for every  graph $G$ of type I.
Thus  for maximal outerplanar graphs the classes $B_0$ and
$B_1^m$ coincide, hence allowing two more shapes of paths in the EPG
representation does not increase the class  of maximal outerplanar graphs,
which can be represented.
  For graphs of type II we showed that $b^m(G)=2$ holds,   while
$b^m(G)=b(G)$ does not necessarily hold. More precisely, for 
 graphs $G$ of type II  the equality $b(G)=1$ holds  iff  $G$ is $M$-free
 (see \autoref{M-free-ouerplanar} and \autoref{out_tri_in_b1}).  Otherwise $b(G)=2$ holds. 
 \smallskip
 
 Our results on cacti can be summarized as follows. 
  For cacti the monotonic bend number and the bend number  coincide and are
  bounded by~$1$, i.e.\ and $b(G)=
b^m_1(G)\leqslant 1$ holds for every cactus $G$. 
Furthermore  $b(G)= b^m(G)= 0$ holds iff the  cactus $G$ is $MC$-free (see \autoref{MC-free-ouerplanar}). Otherwise 
$b(G)=b^m(G)=1$ holds.

Observe that the two investigated  subclasses of
 outerplanar graphs, the maximal outerplanar graphs and the cacti,   behave
 quite differently in terms of the (monotonic) bend number. This is not
 surprising  given  the major structural differences of  these two  graph
 classes.
\smallskip

The full characterization of outerplanar graphs with a   (monotonic) bend
number equal to $1$ or $2$ remains a challenging open question. In particular, 
 the characterization of Halin graphs and series-parallel graphs  with a   
 (monotonic) bend
number equal to $1$ or $2$ are interesting open questions. Note that both 
graphs 
classes belong to the planar
graphs and $2$ is an upper bound on the bend number for both of them.
In~\cite{KHalinGraphs} it was shown that $b^m(G)\leqslant 2$ for every  Halin graph
$G$. In the case of series-parallel graphs the upper bound of $2$ follows form the results in~\cite{D} and the fact that the treewidth of
series-parallel graphs is bounded by $2$, 
see~\cite{OuterplanarGraphsTreewidth2}. 



\ifArXiV 
\bibliographystyle{plain}
\fi 

\bibliography{papers}

\begin{thebibliography}{10}

\bibitem{VPG}
{Andrei} {Asinowski}, {Elad} {Cohen}, {Martin Charles} {Golumbic}, {Vincent}
  {Limouzy}, {Marina} {Lipshteyn}, and {Michal} {Stern}.
\newblock Vertex intersection graphs of paths on a grid.
\newblock {\em Journal of Graph Algorithms and Applications}, 16(2):129--150,
  2012.

\bibitem{H}
Andrei Asinowski and Bernard Ries.
\newblock Some properties of edge intersection graphs of single-bend paths on a
  grid.
\newblock {\em Discrete Mathematics}, 312(2):427--440, 2012.

\bibitem{B}
Andrei Asinowski and Andrew Suk.
\newblock Edge intersection graphs of systems of paths on a grid with a bounded
  number of bends.
\newblock {\em Discrete Applied Mathematics}, 157(14):3174--3180, 2009.

\bibitem{C}
Therese Biedl and Michal Stern.
\newblock Edge-intersection graphs of $k$-bend paths in grids.
\newblock In {\em Computing and Combinatorics}, pages 86--95. Springer, 2009.

\bibitem{OuterplanarGraphsTreewidth2}
Hans~L. Bodlaender.
\newblock A partial $k$-arboretum of graphs with bounded treewidth.
\newblock {\em Theoretical Computer Science}, 209(1–2):1 -- 45, 1998.

\bibitem{OCliqueColoringB1}
Flavia Bonomo, Mar\'\i a P\'\i~a Mazzoleni, and Maya Stein.
\newblock Clique coloring {$B_1$}-{EPG} graphs.
\newblock {\em Discrete Mathematics}, 340(5):1008--1011, 2017.

\bibitem{BooLue76}
Kellogg~S. Booth and George~S. Lueker.
\newblock Testing for the consecutive ones property, interval graphs, and graph
  planarity using pq-tree algorithms.
\newblock {\em Journal of Computer and System Sciences}, 13(3):335--379, 1976.

\bibitem{NMaxIndependentSetInB1EPG}
M.~Bougeret, S.~Bessy, D.~Gon\c{c}alves, and C.~Paul.
\newblock On independent set on {B}1-{EPG} graphs.
\newblock In {\em Approximation and online algorithms}, volume 9499 of {\em
  Lecture Notes in Comput. Sci.}, pages 158--169. Springer, Cham, 2015.

\bibitem{RMaxCliqueInB2}
Nicolas Bousquet and Marc Heinrich.
\newblock Computing maximum cliques in {$B_2$}-{EPG} graphs.
\newblock In {\em Graph-theoretic concepts in computer science}, volume 10520
  of {\em Lecture Notes in Comput. Sci.}, pages 140--152. Springer, Cham, 2017.

\bibitem{knockknee}
Martin~Lee Brady and Majid Sarrafzadeh.
\newblock Stretching a knock-knee layout for multilayer wiring.
\newblock {\em IEEE Transactions on Computers}, 39(1):148--151, Jan 1990.

\bibitem{E}
Kathie Cameron, Steven Chaplick, and Ch{\'\i}nh~T Ho{\`a}ng.
\newblock Edge intersection graphs of {L}-shaped paths in grids.
\newblock {\em Electronic Notes in Discrete Mathematics}, 44:363--369, 2013.

\bibitem{SSplitGraphsB1}
Zakir Deniz, Simon Nivelle, Bernard Ries, and David Schindl.
\newblock On split {$B_1$}-{EPG} graphs.
\newblock In {\em L{ATIN} 2018: {T}heoretical Informatics}, volume 10807 of
  {\em Lecture Notes in Comput. Sci.}, pages 361--375. Springer, Cham, 2018.

\bibitem{J}
Dror Epstein, Martin~Charles Golumbic, and Gila Morgenstern.
\newblock Approximation algorithms for ${B}_{1}$-{EPG} graphs.
\newblock In {\em Algorithms and Data Structures}, pages 328--340. Springer,
  2013.

\bibitem{KHalinGraphs}
Mathew~C. Francis and Abhiruk Lahiri.
\newblock V{PG} and {EPG} bend-numbers of {H}alin graphs.
\newblock {\em Discrete Appl. Math.}, 215:95--105, 2016.

\bibitem{VPT1}
Fănică Gavril.
\newblock A recognition algorithm for the intersection graphs of paths in
  trees.
\newblock {\em Discrete Mathematics}, 23(3):211 -- 227, 1978.

\bibitem{firstEPT1}
Martin~Charles Golumbic and Robert~E. Jamison.
\newblock Edge and vertex intersection of paths in a tree.
\newblock {\em Discrete Mathematics}, 55(2):151 -- 159, 1985.

\bibitem{firstEPT2}
Martin~Charles Golumbic and Robert~E. Jamison.
\newblock The edge intersection graphs of paths in a tree.
\newblock {\em Journal of Combinatorial Theory, Series B}, 38(1):8 -- 22, 1985.

\bibitem{startpaper}
Martin~Charles Golumbic, Marina Lipshteyn, and Michal Stern.
\newblock Edge intersection graphs of single bend paths on a grid.
\newblock {\em Networks}, 54(3):130--138, 2009.

\bibitem{G}
Martin~Charles Golumbic, Marina Lipshteyn, and Michal Stern.
\newblock Single bend paths on a grid have strong helly number 4: errata atque
  emendationes ad “{E}dge intersection graphs of single bend paths on a
  grid”.
\newblock {\em Networks}, 62(2):161--163, 2013.

\bibitem{D}
Daniel Heldt, Kolja Knauer, and Torsten Ueckerdt.
\newblock On the bend-number of planar and outerplanar graphs.
\newblock In {\em LATIN 2012: Theoretical Informatics}, pages 458--469.
  Springer, 2012.

\bibitem{F}
Daniel Heldt, Kolja Knauer, and Torsten Ueckerdt.
\newblock Edge-intersection graphs of grid paths: The bend-number.
\newblock {\em Discrete Applied Mathematics}, 167(0):144 -- 162, 2014.

\bibitem{lekkeikerker1962representation}
Cornelis Lekkeikerker and Johan Boland.
\newblock Representation of a finite graph by a set of intervals on the real
  line.
\newblock {\em Fundamenta Mathematicae}, 51:45--64, 1962.

\bibitem{AsurveyOnWiring}
Paul Molitor.
\newblock A survey on wiring.
\newblock {\em Journal of Information Processing and Cybernetics}, 27(1):3--19,
  April 1991.

\bibitem{LRecognizionB2EPGisNPComplete}
Martin Pergel and Pawe{\l} Rz{\k{a}}{\.{z}}ewski.
\newblock On edge intersection graphs of paths with 2 bends.
\newblock In {\em Graph-theoretic concepts in computer science}, volume 9941 of
  {\em Lecture Notes in Comput. Sci.}, pages 207--219. Springer, Berlin, 2016.

\end{thebibliography}

\end{document}